\documentclass{amsart}

\usepackage{epsfig}
\usepackage{amsmath}
\usepackage{amssymb,mathrsfs}

\usepackage[all]{xy}
\usepackage{tikz}
\usepackage{graphicx}

\usepackage{color}
\usepackage{tabularx}

\usepackage{hyperref}

\def\comment#1{{\sf{[#1]}}}

\def\Z{{\mathbb Z}}
\def\Q{{\mathbb Q}}

\def\R{{\mathbb R}}
\def\C{{\mathbb C}}
\def\P{{\mathbb P}}
\def\F{{\mathbb F}}

\def\L{{\mathbb L}}
\def\V{{\mathbb V}}
\def\Schur{{\mathbb S}}

\def\kk{{\Bbbk}}		

\def\cA{{\mathcal A}}

\def\cG{{\mathcal G}}

\def\M{{\mathcal M}}

\def\cO{{\mathcal O}}
\def\cP{{\mathcal P}}

\def\U{{\mathcal U}}

\def\X{{\mathcal X}}

\def\sD{{\mathscr D}}

\def\sP{{\mathscr P}}

\def\g{{\mathfrak g}}
\def\h{{\mathfrak h}}
\def\k{{\mathfrak k}}
\def\m{{\mathfrak m}}
\def\n{{\mathfrak n}}
\def\p{{\mathfrak p}}
\def\sp{{\mathfrak{sp}}}
\def\fr{{\mathfrak r}}
\def\t{{\mathfrak t}}
\def\u{{\mathfrak u}}

\def\mtm{{\mathfrak{mtm}}}

\def\e{{\epsilon}}
\def\w{{\omega}}
\def\G{{\Gamma}}

\def\bV{\boldsymbol{V}}

\def\bmu{{\boldsymbol{\mu}}}

\def\kappahat{\hat{\kappa}}

\def\phihat{\hat{\varphi}}
\def\tauhat{\hat{\tau}}

\def\chitilde{{\tilde{\chi}}}

\def\rhotilde{{\tilde{\rho}}}
\def\phitilde{{\tilde{\varphi}}}
\def\tautilde{{\tilde{\tau}}}

\def\cGhat{{\widehat{\cG}}}

\def\Phat{{\widehat{\P}}}
\def\Shat{{\widehat{S}}}

\def\ghat{{\widehat{\g}}}

\def\ubar{{\overline{\u}}}

\def\taubar{\overline{\tau}}

\def\gbar{\overline{\g}}

\def\Fbar{\overline{F}}

\def\Qbar{{\overline{\Q}}}
\def\Sbar{{\overline{S}}}

\def\Xbar{{\overline{X}}}

\def\deltabar{{\overline{\delta}}}

\def\vv{{\vec{\mathsf v}}}
\def\ww{{\vec{\mathsf w}}}
\def\vu{{\vec{\mathsf u}}}

\def\uu{{\vec{1}}}
\def\vr{{\vec{r}}}

\def\vV{{\vec{V}}}

\def\MHS{{\mathsf{MHS}}}
\def\MTM{{\mathsf{MTM}}}
\def\MEM{{\mathsf{MEM}}}

\def\Vec{{\mathsf{Vec}}}
\def\Rep{{\mathsf{Rep}}}

\def\cR{{\mathsf{R}}}

\def\et{{\mathrm{\acute{e}t}}}
\def\cts{{\mathrm{cts}}}
\def\un{{\mathrm{un}}}
\def\ss{{\mathrm{ss}}}

\def\rot{{\mathrm{rot}}}

\def\geom{{\mathrm{geom}}}
\def\cusp{{\mathrm{cusp}}}
\def\ES{{\mathrm{ES}}}

\def\To{\longrightarrow}
\def\bdot{{\bullet}}

\def\blank{{\underline{\phantom{x}}}}		

\def\aa{{\mathbf{a}}}
\def\bb{{\mathbf{b}}}
\def\dd{{\mathbf{d}}}
\def\ee{{\mathbf{e}}}
\def\br{{\mathbf{r}}}
\def\bt{{\mathbf{t}}}
\def\bu{{\mathbf{u}}}
\def\bv{{\mathbf{v}}}
\def\bw{{\mathbf{w}}}
\def\zz{{\mathbf{z}}}

\def\bil{{\langle\phantom{x},\phantom{x}\rangle}}
\def\gold{\{\phantom{x},\phantom{x}\}}

\def\Pminus{{\P^1-\{0,1,\infty\}}}
\def\Ql{{\Q_\ell}}

\def\uu{{\vec{1}}}

\def\MT{{\mathrm{MT}}}							
\def\Gm{{\mathbb{G}_m}}
\def\Ga{{\mathbb{G}_a}}
\def\Sp{{\mathrm{Sp}}}

\def\SL{{\mathrm{SL}}}

\newcommand\im{\operatorname{im}} 
\newcommand\id{\operatorname{id}}
\newcommand\ad{\operatorname{ad}}

\newcommand\Hom{\operatorname{Hom}}
\newcommand\Ext{\operatorname{Ext}}

\newcommand\End{\operatorname{End}}
\newcommand\Aut{\operatorname{Aut}}
\newcommand\Diff{\operatorname{Diff}}

\newcommand\Der{\operatorname{Der}}
\newcommand\Gr{\operatorname{Gr}}

\newcommand\Bl{\operatorname{Bl}}

\newcommand\Sym{\operatorname{Sym}}
\newcommand\Gal{\operatorname{Gal}}

\newcommand\Jac{\operatorname{Jac}}

\newcommand\Arf{\operatorname{Arf}}

\newcommand\Tr{\operatorname{Tr}}

\newcommand\CH{\operatorname{CH}}
\newcommand\SAut{\operatorname{SAut}}
\newcommand\SDer{\operatorname{SDer}}

\newcommand\edg{\operatorname{edg}}

\newcommand\comptensor{\operatorname{\widehat{\otimes}}}

\renewcommand\div{\operatorname{div}}

\numberwithin{equation}{section}

\newtheorem{theorem}{Theorem}[section]
\newtheorem{lemma}[theorem]{Lemma}
\newtheorem{proposition}[theorem]{Proposition}
\newtheorem{corollary}[theorem]{Corollary}

\theoremstyle{definition}

\newtheorem{example}[theorem]{Example}

\theoremstyle{remark}
\newtheorem{remark}[theorem]{Remark}

\newtheorem{question}[theorem]{Question}
\newtheorem{problem}[theorem]{Problem}

\newtheorem{conjecture}[theorem]{Conjecture}

\begin{document}

\title{Johnson Homomorphisms}

\dedicatory{Dedicated to the memory of Dennis Johnson}

\author{Richard Hain}
\address{Department of Mathematics\\ Duke University\\
Durham, NC 27708-0320}
\email{hain@math.duke.edu}

\thanks{Supported in part by grant DMS-1406420 from the National Science Foundation.}
\thanks{ORCID iD: {\sf orcid.org/0000-0002-7009-6971}}

\date{\today}

\subjclass{Primary 57K20; Secondary 14C30, 17B62}

\keywords{Johnson homomorphism, Torelli group, mapping class group, Goldman bracket, Turaev cobracket, mixed Hodge structure}

\begin{abstract}
Torelli groups are subgroups of mapping class groups that consist of those diffeomorphism classes that act trivially on the homology of the associated closed surface. The Johnson homomorphism, defined by Dennis Johnson, and its generalization, defined by S.~Morita, are tools for understanding Torelli groups. This paper surveys work on generalized Johnson homomorphisms and tools for studying them. The goal is to unite several related threads in the literature and to clarify existing results and relationships among them using Hodge theory. We survey the work of Alekseev, Kawazumi, Kuno and Naef on the Goldman--Turaev Lie bialgebra, and the work of various authors on cohomological methods for determining the stable image of generalized Johnson homomorphisms. Various open problems and conjectures are included.

Even though the Johnson homomorphisms were originally defined and studied by topologists, they are important in understanding arithmetic properties of mapping class groups and moduli spaces of curves. We define arithmetic Johnson homomorphisms, which extend the generalized Johnson homomorphisms, and explain how the Turaev cobracket constrains their images.
\end{abstract}

\maketitle

\tableofcontents

\section{Introduction}

The Johnson homomorphism was defined by Dennis Johnson in \cite{johnson:homom}, extending earlier work of Andreadakis \cite{andreadakis}, Bachmuth \cite{bachmuth}, Papakyriakopoulos \cite{papakyr} and Sullivan \cite{sullivan}. It is a surjective group homomorphism
\begin{equation}
\label{eqn:classic-johnson}
\tau : T_{S,\partial S} \to \Lambda^3 H_1(S;\Z),
\end{equation}
where $S$ is a compact surface with one boundary component and $T_{S,\partial S}$ is the Torelli subgroup of the mapping class group
$$
\G_{S,\partial S} := \pi_0 \Diff(S,\partial S)
$$
that consists of isotopy classes of diffeomorphisms of $S$ that fix $\partial S$ pointwise and act trivially on the homology of $S$. Equivalently, it is the kernel of the natural homomorphism
\begin{equation}
\label{eqn:def_rho}
\rho : \G_{S,\partial S} \to \Aut H_1(S).
\end{equation}
When $g\ge 3$, Johnson showed (remarkably) in \cite{johnson:fg} that $T_{S,\partial S}$ is finitely generated and in \cite{johnson:h1} that the kernel of the induced surjection
\begin{equation}
\label{eqn:classic_johnson}
\taubar : H_1(T_{S,\partial S}) \to \Lambda^3 H_1(S;\Z)
\end{equation}
is a finite group of exponent 2. In genus 2, the Torelli group is not finitely generated.

Torelli groups and Johnson homomorphisms are easily defined for surfaces with punctures and any number of boundary components. (See equation (\ref{def:torelli}) for the definition in general.) But, for clarity, we will generally limit the discussion of them in the introduction to surfaces with one boundary component.

\subsection{Why study Torelli groups?}

Torelli groups are of fundamental importance in algebraic geometry, arithmetic geometry and geometric topology. These geometric, arithmetic and topological facets of Torelli groups do not exist in separate universes, but interact in a rich soup of geometric topology, algebraic curve theory, Hodge theory and Galois theory. One goal of this survey is to present a unified exposition of Johnson homomorphisms that touches on all of these facets.

The importance of Torelli groups in algebraic geometry comes from the fact that the classifying space $BT_g$ of the Torelli group $T_g$ of a {\em closed} surface of genus $g$ is the homotopy fiber of the (orbifold) period map from the moduli space $\M_g$ of smooth projective curves of genus $g$ to the moduli space $\cA_g$ of principally polarized abelian varieties of dimension $g$.\label{def:Ag} Because the enumerative geometry of these moduli spaces is richly represented in their cohomology rings, the difference between the enumerative geometry of curves and of abelian varieties is strongly influenced by the Torelli group $T_g$. One example of this linkage is the result of Kawazumi and Morita \cite{kawazumi-morita}. They construct an algebra homomorphism
\begin{equation}
\label{eqn:kawa_morita}
\big[\Lambda^\bdot H^1(T_g;\Q)\big]^{\Sp_g(\Z)} \to H^\bdot(\M_g;\Q)
\end{equation}
and show that its image is the the tautological subalgebra of $H^\bdot(\M_g;\Q)$. This result was later used by Morita \cite{morita:taut} to prove the cohomological version of one of Faber's conjectures.

The arithmetic significance of the Torelli group $T_g$ comes from the fact that the moduli space $\M_g$ is a smooth stack over $\Z$ whose reduction mod $p$ is smooth for all primes $p$. This means that the motivic structures (mixed Hodge structures, Galois representations, \dots) that occur in the cohomology of $\M_g$ (with possibly twisted coefficients) should be realizations of motives unramified over $\Z$. Those motivic structures in the cohomology of $\M_g$ that do not come from $\cA_g$, must come from $T_g$. Interesting examples of motives in the cohomology of $\M_3$ that cannot come from $\cA_3$ are given in the paper \cite[\S13]{genus3} by Cl\'ery, Faber and van der Geer. Possible implications for the cohomology of $T_3$ are discussed briefly in the Afterword, Section~\ref{sec:beyond}.

Torelli groups are not only interesting objects of study for geometric group theorists, but can also be used to define finite type invariants of homology 3-spheres, as was discovered by Garoufalidis and Levine \cite{garoufalidis-levine:spheres}. This is because every homology 3-sphere can be obtained from the 3-sphere by decomposing it into two handlebodies and reglueing them back together via an element of the Torelli group of their common boundary. Invariants of homology 3-spheres thus give functions on Torelli groups; those that are of finite type vanish on some term of the lower central series.

The goals of this paper are: (1) survey what is known about Torelli groups via Johnson homomorphisms, especially recent results related to the Goldman--Turaev Lie bialgebra and by cohomological methods; (2) unite various threads in the literature; and (3) clarify the relationships between these threads. One main theme is the role of motivic tools, such as Hodge theory, which enable one to prove results about Torelli groups that may be beyond the reach of topological methods.


\subsection{The Johnson filtration and Johnson homomorphisms}

In order to probe deeper into the Torelli groups, Johnson \cite{johnson:survey} proposed investigating generalizations\footnote{These were previously defined for automorphism groups of free groups by Andreadakis \cite{andreadakis}.} of the homomorphism (\ref{eqn:classic-johnson}).  One starts with the descending central filtration
$$
\label{def:johnson_filt}
T_{S,\partial S} = J^1 T_{S,\partial S} \supset J^2 T_{S,\partial S} \supset J^3 T_{S,\partial S} \supset \cdots
$$
of $T_{S,\partial S}$, called the {\em Johnson filtration}, where
\begin{equation}
\label{eqn:johnson_filt}
J^k T_{S,\partial S} := \ker\big\{T_{S,\partial S} \to \Aut\big(\pi_1(S,x)/L^{k+1}\pi_1(S,x)\big)\big\},
\end{equation}
$x\in \partial S$, and $L^kG $ denotes the $k$th term of the lower central series (LCS) of a group $G$. The {\em higher Johnson homomorphisms} are injective maps
\begin{equation}
\label{eqn:higher_johnson}
\tau_k : \Gr^k_J T_{S,\partial S} \hookrightarrow \Hom_\Z\big(H_1(S),\Gr^{k+1}_L\pi_1(S,x)\big),
\end{equation}
where $\Gr^m_L\pi_1(S,x)$ denotes the $m$th graded quotient $L^m\pi_1(S,x)/L^{m+1}$ of the LCS of $\pi_1(S,x)$. The first, $\tau_1$, is essentially the classical Johnson homomorphism (\ref{eqn:classic_johnson}). Each $\tau_k$ is invariant under the natural action of the symplectic group $\Sp_g(\Z)$ on its source and target. The associated graded
$$
\Gr_J^\bdot T_{S,\partial S} := \bigoplus_{k\ge 0} \Gr_J^k T_{S,\partial S} := \bigoplus_{k\ge 0} J^k T_{S,\partial S}/J^{k+1}T_{S,\partial S}
$$
is a graded Lie algebra over $\Z$ and the higher Johnson homomorphisms constitute a homomorphism into the Lie algebra $\Der\Gr^\bdot_L \pi_1(S,x)$ of derivations of the graded Lie algebra $\Gr^\bdot_L\pi_1(S,x)$. The higher Johnson homomorphisms (after tensoring with $\Q$) are not surjective, as was observed by Morita \cite{morita:homom} who gave the first non-trivial constraint on the image using his trace map.

The kernel of the Johnson homomorphism (\ref{eqn:classic-johnson}) is $J^2T_{S,\partial S}$ and is called the {\em Johnson subgroup} of the Torelli group. Johnson \cite{johnson:twists} showed that it is generated by Dehn twists on separating simple closed curves for all $g\ge 2$. The conventional wisdom, inspired by Mess's proof \cite{mess} that the Torelli group in genus 2 is a countably generated free group, was that $J^2T_{S,\partial S}$ should have infinite first Betti number. Dimca and Papadima \cite{dimca-papadima}, using Hodge theory, defied this conventional wisdom by proving that $H_1(J^2T_{S,\partial S};\Q)$ is finite dimensional for all $g\ge 4$. More recently, Ershov and He \cite{ershov-he} have shown that $J^2T_{S,\partial S}$ is finitely generated when $g\ge 12$ and, more generally, that any subgroup of $T_{S,\partial S}$ that contains $L^kT_{S,\partial S}$ has finitely generated abelianization for all $g\ge 8k-4$. (This result was subsequently improved by Church, Ershov and Putman in \cite{cep} where they show that $L^kT_{S,\partial S}$ is finitely generated when $g \ge 4$ and $g\ge 2k-1$.) Whether or not the first Betti number of the genus 3 Johnson subgroup is finite is still unresolved.

These results illustrate one of the main themes of this paper, discussed in detail in Section~\ref{sec:stability}. Namely that, as the genus increases, each quotient $T_{S,\partial S}/J^{k+1}$ becomes more regular in the sense that each $\Gr^k_J T_{S,\partial S}\otimes \Q$ and $\Gr^k_L T_{S,\partial S}\otimes \Q$ stabilizes in the representation ring of the symplectic group as $g$ increases.

\subsection{From graded to filtered}

The generalized Johnson homomorphism is traditionally regarded as a homomorphism of graded Lie algebras. However, it is more natural to work with filtered Lie algebras and to pass to the associated graded Lie algebras only when necessary. These filtered Lie algebras are obtained by replacing $\pi_1(S,x)$ by its unipotent completion (also called Malcev completion) and the mapping class group $\G_{S,\partial S}$ by its {\em relative} unipotent completion. (Both are reviewed in Section~\ref{sec:completions}.) These are proalgebraic groups over $\Q$. This is the approach taken in \cite{hain:torelli}.

The Lie algebra $\p(S,x)$ of the unipotent completion of $\pi_1(S,x)$ is a free pronilpotent Lie algebra. There are canonical isomorphisms
$$
\big[\Gr_L^\bdot \pi_1(S,x)\big]\otimes \Q \cong \Gr_L^\bdot \p(S,x) \cong \L(H)
$$
of the associated graded of the LCS of $\pi_1(S,x)$ (tensored with $\Q$) and of $\p(S,x)$ with the free Lie algebra on $H:=H_1(S;\Q)$. The Lie algebra of the relative completion of $\G_{S,\partial S}$ with respect to the homomorphism (\ref{eqn:def_rho}) is an extension
$$
0 \to \u_{S,\partial S} \to \g_{S,\partial S} \to \sp(H) \to 0
$$
where $\sp(H)$ is the Lie algebra of the symplectic group of $H_1(S)$ and $\u_{S,\partial S}$ is pronilpotent. The {\em geometric Johnson homomorphism} is the homomorphism
\begin{equation}
\label{eqn:johnson_univ}
\tau_{S,\partial S} : \g_{S,\partial S} \to \Der \p(S,x)
\end{equation}
induced by the action of $\G_{S,\partial S}$ on $\p(S,x)$. The higher Johnson homomorphisms $\tau_k \otimes \Q$ are packaged in $\tau_{S,\partial S}$. They can be recovered from the geometric Johnson homomorphism by taking graded quotients:
$$
\xymatrix@C=16pt{
\big(\Gr^k_J T_{S,\partial S}\big)\otimes\Q \cong \Gr^k_J \u_{S,\partial S}  \ar@{^{(}->}[r] & \Der_k \Gr^\bdot_L \p(S,x) =  \Hom_\Q\big(H,\Gr^{k+1}_L\L(H)\big), 
}
$$
where the Johnson filtration of $\u_{S,\partial S}$ is defined similarly to the Johnson filtration (\ref{eqn:johnson_filt}) of $T_{S,\partial S}$, and where $\Der_k$ denotes the derivations of ``degree $k$''. Unlike the case of higher Johnson homomorphisms, which are injective by definition, the geometric Johnson homomorphism is not known to be injective. One of the fundamental problems in the subject, discussed in Section~\ref{sec:injectivity}, is to understand its kernel.

\subsection{The role of Hodge theory}

Replacing a filtered object, such as $T_{S,\partial S}$, by its associated graded $\Gr^\bdot_J T_{S,\partial S}$ typically results in a significant loss of information as the associated graded functor is rarely exact. However, many topological invariants of complex algebraic varieties possess a natural {\em weight filtration} that has good exactness properties. The {\em yoga of weights}, which exploits the exactness properties of weight filtrations, is a powerful tool for studying the topology of algebraic varieties in general and Johnson homomorphisms in particular. Some early and illuminating examples of it are given in Deligne's 1974 Vancouver ICM talk \cite{deligne:poids}.

Weight filtrations were introduced by Deligne \cite{deligne:h2,deligne:h3} as part of the data of a mixed Hodge structure (MHS). There he showed that the rational cohomology of every complex algebraic variety has a natural MHS, and thus a weight filtration. His results were extended to many other invariants of varieties that are accessible to differential forms in \cite{morgan}, \cite{hain:dht}, and \cite{hain:malcev}. A brief overview of Hodge theory is given in Section~\ref{sec:hodge}.

The weight filtration of a MHS $V$ is an increasing filtration
$$
 0 = W_a V \subseteq \cdots \subseteq W_{k-1}V \subseteq W_kV \subseteq \dots \subseteq W_b V = V
$$
of its underlying $\Q$ vector space. In many cases, such as those in this survey, the weight filtration is determined by the underlying topology. The important fact for us is that there is a natural (though not canonical) isomorphism\footnote{The direct sum is replaced by a direct product if $V$ is a pro-MHS, such as $\g_{S,\partial S}$ or $\p(S,x)$.} 
$$
V \cong \bigoplus_{k\in \Z} \Gr^W_k V := \bigoplus_{k\in \Z} W_kV/W_{k-1}V.
$$
These splittings are natural in the sense that they are preserved by morphisms of MHS. 

In the current context, invariants such as $\p(S,x)$ and $\g_{S,\partial S}$ carry a natural pro-MHS for each choice of an algebraic structure on $S$. These MHS, but not their weight filtrations, depend non-trivially on this algebraic structure.  When $S$ has at most one boundary component, the weight filtration on $\p(S,x)$ is related to its lower central series filtration $L^\bdot$ by
$$
W_{-k} \p(S,x) = L^k \p(S,x).
$$
Similarly, when $g\ge 3$, the weight filtration of $\g_{S,\partial S}$ is essentially the lower central series of $\u_{S,\partial S}$:
$$
W_0 \g_{S,\partial S} = \g_{s,\partial S}, \quad W_{-k}\u_{S,\partial S} = W_{-k} \g_{S,\partial S} = L^k \u_{S,\partial S} \text{ when } k\ge 1.
$$

For each choice of a complex structure on $S$, the Johnson homomorphism is a morphism of MHS. This is useful as it enables the use of the yoga of weights. In particular, it implies that there are natural isomorphisms of each with its associated weight graded Lie algebra so that the diagram
$$
\xymatrix@C=40pt{
\g_{S,\partial S} \ar[r]^{\tau_{S,\partial S}} \ar[d]^\simeq & \Der\p \ar[d]^\simeq \cr
\prod_{k\ge 0} \Gr^W_{-k}\g_{S,\partial S} \ar[r]^{\Gr^W_\bdot\tau_{S,\partial S}} & [\Der \Gr^W_\bdot \p(S,x)]^\wedge
}
$$
commutes.

\subsection{The arithmetic Johnson homomorphism}

The category of MHS is equivalent to the category of representations of an affine group over $\Q$. Every MHS is a representation of this group and morphisms of MHS are equivariant under this group action. This action is the Hodge theoretic analogue of the action of the Galois group of a number field $K$ on the profinite topological invariants of varieties defined over $K$. In this way, Hodge theory provides hidden symmetries which play a significant role in the story. In Section~\ref{sec:arith_johnson} these extra symmetries are used to construct an enlargement $\ghat_{S,\partial S}$ of the image $\gbar_{S,\partial S}$ of $\g_{S,\partial S}$ in $\Der\p(S,x)$. We show that the geometric Johnson homomorphism (\ref{eqn:johnson_univ}) induces a homomorphism
\begin{equation}
\label{eqn:johnson_arith}
\tauhat_{S,\partial S} : \ghat_{S,\partial S} \hookrightarrow \Der \p(S,x)
\end{equation}
that we call the {\em arithmetic Johnson homomorphism}. The Lie algebra $\ghat_{S,\partial S}$ also carries a natural MHS, and thus a weight filtration. In Section~\ref{sec:arith_johnson} we show that the quotient Lie algebra
$$
\ghat_{S,\partial S}/\gbar_{S,\partial S}
$$
is isomorphic to the motivic Lie algebra
$$
\k = \L(\sigma_3,\sigma_5,\sigma_7,\sigma_9, \dots)^\wedge
$$
associated to mixed Tate motives unramified over $\Z$, where the generator $\sigma_{2m+1}$ has weight $-4m-2$. The proof of this result uses the proof of the Oda Conjecture by Takao \cite{takao} (building on earlier work of Ihara, Matsumoto and Nakamura) and Brown's fundamental result \cite{brown}.

\subsection{Bounding the Johnson image}

An important open problem is to determine the ``Johnson image''. That is, determine the images of the higher Johnson homomorphisms (\ref{eqn:higher_johnson}). This, by Hodge theory, is equivalent to the problem of determining the image of the geometric Johnson homomorphism (\ref{eqn:johnson_univ}). The main results of \cite{hain:torelli} imply that, when $g\ge 3$, the Johnson image is generated by its ``degree 1 part'' $\Gr^W_{-1}\gbar_{S,\partial S}$, which is the image of the classical Johnson homomorphism $\tau$. It is equally important to determine the image $\ghat_{S,\partial S}$ of the arithmetic Johnson homomorphism (\ref{eqn:johnson_arith}) and find explicit generators of its associated weight graded Lie algebra.\footnote{This is the problem of determining the image of the $\sigma_{2n+1}$'s in $\Der\Gr^W_\bdot\p(S,x)$ mod the {\em geometric derivations} $\Gr^W_\bdot\gbar_{S,\partial S}$ and mod commutators of the $\sigma_{2m+1}$'s.}

The image of the universal and arithmetic Johnson homomorphisms lie in the Lie subalgebra \label{def:dertheta}
$$
\Der^\theta \p(S,x) := \{D\in \Der\p(S,x) : D(\log \sigma_o) = 0\}
$$
of $\Der \p(S,x)$, where $\sigma_o$ is the boundary loop and $\log \sigma_o\in \p(S,x)$ is the logarithm of its image in the unipotent completion of $\pi_1(S,x)$. The standard approach to bounding the Johnson image $\gbar_{S,\partial S}$ is to find linear functions on $\Der^\theta\p(S,x)$ (or its associated weight graded) that vanish on it. The first such maps were Morita's trace maps, which were constructed in \cite{morita:homom} and have been generalized by Enomoto--Satoh \cite{enomoto-satoh} and Conant \cite{conant}. The most conceptually appealing methods for producing such functionals was proposed by Kawazumi and Kuno \cite{kk:intersections} and uses the Goldman--Turaev Lie bialgebra, an important new geometric tool for studying Johnson homomorphisms.

Briefly, the Goldman--Turaev Lie bialgebra associated to a surface $S$ with a framing $\xi$ is the free $\Z$-module $\Z\lambda(S)$ generated by the set of conjugacy classes $\lambda(S)$ of the fundamental group of $S$. It has a Lie bracket
$$
\gold : \Z\lambda(S) \otimes \Z\lambda(S) \to \Z\lambda(S)
$$
defined by Goldman \cite{goldman}, which does not require a framing, and a cobracket
$$
\delta_\xi : \Z\lambda(S) \to \Z\lambda(S) \otimes \Z\lambda(S)
$$
which does. A preliminary version of the cobracket was defined by Turaev \cite{turaev:78}. We use the refined version for framed surfaces introduced by Turaev \cite{turaev:skein} and Alekseev, Kawazumi, Kuno and Naef  \cite{akkn2}.

Kawazumi and Kuno introduced the {\em completed} Goldman--Turaev Lie bialgebra $\Q\lambda(S)^\wedge$ of a framed surface. It is the {\em cyclic quotient} of the completion of the group algebra $\Q\pi_1(S,x)$ with respect to the powers of its augmentation ideal. One important property of it is that for surfaces with one boundary component, there is an inclusion
$$
\Der^\theta \p(S,x) \hookrightarrow \Q\lambda(S)^\wedge
$$
of Lie algebras. Consequently, the Turaev cobracket can be applied to $\Der^\theta \p(S,x)$. When $S$ and its framing are algebraic, the bracket and cobracket are both morphisms of MHS \cite{hain:goldman,hain:turaev}, so that the yoga of weights allows us to identify each with its associated weight graded in a way that is compatible with the geometric and arithmetic Johnson homomorphisms.

An important property of the completed Goldman--Turaev Lie bialgebra is that the cobracket almost vanishes on the arithmetic Johnson image. More precisely, as we show in Section~\ref{sec:johnson},
$$
W_{-2} \gbar_{S,\partial S} \subseteq \ker \delta_\xi
$$
and $\delta_\xi$ induces inclusions
$$
W_{-2} \ghat_{S,\partial S}/W_{-2} \gbar_{S,\partial S} \cong
H_1\big(\L(\sigma_3,\sigma_5,\dots)^\wedge\big) \hookrightarrow \Q\lambda(S)^\wedge/\ker \delta_\xi \hookrightarrow \Q\lambda(S)^\wedge\comptensor\Q\lambda(S)^\wedge
$$
on the abelianization of the motivic Lie algebra $\k = \ghat_{S,\partial S}/\gbar_{S,\partial S}$. This last property suggests that the cobracket should be a useful tool for identifying the images of the $\sigma_{2n+1}$'s in $\Der^\theta\p(S,x)$ modulo the geometric derivations $\gbar_{S,\partial S}$ and commutators of the $\sigma_{2m+1}$'s, the best one can hope for without specifying (for example) a pants decomposition of $S$.

The Goldman--Turaev Lie bialgebra is reviewed in Section~\ref{sec:gt-lie-bialg}. Work of Kawazumi and Kuno on how the Turaev cobracket constrains the image of the Johnson homomorphism is discussed and refined in Section~\ref{sec:johnson}.

\subsection{Cohomology and the Johnson image}

Another approach to computing the weight graded quotients of the image of the Johnson homomorphism is via cohomology. As we explain in Section~\ref{sec:coho}, the cohomology of $\u_{S,\partial S}$ plus its weight filtration determine the graded quotients of its lower central series. This shifts the focus to the problem of computing the cohomology of $\u_{S,\partial S}$ and its weight filtration. A more realistic goal is to compute the {\em stable} cohomology of $\u_{S,\partial S}$ with its weight filtration. This will determine the stable graded quotients of $\g_{S,\partial S}$.

As explained in Section~\ref{sec:comparison}, for each $\Sp(H)$-module $V$, there is a homomorphism
\begin{equation}
\label{eqn:cts_coho_u}
\big[H^\bdot(\u_{S,\partial S})\otimes V\big]^{\Sp(H)} \to H^\bdot(\Gamma_{S,\partial S},V)
\end{equation}
which generalizes the homomorphism (\ref{eqn:kawa_morita}) constructed by Kawazumi and Morita. The computation \cite{madsen-weiss} of the stable cohomology of mapping class groups by Madsen and Weiss and a result \cite{kawazumi-morita} of Kawazumi and Morita, imply that this map is stably surjective when $V$ is the trivial representation. Combined with results of Looijenga \cite{looijenga} and Garoufalidis and Getzler \cite{garoufalidis-getzler},\footnote{With corrections by Petersen \cite{petersen:letter} and alternative proof by Kupers and Randal-Williams \cite{krw}.} one can show that this map is stably surjective for all $V$. If one can show that the MHS on the cohomology group $H^m(\u_{S,\partial S})$ is {\em stably pure} in the sense that it has weight $m$ when $g$ is sufficiently large, then the homomorphism (\ref{eqn:cts_coho_u}) will be a stable isomorphism. This will give a computation of the stable value of $\Gr^W_\bdot \g_{S,\partial S}$ and an upper bound for the weight graded quotients of the Johnson image $\gbar_{S,\partial S}$. Such a ``purity'' result for $H^\bdot(\u_{S,\partial S})$ is equivalent to the question \cite[Q.~9.14]{hain-looijenga} of whether $H^\bdot(\u_{S,\partial S})$ is (stably) Koszul dual to the enveloping algebra of $\u_{S,\partial S}$.

\subsection{Is the Johnson homomorphism motivic?}

The geometric and arithmetic Johnson homomorphisms are not just morphisms of MHS, they also have an algebraic de~Rham description and are, after tensoring with $\Ql$, equivariant with respect to the natural Galois actions provided suitable base points are used. This provides strong evidence that they are motivic. One can argue about what this means, but Johnson homomorphisms should be closely related to algebraic cycles and algebraic $K$-theory. There is considerable evidence that this is the case. For example, Johnson's original homomorphism (\ref{eqn:classic_johnson}) for an algebraic curve $C$ is related to the Ceresa cycle in its jacobian $\Jac C$ (\cite[\S6]{hain:completions}, \cite[\S10]{hain:msri}) and also the Gross--Schoen cycle \cite{gross-schoen} in the 3-fold $C^3$. Another example is the connection between Johnson homomorphisms and mixed Tate motives, as explained in Section~\ref{sec:decompositions}.

Kawazumi and Morita \cite{kawazumi-morita} showed that when $\V$ is the trivial coefficient system $\Q$, the image of the homomorphism (\ref{eqn:cts_coho_u}) in $H^\bdot(\M_{g,1},\Q)$ is its (cohomological) tautological subring. More recently, Petersen, Tavakol and Yin \cite{pty} have defined for each partition $\bmu$ of an integer $|\bmu|$, the {\em twisted Chow group} $\CH^\bdot(\M_g,\bV_\bmu)$ of $\M_g$. There is a cycle class map
\begin{equation}
\label{eqn:chow}
\CH^k(\M_g,\bV_\bmu) \to H^{2k-|\bmu|}(\M_g;\V_\bmu)
\end{equation}
where $\V_\bmu$ denotes the variation of Hodge structure of weight $-|\bmu|$ over $\M_g$ that corresponds to the irreducible $\Sp(H)$-module corresponding to $\bmu$. They also define the {\em twisted tautological subgroups} $R^\bdot(\M_g,\bV_\bmu)$ of $\CH^\bdot(\M_g,\bV_\bmu)$.
The constructions of Petersen, Tavakol and Yin and the construction of Kawazumi and Morita imply that the image of
$$
\big[W_{2k-2|\bmu|}H^{2k-|\bmu|}(\u_g)\otimes V_\bmu\big]^{\Sp(H)} \to H^{2k-|\bmu|}(\M_g;\V_\bmu)
$$
equals the image of $R^k(\M_g,\bV_\bmu)\otimes\Q$ under the cycle map (\ref{eqn:chow}). Computations of Petersen--Tavakol--Yin \cite{pty} and of Morita--Sakasai--Suzuki \cite{morita-sakasai-suzuki} imply that this map is stably an isomorphism onto its image when $2k-|\bmu| \le 6$. This suggests that this map is stably an isomorphism onto its image in $H^\bdot(\M_g;\V_\bmu)$ and that there is a deep connection between $\u_g$ and tautological algebraic cycles.

\begin{problem}
For each $g\ge 0$, construct a Voevodsky motive whose Hodge/Betti realization is the coordinate ring $\cO(\cG_g)$ of $\cG_g$ with its natural MHS, where $\cG_g$ is the relative completion of the fundamental group of $\M_g$ with respect to a suitable base point. Use this to show that the Johnson homomorphism is motivic.
\end{problem}

This is known in genus 0 by the work of Deligne and Goncharov \cite{deligne-goncharov}.

\subsection{The most optimistic landscape}

A summary of the most optimistic statements about the Johnson homomorphisms (for surfaces with one boundary component) that, in the opinion of the author, have a reasonable chance of being true are:
\begin{enumerate}

\item The geometric Johnson homomorphism $\g_{S,\partial S} \to \Der\p(S,x)$ is injective, at least stably. Equivalently, for each $m\ge 0$, $\g_{S,\partial S}/W_{-m} = \gbar_{S,\partial S}/W_{-m}$, at least when $g$ is sufficiently large.

\item The kernel of the Turaev cobracket $\delta_\xi : \Q\lambda(S)^\wedge \to \Q\lambda(S)^\wedge \comptensor \Q\lambda(S)^\wedge$ satisfies:
$$
\Der^\theta \p(S,\vv) \cap \ker\delta_\xi =  \ghat_{S,\partial S} \cap \ker\delta_\xi.
$$
Corollary~\ref{cor:im_ghat} would then determine the size of the Johnson image $\gbar_{S,\partial S}$ as $\ghat_{S,\partial S}/\gbar_{S,\partial S} \cong \k$ and
$$
\ghat_{S,\partial S}/(\ghat_{S,\partial S} \cap \ker\delta_\xi) \cong H_1(S) \oplus H_1(\k) = H_1(S)\oplus \bigoplus_{m\ge 1} \Q[\sigma_{2m+1}].
$$

\item For each weight $m$ and partition $\bmu$, the natural map
$$
\Gr^W_m\big[H^\bdot(\u_{S,\partial S})\otimes V_\bmu\big]^{\Sp(H)} \to \Gr^W_m H^\bdot(\G_{S,\partial S},V_\bmu)
$$
is stably an isomorphism, where $V_\bmu$ is the $\Sp_g$-module $V_\bmu$ corresponding to $\bmu$. Equivalently, for each $m\ge 0$, the stable value of $H^m(\u_{S,\partial S})$ is a pure Hodge structure of weight $m$. If so, this will determine the stable values of each $(\Gr^m_L T_g)\otimes \Q$.

\end{enumerate}
These statements are discussed in much more detail in the body of the paper. See Questions~\ref{quest:johnson_inj} and \ref{quest:ker_delta}, Conjectures~\ref{conj:stab_iso} and \ref{conj:purity}, and Corollary~\ref{cor:conj}.

\subsection{Of what use is mixed Hodge theory to topologists?}

Many of the deep results about the topology of algebraic varieties do not have have a known (if any) topological proof. The prime example of such a result is the Hard Lefschetz Theorem, which asserts that if $X$ is a smooth projective variety of dimension $n$ in $\P^N$, then intersecting with a codimension $k$ linear subspace $L^k$ of $\P^N$ induces an isomorphism
$$
\cap\, L^k : H^{n+k}(X;\Q) \to H^{n-k}(X;\Q)
$$
whenever $0\le k \le n$. It was first proved by Hodge in his 1950 ICM talk \cite{hodge:icm} using Hodge theory.

Orbifold fundamental groups of complex algebraic varieties that are quotients of algebraic manifolds by a finite group are very special. Much of what is known about such groups has been proved using Hodge theory, often using the yoga of weights. Early and important examples of the use of Hodge theory in this context are Morgan's restrictions on fundamental groups of smooth varieties, \cite{morgan}. Mapping class groups, being fundamental groups of moduli spaces of curves, are examples of such groups. One application of Hodge theory to mapping class groups is the statement that when $g\ge 3$, the Johnson image Lie algebra
$$
\big(\bigoplus_{k\ge 1} \Gr_J^k T_{S,\partial S}\big) \otimes \Q
$$
is generated in degree 1. This is an immediate corollary of \cite[Lem.~4.5]{hain:torelli}, Johnson's computation of $H_1(T_{S,\partial S};\Q)$ and the yoga of weights. (See Corollary~\ref{cor:gen_deg_1}.) It is a standard tool in the study of Johnson homomorphisms as explained in Morita's survey \cite{morita:survey}. It has no known topological proof.

Another application of Hodge theory to Torelli groups is to show that when $g\ge 3$, the Lie algebra $\t_g$ of the unipotent completion of the Torelli group $T_g$ is finitely presented and is quadratically presented when $g\ge 4$. The proof in \cite{hain:torelli} combines basic results on the Hodge theory of mapping class groups  with Kabanov's result that in degree 2, the cohomology of $\M_g$ with coefficients in a symplectic local system $\V_\bmu$ is isomorphic to the intersection cohomology of the Satake compactification of $\M_g$ when $g\ge 6$. These are combined with Gabber's purity result for intersection cohomology \cite{bbd} (or Saito's Hodge modules \cite{saito}) to obtain the result. Once again, there is no known topological proof of this fact.

\subsection{Structure of the paper} There are five parts. Each part has its own introduction which serves as a guide to its contents. There are also appendices and an index of notation.

The first part provides topological background on mapping class groups, the Goldman--Turaev Lie bialgebra of a surface and a discussion of mapping class group actions on framings of surfaces. The second part is central. In it completions of all of the basic objects --- surface groups, mapping class groups, and the Goldman--Turaev Lie bialgebra --- are constructed.  These are used to construct the geometric Johnson homomorphism. There is also an introduction to Deligne's mixed Hodge theory with a focus on those aspects most relevant to studying the topological and combinatorial aspects of Johnson homomorphisms. Presentations of the completions of mapping class groups are given using the splittings of the weight filtration given by Hodge theory. Arithmetic Johnson homomorphisms are introduced in the third part. Takao's affirmation of Oda's conjecture is used to prove that the arithmetic Johnson image divided by the geometric Johnson image is the motivic Lie algebra of mixed Tate motives. The ``almost vanishing'' of the cobracket on the arithmetic Johnson image is established. The fourth part concerns the cohomological approach to determining the images of the Johnson homomorphisms, at least stably. The fifth part is an afterword. It contains a discussion of what aspects of Torelli groups lie outside the realm of Johnson homomorphisms and their significance for arithmetic geometry and topology.

\subsection{What is not included}

Due to length constraints, we have not covered every aspect of Johnson homomorphisms. Three notable topics that we have omitted are:
\begin{enumerate}

\item The relationship between the stable cohomology of $\Der^\theta\p(S,x)$ and the homology of outer automorphism groups of free groups that is due to Kontsevich \cite{kontsevich}. For an exposition, see \cite{conant-vogtmann}.

\item Drinfeld's prounipotent version GRT of the Grothendieck--Teichm\"uller group, \cite{drinfeld} and its relation to Johnson homomorphisms. (See \cite{enriquez} for the genus 1 case.)

\item Connections to non-commutative Poisson geometry via the Kashiwara--Vergne problem. See \cite{akkn2} and the references therein.

\end{enumerate}

\subsection{Further reading}
This survey focuses mainly on recent developments, especially those involving the Goldman--Turaev Lie bialgebra, Hodge theory and, to a lesser extent, Galois actions. Earlier surveys include Johnson's original survey \cite{johnson:survey}. Details of the equivalence of the various definitions of the original Johnson homomorphism can be found in the author's article \cite{hain:msri}. The first comprehensive survey of higher Johnson homomorphisms was given by Morita in \cite{morita:survey}. Although it is primarily a survey of the topological aspects of Johnson homomorphisms and their relation to the cohomology of outer automorphism groups of free group (via the work of Kontsevich), Morita does discuss the problem of understanding the Galois image from a topologists point of view. The most recent survey is Satoh's \cite{satoh:survey}, which gives a nice introduction to Johnson homomorphisms associated to mapping class groups and automorphism groups of free groups.  Open problems regarding Torelli groups can be found in \cite{morita:problems} and \cite{hain:problems}.

\bigskip

\noindent{\em Acknowledgments:}
I am especially grateful to Florian Naef, who patiently answered my questions about his joint work with Alekseev, Kawazumi and Kuno, and to Dan Petersen for his correspondence regarding the material on the stable cohomology of relative completion in Section~\ref{sec:coho}. I would like to thank Anton Alekseev, Nariya Kawazumi, Yusuke Kuno and Shigeyuki Morita for comments on various drafts of this survey. I would also like to thank the two referees for their careful reading of the manuscript and their helpful suggestions for improving the exposition.

\part{MCGs and the Goldman--Turaev Lie Bialgebra}

This part summarizes the topological background central to the rest of this survey. It reviews some basic facts about mapping class and Torelli groups, and introduces the Goldman--Turaev Lie bialgebra. Since the Turaev cobracket depends on a framing, we also review Kawazumi's work on mapping class group orbits of framings of a surface and determine the stabilizer of the homotopy class of a framing in the corresponding mapping class group.

\section{Mapping Class Groups}
\label{sec:topology}

Suppose that $\Sbar$ is a connected, closed, oriented surface of genus $g$ and that $P$ and $Q$ are disjoint finite subsets of $\Sbar$. Suppose that $\vV$ is a set $\{\vv_q : q\in Q\}$ of non-zero tangent vectors of $\Sbar$, where $\vv_q \in T_q\Sbar$. Set
$$
S = \Sbar - (P\cup Q).
$$
We will always assume that $S$ is hyperbolic. That is, that $S$ has negative Euler characteristic:
$$
2g-2 + \#P + \#Q > 0.
$$

The corresponding mapping class group \label{def:mcg}
$$
\G_{\Sbar,P+\vV} := \pi_0(\Diff^+(\Sbar,P\cup\vV))
$$
is the group of isotopy classes of orientation preserving diffeomorphisms of $\Sbar$ that fix $P\cup Q$ pointwise and fix each of the tangent vectors $\vv_q$. The classification of surfaces implies that the mapping class groups of two surfaces of the same type are isomorphic, and any two such isomorphisms that respect the labelling of the boundary components and marked points will differ by conjugation by an element of the mapping class group of one of the surfaces. So, up to an inner automorphism, the mapping class group of a surface depends only on $g$ and the cardinalities $n$ of $P$ and $r$ of $\vV$. Often we will denote it by $\G_{g,n+\vr}$. We will refer to the $n+r$ punctured surface with its tangent vectors $(\Sbar,P,\vV)$ as a surface of type $(g,n+\vr)$.\footnote{Sometimes we will be less formal, and refer to $(S,\vV)$ as a surface of type $(g,n+\vr)$.} Note that $S$ is hyperbolic if and only if $2g-2+n+r>0$. The definition of $\G_{S,P+\vV}$ applies equally to disconnected surfaces, each of whose components is hyperbolic. The product of the mapping class groups of its components is a normal subgroup of finite index.

A complex structure on $(\Sbar,P,\vV)$ is an orientation preserving diffeomorphism
\begin{equation}
\label{eqn:cplex_str}
\phi : (\Sbar,P,\vV) \to (\Xbar,Y,\vV')
\end{equation}
where $\Xbar$ is a compact Riemann surface, $Y$ is a finite subset and $\vV$ is a finite set of non-zero tangent vectors. We will refer to $(\Xbar,Y,\vV')$ as a {\em complex curve} of type $(g,n+\vr)$.

There is a moduli space $\M_{g,n+\vr}$ of complex curves of type $(g,n+\vr)$. \label{def:mod_sp} It will be viewed as an orbifold. As such, it is the classifying space of $\G_{g,n+\vr}$. The complex structure $\phi$ determines a point in its universal covering (Teichm\"uller space) and a canonical isomorphism
$$
\phi_\ast : \G_{\Sbar,P+\vV} \overset{\simeq}{\To} \pi_1(\M_{g,n+\vr},\phi).
$$
Note that $\M_{g,n+\vr}$ is a $(\C^\ast)^r$ bundle over $\M_{g,n+r}$ so that one has a central extension
$$
0 \to \Z^r \to \G_{g,n+\vr} \to \G_{g,n+r} \to 1.
$$

\begin{remark}
\label{rem:blow-up}
From a topological point of view, each tangent vector $\vv \in \vV$ corresponds to a boundary component with a marked point on it. This can be seen using real oriented blow ups. The surface $(S,\vV)$ plays the same role as the surface $(\Shat,\partial \Shat)$, where $\Shat$ is the real oriented blow up of $\Sbar-P$ at $Q$. Blowing up replaces each $q\in Q$ by the circle $(T_q\Sbar-\{0\})/\R_{>0}^\times$ of tangent directions at $q$, and replaces the corresponding $\vv_q\in \vV$ by its $\R_{>0}$-orbit $v_q$. The inclusion $S\hookrightarrow \Shat$ is a homotopy equivalence. It induces an isomorphism $\pi_1(S,\vv_q) \to \pi_1(\Shat,v_q)$. In fact, this isomorphism may be taken as the definition of $\pi_1(S,\vv_q)$. Rotating the tangent vector $\vv_q$ in $T_q\Sbar$ corresponds to a Dehn twist about the corresponding boundary component. More details can be found in \cite[\S12.1]{hain:goldman}.
\end{remark}

For each commutative ring $A$, set \label{def:H}
$$
H_A := H_1(\Sbar;A).
$$
The intersection pairing $\bil$ on $H_A$ is unimodular. Denote the corresponding symplectic group (the automorphisms of $H_A$ that preserve the intersection pairing) by $\Sp(H_A)$. The functor $A$ to $\Sp(H_A)$ is an affine group that we will denote by $\Sp(H)$. Likewise, we will regard $H$ as the unipotent group with $A$-rational points $H_A$. The choice of a symplectic basis $\aa_1,\dots,\aa_g,\bb_1,\dots,\bb_g$ of $H_A$ determines an isomorphism of $\Sp(H_A)$ with the group $\Sp_g(A)$ of $2g\times 2g$ that preserve the standard symplectic inner product on $A^{2g}$.

The action of $\Diff^+(\Sbar,P\cup\vV)$ on $\Sbar$ induces an action of $\G_{\Sbar,P+\vV}$ on $H_1(\Sbar)$ that preserves the intersection pairing. This corresponds to a homomorphism \label{def:rho}
$$
\rho : \G_{\Sbar,P+\vV} \to \Sp(H_\Z).
$$
It is well-known to be surjective. The Torelli group of type $(g,n+\vr)$ is its kernel:
\begin{equation}
\label{def:torelli}
T_{g,n+\vr} = T_{\Sbar,P+\vV} :=\ker \rho.
\end{equation}

For each $\vv \in \vV$, one has the fundamental group $\pi_1(S,\vv)$.\label{def:pi(S,v)} This is Deligne's fundamental group of $S$ with tangential base point $\vv$. Elements of this group are homotopy classes of piecewise smooth paths $\gamma : [0,1]\to \Sbar$ that begin at the anchor point $q_o$ of $\vv$ with initial velocity $\vv$ and return to it with velocity $-\vv$. We also require that $\gamma(t) \in S$ when $0<t<1$. Equivalently, $\pi_1(S,\vv)$ is the fundamental group of the surface obtained from $\Sbar$ by removing all points in $P\cup Q$ except $q_o$ and replacing $q_o$ by the circle of directions in $T_{q_o}\Sbar$. This give a new surface $\Shat$ with one boundary circle and a distinguished point $[\vv] \in \partial \Shat$ that corresponds to the direction of $\vv$. In this setup
$$
\pi_1(S,\vv) := \pi_1(\Shat,[\vv]).
$$
More details can be found in \cite[\S12.1]{hain:goldman}.

\begin{remark}
We are forced to use tangential base points as surfaces with boundary circles do not support an algebraic structure. In order to apply Hodge theory, we need to work with algebraic curves.
\end{remark}

The mapping class group action induces an injective homomorphism
$$
\G_{\Sbar,P+\vV} \hookrightarrow \Aut \pi_1(S,\vv).
$$
Its image lies in the subgroup of automorphisms that fix the ``boundary loop'' $\sigma_o$ that corresponds to rotating $\vv$ once about its anchor point $q$ in the positive direction. (Equivalently, traversing the corresponding boundary loop once.)

In this paper, we are primarily interested in the case where $(g,n+\vr) = (g,\uu)$, where $S$ is a surface of genus $g\ge 1$ with one ``boundary component''.

\section{The Goldman--Turaev Lie bialgebra}
\label{sec:gt-lie-bialg}

Suppose that $\kk$ is a commutative ring. Denote the set of free homotopy classes of maps $S^1 \to S$ in the surface $S$ by $\lambda(S)$ and the free $\kk$-module it generates by $\kk\lambda(S)$. The set $\lambda(S)$ is the set of conjugacy classes of $\pi_1(S,x)$ and $\kk\lambda(S)$ can be described algebraically as the ``cyclic quotient''
$$
\label{def:cyclic_gp_alg}
|\kk\pi_1(S,x)| := \kk\pi_1(S,x)/\big\langle uv-vu: u,v \in \kk\pi_1(S,x)\big\rangle
$$
of the group algebra. Note that $\langle uv-vu: u,v \in \kk\pi_1(S,x)\rangle$ is the {\em subspace} spanned by the commutators $uv-vu$. It is not an ideal.

Denote the augmentation ideal of $\kk\pi_1(S,x)$ by $I_\kk$. The decomposition $\kk\pi_1(S,x) = \kk \oplus I_\kk$
descends to a canonical decomposition
$$
\kk\lambda(S) = \kk \oplus I\kk\lambda(S)
$$
where the copy of $\kk$ is spanned by the trivial loop.

\subsection{The Goldman bracket}

Goldman \cite{goldman} defined a binary operation
$$
\label{def:goldman}
\gold : \Z\lambda(S)\otimes \Z\lambda(S) \to \Z\lambda(S)
$$
giving $\Z\lambda(S)$ the structure of a Lie algebra over $\Z$. Briefly, the bracket of two oriented loops $a$ and $b$ is defined by choosing transverse, immersed representatives $\alpha$ and $\beta$ of the loops, and then defining
$$
\{a,b\} = \sum_p \e_p(\alpha,\beta)\, [\alpha \#_p \beta]
$$
where the sum is taken over the points $p$ of intersection of $\alpha$ and $\beta$, $\e_p(\alpha,\beta)\in \{\pm 1\}$ denotes the local intersection number of $\alpha$ and $\beta$ at $p$, and where $[\alpha \#_p \beta]$ denotes the free homotopy class of the oriented loop $\alpha \#_p \beta$ obtained by joining $\alpha$ and $\beta$ at $p$ by a simple surgery.

\subsection{The Turaev cobracket}

Since $S$ is oriented, framings of $S$ correspond to nowhere vanishing vector fields on $S$. For each choice of a framing $\xi$ (a vector field) of $S$, there is a map
$$
\delta_\xi : \Z\lambda(S) \to \Z\lambda(S)\otimes \Z\lambda(S),
$$
called the {\em Turaev cobracket}, that gives $\Z\lambda(S)$ the structure of a Lie coalgebra.\footnote{The axioms of a Lie bialgebra $C$ are obtained by reversing the arrows in the definition of a Lie algebra. Specifically, the cobracket is skew symmetric, and the dual Jacobi identity holds.} The value of the cobracket on a loop $a \in \lambda(X)$ is obtained by representing it by an immersed circle $\alpha : S^1 \to X$ with transverse self intersections and trivial winding number relative to $\xi$. Each double point $p$ of $\alpha$ divides it into two loops based at $p$, which we denote by $\alpha'_p$ and $\alpha_p''$. Let $\e_p = \pm 1$ be the intersection number of the initial arcs of $\alpha_p'$ and $\alpha_p''$. The cobracket of $a$ is then defined by
\begin{equation}
\label{eqn:def_turaev}
\delta_\xi(a) = \sum_p \e_p(a'_p \otimes a''_p - a''_p \otimes a'_p),
\end{equation}
where $a_p'$ and $a_p''$ are the homotopy classes of $\alpha_p'$ and $\alpha_p''$, respectively. It depends only on the homotopy class of $\xi$.

The cobracket $\delta_\xi$ and the Goldman bracket endow $\Z\lambda(S)$ with the structure of a Lie bialgebra \cite{turaev:skein,kk:intersections}. This simply means that the cobracket and bracket satisfy
\begin{equation}
\label{eqn:derivation}
\delta_\xi(\{a,b\}) = a\cdot \delta_\xi(b) - b \cdot \delta_\xi(a)
\end{equation}
where the dot denotes the adjoint action
$$
a : u\otimes v \mapsto \{a,u\}\otimes v + u\otimes\{a,v\}
$$
of $\Z\lambda(S)$ on $\Z\lambda(S)^{\otimes 2}$. This Lie bialgebra is {\em involutive}, as was observed by Chas \cite{chas}, in the sense that the composition
$$
\xymatrix{
\Z\lambda(S) \ar[r]^(.35){\delta_\xi} & \Z\lambda(S)\otimes \Z\lambda(S) \ar[r]^(.6){\gold}  & \Z\lambda(S)
}
$$
vanishes.

The {\em reduced cobracket}
$$
\label{def:red_cobracket}
\deltabar : I\Z\lambda(S) \to I\Z\lambda(S) \otimes I\Z\lambda(S)
$$
is the map induced by composing the restriction of $\delta_\xi$ to $I\Z\lambda(S)$ with the square of the projection $\Z\lambda(S) \to I\Z\lambda(S)$. It does not depend on the framing $\xi$. The reduced cobracket was first defined by Turaev \cite{turaev:78} on $\Z\lambda(S)/\Z\cong I\Z\lambda(S)$ and lifted to $\Z\lambda(S)$ for framed surfaces in \cite[\S18]{turaev:skein} and \cite{akkn2}.

\begin{remark}

Framings of $S$ form a principal homogeneous space under $H^1(S;\Z)$. (The standard explanation is recalled at the beginning of Section~\ref{sec:framings}.) The definition of the cobracket $\delta_\xi$ implies that the corresponding cobrackets also form a principal homogeneous space under $H^1(S;\Z)$. Suppose that $\xi_0$ and $\xi_1$ are two framings of $S$. Then $\xi_1-\xi_0 = \alpha \in H^1(S;\Z)$. Since the reduced cobracket does not depend on the framing, $\deltabar_{\xi_0} = \deltabar_{\xi_1}$. There is therefore a function $f_\alpha : \Z\lambda(S)\to \Z\lambda(S)$ that depends only on $\alpha \in H^1(S;\Z)$ such that
$$
\delta_{\xi_1} = \delta_{\xi_0} + f_\alpha \otimes 1 - 1 \otimes f_\alpha.
$$
Property (\ref{eqn:derivation}) implies that $f_\alpha$ is a 1-cocycle on $\Z\lambda(S)$:
$$
f_\alpha(\{u,v\}) = u\cdot f_\alpha(v) - v\cdot f_\alpha(u).
$$
\end{remark}

\subsection{The Kawazumi--Kuno Action}
\label{sec:kk-action}

Suppose that $\vv',\vv'' \in \vV$. Denote the torsor of homotopy classes of paths in $S$ from $\vv'$ to $\vv''$ by $\pi(S;\vv',\vv'')$. There is an ``action''
$$
\kappa : \Z\lambda(S) \otimes \Z\pi(S;\vv',\vv'') \to \Z\pi(S;\vv',\vv'')
$$
which was defined by Kawazumi and Kuno \cite{kk:log}. The definition is similar to the definition of the Goldman bracket given above. It is compatible with path multiplication
$$
\Z\pi(S;\vv_1,\vv_2)\otimes \Z\pi(S;\vv_2,\vv_3) \to \Z\pi(S;\vv_1,\vv_3) \qquad \vv_1,\vv_2,\vv_3 \in \vV
$$
in the sense that
\begin{equation}
\label{eqn:kk-derivation}
\kappa(\gamma \otimes (bc)) = \kappa(\gamma\otimes b) c + b\kappa(\gamma\otimes c).
\end{equation}
In particular, when $\vv'=\vv''=\vv$, there is a Lie algebra homomorphism
$$
\label{def:kappa_v}
\kappa_\vv : \Z\lambda(S) \to \Der \Z\pi_1(S,\vv)
$$
Its image is contained in the Lie subalgebra
$$
\Der^\theta \Z\pi_1(S,\vv) := \{D \in \Der^\theta \Z\pi_1(S,\vv) : D(\sigma_o) = 0\},
$$
where $\sigma_o \in \pi_1(S,\vv)$ is the loop that corresponds to rotating $\vv$ once in the positive direction about its anchor point.

\subsection{Power operations}
\label{sec:power_ops}

For $n\in \Z$, the power operation $\psi_n : \kk\pi_1(S,\vv) \to  \kk\pi_1(S,\vv)$ is defined by taking $\gamma \in \pi_1(S,\vv)$ to $\gamma^n$. Denote the standard coproduct on the group algebra by $\Delta$. Since $\Delta(\gamma^n) = \gamma^n\otimes \gamma^n$, it follows that the power operations commute with the coproduct
\begin{equation}
\label{eqn:adams+coprod}
\Delta\circ \psi_n = (\psi_n\otimes \psi_n)\circ \Delta.
\end{equation}
The power map $\psi_n$ of $\kk\pi_1(S,\vv)$ descends to a power map
$$
\psi_n : \kk\lambda(S) \to \kk\lambda(S),
$$
It would be useful to know how the power operations $\psi_n$ interact with the bracket and cobracket. One case where this can be understood is where $\gamma$ is an an imbedded circle. In this case
\begin{equation}
\label{eqn:psi-scc}
\delta_\xi \circ \psi_n (\gamma) = n\cdot \rot_\xi(\gamma) \big( \psi_n(\gamma)\otimes 1 - 1 \otimes \psi_n(\gamma)\big)
\end{equation}
where $\rot_\xi(\gamma)$ denotes the rotation number of $\gamma$ with respect to $\xi$. \label{def:rot}

\section{MCG Orbits of framings and the stabilizer of the cobracket}
\label{sec:framings}

In this section we assume, for simplicity, that $(\Sbar,\vV)$ is a surface of type $(g,\uu)$ where $g>0$. Similar results hold in general.\footnote{For the general case, see \cite[\S11]{hain:turaev}.} The cobracket $\delta_\xi$ depends non-trivially on the framing $\xi$ and the action of the mapping class group $\Gamma_{g,\uu}$ on $\Z\lambda(S)$ does not preserve the cobracket. Here we identify the stabilizer of a framing. It preserves the cobracket.

Since $\Sbar$ is oriented, we can (and will) regard its tangent bundle $T\Sbar$ as a smooth complex line bundle. If $\xi_0$ and $\xi_1$ are two framings of $S$, then $\xi_1 = f \xi_0$, where $f: S \to \C^\ast$. Their homotopy classes differ by the homotopy class of $f$, which is an element of $H^1(S;\Z) = H^1(\Sbar;\Z)$.

Since the tangent bundle of $S$ is trivial, there is a short exact sequence
\begin{equation}
\label{eqn:ses}
0 \to \Z \to H_1(T'S;\Z) \to H_1(S;\Z) \to 0,
\end{equation}
where $T'S$ denotes the bundle of non-zero tangent vectors of $S$. Each framing $\xi : S \to T'S$ of $S$ induces a splitting $s(\xi)$ of this sequence. The difference between two such splittings is naturally an element of $H^1(S;\Z)$:
$$
s(\xi_1) - s(\xi_0) \in \Hom(H_1(S);\Z) \cong H^1(\Sbar;\Z).
$$
This equals the class of $f : S \to \C^\ast$.

The splitting $s(\xi_0)$ induces an isomorphism
$$
H_1(T'S;\Z) \cong H_1(S;\Z) \oplus \Z = H_\Z \oplus \Z
$$
and an isomorphism of the group of automorphisms of the extension (\ref{eqn:ses}) that respect the intersection pairing on $H_1(S)$ with
$$
\Sp(H_\Z)\ltimes H_\Z.
$$

The homotopy action of $\G_{g,\uu}$ on $S$ induces an action of it on $H_1(T'S)$ that preserves the sequence (\ref{eqn:ses}), and therefore a homomorphism $\rhotilde_{\xi_0} : \G_{g,\uu} \to \Sp(H_\Z)\ltimes H_\Z$. It also induces a left action on framings. Identify $\Sp(H_\Z)$ with the subgroup of $\Sp(H_\Z)\ltimes H_\Z$ consisting of the $(\phi,u)$ with $u=0$.

\begin{proposition}[{\cite[Cor.~11.3]{hain:turaev}}]
When $g\ge 2$, the stabilizer of the homotopy class of $\xi_0$ in $\G_{g,\uu}$ is the inverse image of $\Sp(H_\Z)$ under $\rhotilde_{\xi_0}$.
\end{proposition}

Denote the stabilizer of the homotopy class of $\xi$ by $\G_{g,\uu}^\xi$. \label{def:stab_xi} Since the cobracket $\delta_\xi$ depends only on the homotopy class of $\xi$, we have:

\begin{corollary}
The action of the mapping class group $\G_{g,\uu}$ on $\Z\lambda(S)$ preserves the Goldman bracket and the stabilizer $\G_{g,\uu}^\xi$ preserves the cobracket $\delta_\xi$.
\end{corollary}

As explained below, the next result is a special case of \cite[Cor.~11.3]{hain:turaev}.

\begin{proposition}
The restriction of $\rhotilde_{\xi_0}$ to the Torelli group is the homomorphism
$$
\xymatrix{
T_{g,\uu} \ar[r] & H_1(T_{g,\uu}) \ar[r]^{\tau} & \Lambda^3 H_\Z \ar[r]^{2c} & H_\Z
}
$$
where $\tau$ is Johnson's homomorphism and $c$ is the $\Sp(H)$-equivariant contraction
$$
u\wedge v \wedge w \mapsto \langle u,v \rangle w + \langle v, w \rangle u + \langle w,u \rangle v.
$$
The restriction of $\rhotilde_{\xi_0} : H_1(T_{g,\uu}) \to H_\Z$ to the image of the ``point pushing subgroup'' of $T_{g,\uu}$ is the multiplication by $2g-2$ map $H_\Z \to H_\Z$.
\end{proposition}

This follows from \cite[Cor.~11.3]{hain:turaev} with $\dd = (2g+2)$ and the fact that the homomorphism $H_1(T_{g,1}) \to H$ induced by the normal function on $\M_{g,1}$ that takes $[C,x]$ to $(2g-2)x - K_C \in \Jac (C)$ is the map $2c\circ \tau$.

\begin{corollary}
When $g\ge 2$, the function $\G_{g,\uu}/\G_{g,\uu}^{\xi_0} \to H_\Z$ induced by $\rhotilde_{\xi_0}$ that takes the coset of $\phi$ to $\rhotilde_{\xi_0}(\xi_0) = \phi_\ast \xi_0 - \xi_0$
is injective and has image $2(g-1)H_\Z$.
\end{corollary}

Note that, since $\G_{g,\uu}^{\xi_0}$ is not a normal sugbroup of $\G_{g,\uu}$, the coset space $\G_{g,\uu}/\G_{g,\uu}^{\xi_0}$ is not a group and the function above is not a homomorphism. However, this is a map of $H$-torsors, where $u \in H$ acts on $H$ by $u : v \mapsto v+(2g-2)u$ and the $H$-action on $\G_{g,\uu}^{\xi_0}$ is induced by left translation by elements of the ``point pushing'' subgroup of $\G_{g,\uu}$.

\subsection{Orbits of framings}

Kawazumi \cite{kawazumi:framings} determined the mapping class group orbits of homotopy classes of framings of a surface with finite topology. We recall the classification in the case of surfaces $(\Sbar,\vv)$ of type $(g,\uu)$, where $g\ge 1$ and $\vv\in T_q\Sbar$. Set $S=\Sbar - \{q\}$.

For a simple closed curve $c$ on $S$, let
$$
f_\xi(c) = 1 + \rot_\xi(c) \bmod 2 \in \F_2,
$$
where $\rot_\xi(c)$ denotes the winding number of $c$ relative to $\xi$. The Poincar\'e--Hopf Theorem implies that the index (local winding number) of a framing of $S$ at $q$ is $2-2g$. Then an elementary argument implies that $f_\xi(c)$ depends only on the homology class of $c$ in $H_1(\Sbar;\F_2)$, so that $f_\xi$ induces a well defined map $H_1(\Sbar;\F_2) \to \F_2$. It is easily verified to be an $\F_2$-quadratic form. The Arf invariant of $\xi$ is defined to be the Arf invariant
$$
\Arf(\xi) := \sum_{j=1}^g f_\xi(a_j)f_\xi(b_j)
$$
 of $f_\xi$, where $a_1,\dots,a_g,b_1,\dots,b_g$ is a symplectic basis of $H_{\F_2}$. The Arf invariant is constant on mapping class group orbits of homotopy classes of framings.

\begin{theorem}[Kawazumi \cite{kawazumi:framings}]
When $g>1$ there there are two $\G_{g,\uu}$ orbits of homotopy classes of framings of $S$. These are distinguished by their Arf invariants. If $g=1$, then the homotopy classes of two framings $\xi_0$ and $\xi_1$ are in the same $\G_{1,\uu}$ orbit if and only if $A(\xi_0) = A(\xi_1)$, where
$$
A(\xi) = \gcd\{\rot_{\xi}(\gamma): \gamma \text{ is a non-separating simple closed curve in } S\}.
$$
This is congruent to $1+\Arf(\xi)$ mod $2$.
\end{theorem}

\part{Completions and the Johnson Homomorphism}

This part is the core of the survey. Unipotent and relative unipotent completions of discrete groups are reviewed. These linearize a discrete group by replacing it by a proalgebraic group (i.e., an inverse limit of affine algebraic groups) and make power series methods available. The action of the mapping class group $\G_{S,\partial S}$ on $\pi_1(S,x)$ induces a homomorphism from the relative completion $\cG_{S,\partial S}$ of $\G_{S,\partial S}$ to the automorphism group of the unipotent completion of $\pi_1(S,x)$. The geometric Johnson homomorphisms is the induced Lie algebra homomorphism.

The second fundamental tool introduced in this section is Deligne's mixed Hodge theory. A brief and direct introduction to it is given in Section~\ref{sec:hodge}. Perhaps the most important aspect of mixed Hodge theory in this paper is that every mixed Hodge structure (MHS) has a natural weight filtration $W_\bdot$ that is preserved by morphisms of MHS and which has strong exactness properties, which we explain in terms of the tannakian property of the category of (graded polarizable) mixed Hodge structures.

 All of the completed invariants considered in this paper (unipotent fundamental group, relative completion of mapping class groups, Goldman--Turaev Lie bialgebra) that are associated to the decorated surface $(\Sbar,P,\vV)$ inherit a natural MHS for which their algebraic operations are morphisms of MHS once one has chosen an algebraic structure on the decorated surface. The exactness properties of mixed Hodge structures allow us to study the associated weight graded objects. In particular, they allow for the study of the higher Johnson homomorphisms. For example, it is Hodge theory in conjunction with Johnson's basic computation of the abelianization of the Torelli group $T_{S,\partial S}$ that imply that when $g\ge 3$, the image of the higher Johnson homomorphisms, regarded as a Lie algebra, is generated in degree 1 after tensoring with $\Q$. This result is Corollary~\ref{cor:gen_deg_1} in this paper.

The unipotent completion of the Goldman--Turaev Lie bialgebra is also presented and the basic structural results of Kawazumi and Kuno are presented. Of these, the most important is that, for a surface $(S,\vv)$ of type $(g,\uu)$, the Kawazumi--Kuno action induces a Lie algebra isomorphism
$$
\Der^\theta \Q\pi_1(S,\vv)^\wedge \cong I\Q\lambda(S)^\wedge.
$$
This result is fundamental as it implies that the Turaev cobracket can be applied to $\Der^\theta \p(S,\vv)$, where $\p(S,\vv)$ denotes the Lie algebra of the unipotent completion of $\pi_1(S,\vv)$. In the next part, we will show that the cobracket ``almost vanishes'' on the image of the arithmetic Johnson homomorphism and therefore on the image of geometric Johnson homomorphism.

In Section~\ref{sec:mcg_hodge} we explain the initial consequences of the existence of MHSs for the unipotent fundamental groups of smooth algebraic curves, their associated Goldman--Turaev Lie bialgebras and for completed mapping class groups. We give explicit formulas for the associated gradeds of the Goldman bracket and Turaev cobracket for surfaces of type $(g,\uu)$. In Section~\ref{sec:presentations}, we expand on this by giving explicit presentations of the relative completions of mapping class groups in all genera except 2, where only a partial presentation is known.

\section{Completions}
\label{sec:completions}

\subsection{Unipotent completion}

Unipotent (aka, Malcev) completion of a discrete group has become a standard tool for studying certain discrete groups, especially those close to being free groups. The approach below follows the one given by Quillen in \cite[Appendix B]{quillen} that uses the $I$-adic completion of the group algebra. Additional explanations are given in \cite[\S3]{hain:completions}. One can also regard it as a special case of relative completion, which is explained in the next section.

Suppose that $\pi$ is a discrete group and that $\kk$ is a commutative ring. The group algebra $\kk\pi$ is a Hopf algebra with comultiplication $\Delta : \kk\pi \to \kk\pi \otimes \kk\pi$ induced by taking each $\gamma\in\pi$ to $\gamma\otimes\gamma$. As previously, $I_\kk$ denotes the augmentation ideal of $\kk\pi$. Its powers define a topology on $\kk\pi$. The $I$-adic completion of $\kk\pi$ is
$$
\label{def:comp_gp_alg}
\kk\pi^\wedge := \varprojlim_n \kk\pi/I_\kk^n.
$$
It is a complete Hopf algebra with (completed) coproduct
$$
\label{def:coprod}
\Delta : \kk\pi^\wedge \to \kk\pi^\wedge \comptensor \kk\pi^\wedge
$$
induced by the coproduct of $\kk\pi$.

When $\pi$ is finitely generated, the unipotent completion $\pi^\un$ of $\pi$ is the prounipotent $\Q$-group whose group of $\kk$-rational points (where $\kk$ is a $\Q$-algebra) is the set of grouplike elements
$$
\label{def:pi^un}
\pi^\un(\kk) = \{u \in \kk\pi^\wedge : u\neq 0,\ \Delta u = u\otimes u\} \subset 1 + I_\kk^\wedge.
$$
Its Lie algebra $\p$ is the set
$$
\label{def:p}
\p := \{v \in \Q\pi^\wedge : \Delta v = v\otimes 1 + 1 \otimes v\} \subset I_\Q^\wedge.
$$
The natural map $\pi \to \Q\pi^\wedge$ induces a natural homomorphism $\pi \to \pi^\un(\kk)$.

The exponential and logarithm maps
$$
\xymatrix{
I_\kk^\wedge \ar@/^/[r]^\exp & \ar@/^/[l]^\log 1 + I_\kk^\wedge
}
$$
are mutually inverse bijections and restrict to a bijection $\p\otimes \kk \cong \pi^\un(\kk)$. This bijection is a group isomorphism if we give $\p$ the multiplication given by the Baker--Campbell--Hausdorff (BCH) formula. It implies that every element of $\pi^\un(\Q)$ has a logarithm that lies in $\p$.

The universal enveloping algebra $U\p$ is a topological algebra whose $I$-adic completion is naturally isomorphic to $\Q\pi^\wedge$ and the coordinate ring of $\pi^\un$ is the continuous dual of $\Q\pi$:
$$
\cO(\pi^\un) = \Hom^\cts_\Q(\Q\pi^\wedge,\Q) := \varinjlim_n\Hom_\Q(\Q\pi/I^n,\Q).
$$

\subsection{The completed Goldman--Turaev Lie bialgebra}

Now suppose that $(\Sbar,P)$ is a decorated surface and that $\kk$ is a $\Q$-algebra. The $I$-adic topology on $\kk\lambda(S)$ is the quotient topology induced by the cyclic quotient map
$$
\kk\pi_1(S,x) \to |\kk\pi_1(S,x)| = \kk\lambda(S).
$$
It does not depend on the choice of base point $x\in S$. Denote the image of $I_\kk^k$ by $I^k\kk\lambda(S)$. The $I$-adic completion of $\kk\lambda(S)$ is
$$
\label{def:gt_comp}
\kk\lambda(S)^\wedge = \varprojlim_n \kk\lambda(S)/I^n \kk\lambda(S).
$$
Kawazumi and Kuno showed in \cite[\S4]{kk:groupoid} that the Goldman bracket is continuous in the $I$-adic topology on $\kk\lambda(S)^\wedge$ and therefore induces a Lie bracket
$$
\gold : \kk\lambda(S)^\wedge \otimes \kk\lambda(S)^\wedge \to \kk\lambda(S)^\wedge.
$$
(This is also proved in \cite[\S8.2]{hain:goldman}.) In \cite[\S3]{kk:intersections} they show that when $S$ is framed with framing $\xi$, the cobracket $\delta_\xi$ is continuous so that it induces a cobracket
$$
\delta_\xi : \kk\lambda(S)^\wedge \to \kk\lambda(S)^\wedge \comptensor \kk\lambda(S)^\wedge.
$$
(This is also proved in \cite[Prop.~6.7]{hain:turaev}.)

Similarly, for a surface $(S,\vv)$, Kawazumi and Kuno show that the action
$$
\kappa : \kk\lambda(S) \otimes \kk\pi_1(S,\vv) \to \kk\pi_1(S,\vv)
$$
is continuous in the $I$-adic topology, and therefore induces a continuous map
$$
\kappahat : \kk\lambda(S)^\wedge \otimes \kk\pi_1(S,\vv)^\wedge \to \kk\pi_1(S,\vv)^\wedge
$$
and a continuous Lie algebra homomorphism
$$
\label{def:kappahat_v}
\kappahat_\vv : \kk\lambda(S)^\wedge \to \Der^\theta \kk\pi_1(S,\vv)^\wedge.
$$

\subsubsection{Relation to the derivation algebra}

Denote the Lie algebra of $\pi_1^\un(S,\vv)$ by $\p(S,\vv)$. The Lie algebra
$$
\Der^\theta \p(S,\vv)
$$
of continuous derivations of $\p(S,\vv)$ that fix $\theta := \log \sigma_o$, is the recipient of the Johnson homomorphism. It is the Lie subalgebra of $\Der^\theta \Q\pi_1(S,\vv)^\wedge$ consisting of those derivations of $\Q\pi_1(S,\vv)^\wedge$ that respect its Hopf algebra structure.

The following is a special case of a result \cite[Thm.~6.2.1]{kk:groupoid} of Kawazumi and Kuno. We sketch a proof using Hodge theory in Section~\ref{sec:proof-kk-isom}.

\begin{theorem}[Kawazumi--Kuno]
\label{thm:kk-isom}
If $(\Sbar,\vv)$ is a surface of type $(g,\uu)$, then the Lie algebra homomorphism
$$
\kappahat_\vv : \Q\lambda(S)^\wedge \to \Der^\theta \Q\pi_1(S,\vv)^\wedge.
$$
is surjective and has 1-dimensional kernel, which is spanned by the trivial loop.
\end{theorem}

This result implies that the Turaev cobracket induces a cobracket on $\Der^\theta\Q\pi_1(S,\vv)^\wedge$. This is key to bounding the size of the Johnson image, which we shall explain in Section~\ref{sec:johnson}.

\subsubsection{A dual PBW-like decomposition of $\Q\lambda(S)^\wedge$}

The Poincar\'e--Birkhoff--Witt Theorem implies that the symmetrization map
\begin{equation}
\label{eqn:pbw}
\prod_{n\ge 0} \Sym^n \p(S,\vv) \to \Q\pi_1(S,\vv)^\wedge
\end{equation}
defined by taking
$$
u_1 u_2 \dots u_n \in \Sym^n \p(S,\vv) \text{ to } \frac{1}{n!}\sum_{\sigma\in\Sigma_n} u_{\sigma(1)} u_{\sigma(2)} \dots u_{\sigma(n)} \in \Q\pi_1(S,\vv)^\wedge
$$
is a complete coalgebra isomorphism. This induces a PBW-like decomposition of $|\Q\lambda(S)^\wedge|$. Denote the image of $\Sym^n \p(S,\vv)$ in $\Q\lambda(S)^\wedge$ by $|\Sym^n \p(S,\vv)|$. \label{def:cyclic_sym}

\begin{lemma}
For each $k\ge 1$ and $n\ge 0$, $\psi_k$  acts on $\Sym^n\p(S,\vv)$ and $|\Sym^n\p(S,\vv)|$ as multiplication by $k^n$.
\end{lemma}

\begin{proof}
Observe that the projection map $\Q\pi_1(S,\vv)^\wedge \to \Q\lambda(S)^\wedge$ commutes with each $\psi_k$. Since $\psi_k$ acts on the group-like elements of $\Q\pi_1(S,\vv)^\wedge$ by the $k$th-power map, it acts on $\p(S,\vv)$ as multiplication by $k$. Since $\psi_k$ commutes with the coproduct (\ref{eqn:adams+coprod}), and since the PBW isomorphism is a coalgebra isomorphism, it follows by induction on $n$, that the restriction of $\psi_k$ to $\Sym^n\p(S,\vv)$ is multiplication by $k^n$ as the reduced diagonal
$$
\Sym^n\p(S,\vv) \to \sum_{\substack{a+b=n\cr a,b>0}} \Sym^a\p(S,\vv)\otimes \Sym^b\p(S,\vv)
$$
is injective when $n>1$. The second statement follows as the projection $\Q\pi_1(S,\vv)^\wedge \to \Q\lambda(S)^\wedge$ commutes with $\psi_k$.
\end{proof}

As observed in \cite[\S8.3]{hain:goldman}, a direct consequence is that $\Q\lambda(S)^\wedge$ decomposes as the completed product of the $|\Sym^n\p(S,\vv)|$.

\begin{corollary}
\label{cor:pbw}
The PBW isomorphism (\ref{eqn:pbw}) descends to a direct product decomposition
\begin{equation}
\label{eqn:adams_decomp}
\Q\lambda(S)^\wedge = |\Q\pi_1(S,\vv)^\wedge| \cong \prod_{n\ge 0} |\Sym^n \p(S,\vv)|.
\end{equation}
This decomposition does not depend on $\vv$. \qed
\end{corollary}

There is a canonical isomorphism $|\p(S,\vv)| \cong H_1(S;\Q)$ as $|[u,v]| = 0$ for all $u,v\in \p(S,\vv)$. When $(S,\vv)$ is of type $(g,\uu)$ we can say more. The proof is sketched in subsection~\ref{sec:proof}.

\begin{proposition}
ls p\label{prop:S2}
If $g\ge 1$ and $(S,\vv)$ is a surface of type $(g,\uu)$, then the subspace $|\Sym^2\p(S,\vv)|$ of $\Q\lambda(S)^\wedge$ is a Lie subalgebra and $\kappa_\vv$ induces a Lie algebra isomorphism $|\Sym^2\p(S,\vv)| \to \Der^\theta\p(S,\vv)$.
\end{proposition}

\subsection{Relative unipotent completion}

This is a very brief review of relative unipotent completion of discrete groups, especially of mapping class groups. A more detailed exposition can be found in \cite[\S3--4]{hain:morita}, and complete results in \cite[\S3--5]{hain:torelli}.

Suppose that:
\begin{enumerate}

\item $\G$ is a discrete group

\item $\kk$ is a field of characteristic zero,

\item $R$ is a reductive group over $\kk$,

\item $\rho : \G \to R(\kk)$ is a Zariski dense homomorphism.

\end{enumerate}

The {\em completion of $\G$ relative to $\rho$} (or the relative completion of $\G$) consists of an affine group\footnote{Equivalently, a proalgebraic group.} $\cG$ over $\kk$ and a homomorphism $\rhotilde : \G \to \cG(\kk)$. The group $\cG$ is an extension of $R$ by a prounipotent group:
$$
1 \to \U \to \cG \to R \to 1
$$
where the composition $\G \to \cG(\kk) \to R(\kk)$ is $\rho$. These have the property that if
$$
1 \to U \to G \to R \to 1
$$
is an extension of affine groups over $\kk$, where $U$ is prounipotent, and if $\phi : \G \to G(\kk)$ is a homomorphism through which $\rho$ factors $\G \to G(\kk) \to R(\kk)$, then there is a {\em unique} homomorphism of affine groups $\cG \to G$ that commutes with projections to $R$ such that $\phi$ is the composition $\G \to \cG(\kk) \to G(\kk)$. That is, $(\cG,\rhotilde)$ is an initial object of a category whose objects are pairs $(G,\phi)$ and whose morphisms are appropriately defined.

\begin{remark}
Several comments are in order:
\begin{enumerate}

\item One can also define relative completion using tannakian categories, as explained in \cite[\S10.1]{hain-matsumoto:mem}.

\item The homomorphism $\rhotilde: \G \to \cG(\kk)$ is Zariski dense.

\item The prounipotent group $\U$ is determined by its Lie algebra.

\item When $R$ is trivial, relative completion reduces to unipotent completion. This is the correct way to define unipotent completion for discrete groups $\pi$ with infinite dimensional $H_1(\pi;\Q)$. This is relevant as genus 2 Torelli groups are countably generated  \cite{mess} and have infinite dimensional abelianization.

\end{enumerate}
\end{remark}

\subsection{Relative completion of mapping class groups}
\label{sec:rel_comp_mcg}

In this section, we take $\kk=\Q$. The group $\Sp(H)$ is a reductive $\Q$-group. The completion of $\G_{g,n+\vr}$ relative to the standard homomorphism $\rho : \G_{g,n+\vr} \to \Sp(H_\Q)$ will be denoted by $\cG_{g,n+\vr}$. Denote its prounipotent radical by $\U_{g,n+\vr}$ and the  Lie algebras of $\cG_{g,n+\vr}$ and $\U_{g,n+\vr}$ by $\g_{g,n+\vr}$ and $\u_{g,n+\vr}$. With the help of Hodge theory, we will give presentations of various $\g_{g,n+\vr}$ in Section~\ref{sec:hodge_mcg}.

Relative completion is just unipotent completion in genus 0 as $H=0$. The Lie algebras $\g_{0,n+\vr}$ are well understood. We recall their presentations in Section~\ref{sec:genus0}.

\begin{theorem}
For all $g\ge 2$, $H_1(\u_{g,n+\vr})$ is finite dimensional, so that $\u_{g,n+\vr}$ is finitely (topologically) generated. When $g=1$, $H_1(\u_{1,n+\vr})$ is infinite dimensional, so that $\u_{1,n+\vr}$ is not finitely generated.
\end{theorem}

The genus 1 case is important as it is closely related to classical modular forms. It will be discussed in more detail in Section~\ref{sec:genus1}.

Since the Torelli group $T_{g,n+\vr}$ is the kernel of $\rho$, its image in $\cG_{g,n+\vr}(\Q)$ lies in $\U_{g,n+\vr}(\Q)$. Since this is prounipotent, the restriction of $\rho$ to the Torelli group factors through relative completion:
$$
T_{g,n+\vr} \to T_{g,n+\vr}^\un(\Q) \to \U_{g,n+\vr}(\Q).
$$
Denote the Lie algebra of $T_{g,n+\vr}^\un$ by $\t_{g,n+\vr}$. \label{def:Lie_torelli}

\begin{theorem}
\label{thm:lie_torelli}
If $g\ge 2$, the homomorphism $T_{g,n+\vr}^\un \to \U_{g,n+\vr}$ is surjective.\footnote{That is, it is injective on coordinate rings. For prounipotent groups in characteristic zero (our case), this is equivalent to the induced homomorphism of Lie algebras being surjective.} When $g\ge 3$, its kernel is a copy of the additive group $\Ga$ contained in the center of $T_{g,n+\vr}^\un$, so that there is a non-trivial central extension
$$
0 \to \Q \to \t_{g,n+\vr} \to \u_{g,n+\vr} \to 0.
$$
of pronilpotent Lie algebras. When $g=2$, the kernel of $T_{g,n+\vr}^\un \to \U_{g,n+\vr}$ is infinitely generated and free when $r=n=0$.
\end{theorem}

\begin{proof}[Remarks on the proof]
I understand this much better than when I first published a proof of the genus $g\ge 3$ case in \cite{hain:completions}. In the interests of clarity, and because of its relevance to the study of Johnson homomorphisms, I will give a brief sketch of the argument.

The first step is to note that, when $g\ge 2$, standard cohomology vanishing theorems imply that the completion of $\Sp(H_\Z)$ with its inclusion into $\Sp(H_\Q)$ is just $\Sp(H)$. (See Example 3.2 in \cite{hain:morita}.) The second is to use the fact that relative completion is right exact \cite[Prop.~3.7]{hain:morita} to see the sequence
$$
T_{g,n+\vr}^\un \to \cG_{g,n+\vr} \to \Sp(H) \to 1
$$
is exact. This implies that $T_{g,n+\vr}^\un \to \U_{g,n+\vr}$ is surjective when $g\ge 2$ as every group is Zariski dense in its unipotent completion.

Now suppose that $g\ge 3$. The next ingredient is to use \cite[Prop.~4.13]{hain:completions} to see that the kernel of this map is central and there is an exact sequence
\begin{equation}
\label{eqn:completion}
H_2(\Sp(H_\Z;\Q)) \to T_{g,n+\vr}^\un \to \U_{g,n+\vr} \to 1.
\end{equation}
The proof of this requires that $H_1(T_{g,n+\vr};\Q)$ be finite dimensional. This holds when $g\ge 3$ by Johnson's computation of $H_1(T_{g,\uu})$ in \cite{johnson:h1} and fails in genus 2 in view of Mess's computation \cite{mess}. As explained in \cite[\S8]{hain:completions}, the injectivity on the left hand end is equivalent to the non-triviality of the biextension line bundle over $\M_g$ associated to the Ceresa cycle $C-C^-$ in the jacobian of a smooth projective curve $C$ of genus $g$. This follows from a computation of Morita in \cite[5.8]{morita:chern}. Another proof closer to algebraic geometry can be found in \cite[Thm.~7]{hain-reed:morita}. The final step is to note that $H_2(\Sp(H_\Z);\Q)$ is 1-dimensional for all $g\ge 3$. This was communicated to me by Borel and stated without proof in \cite[Thm.~3.2]{hain:torelli}. A proof may be found in \cite{borel_stab}.
\end{proof}

The sequence (\ref{eqn:completion}) is not exact in genus 2 as $H_1(T_2)$ has infinite rank. Mess \cite{mess} proved that $T_2$ is a countably generated free group and $H_1(T_2)$ is the free $\Z$-module generated by the homology decompositions $H_\Z = A\oplus B$, where the restriction of the intersection pairing to $A$ and $B$ are unimodular. This implies that $H_1(T_{2,n+\vr})$ has infinite rank of all $n$ and $r$. On the other hand, Watanabe \cite{watanabe}, using work of Petersen \cite{petersen}, proved that $\u_2$ is generated by the irreducible $\Sp(H)$ representation $V_\boxplus$ that corresponds to the partition $[2^2]$, the highest weight submodule of $\Sym^2\Lambda^2 H$. This implies that
$$
H_1(\u_{2,n+\vr}) \cong V_\boxplus \oplus H_\Q^{n+r}
$$
as $\Sp(H)$-modules.

\subsection{The geometric Johnson homomorphism}

Suppose that $(\Sbar,P,\vV)$ is a surface of type $(g,n+\vr)$ and that $\vv \in \vV$. Denote the Lie algebra of the unipotent completion of $\pi_1(S,\vv)$ by $\p(S,\vv)$. Since unipotent completion is a functor, the action of $\G_{g,n+\uu}$ on $\pi_1(S,\vv)$ induces a homomorphism
$$
\phi_\vv : \G_{g,n+\uu} \to \Aut^\theta \p(S,\vv),
$$
where
$$
\Aut^\theta \p(S,\vv) := \{h \in \Aut\p(S,x) : h(\log\sigma_\vv) = \log\sigma_\vv\}
$$
and $\log\sigma_\vv \in \p(S,\vv)$ is the logarithm of the class of a small loop $\sigma_\vv$ based at $\vv$ that encircles the anchor point of $\vv$. It has Lie algebra
$$
\Der^\theta \p(S,\vv) := \{D\in \Der\p(S,x) : D(\log \sigma_\vv) = 0\}.
$$
The universal mapping property of relative completion implies that this factors uniquely through a homomorphism
$$
\varphi_\vv : \cG_{g,n+\uu} \to \Aut^\theta \p(S,\vv)
$$
of affine $\Q$-groups. This induces a Lie algebra homomorphism
\begin{equation}
\label{eqn:geom_J}
d\varphi_\vv : \g_{g,n+\vr} \to \Der^\theta \p(S,\vv).
\end{equation}
This is the {\em geometric Johnson homomorphism}. It factors through $\g_{g,n+\vr} \to \g_{g,n+r-1+\uu}$ as the action of $\G_{g,n+\vr}$ on $\pi_1(S,\vv)$ factors through $\G_{g,n+\vr} \to \G_{g,n+r-1+\uu}$.

\begin{question}
\label{quest:johnson_inj}
Suppose that $g\neq 1$ and that $(S,\vv)$ is a surface of type $(g,\uu)$. Is the geometric Johnson homomorphism $\g_{g,\uu} \to \Der^\theta \p(S,\vv)$ injective?
\end{question}

The geometric Johnson homomorphism in not injective when $g=1$ as it factors through the homomorphism $\g_{1,\uu} \to \g^\MEM_{1,\uu}$, where $\g^\MEM_{1,\uu}$ is the Lie algebra $\g^\MEM_\uu$ defined in \cite[\S1]{hain-matsumoto:mem}. This has a very large kernel. In this case the correct question appears to be whether $\g^\MEM_{1,\uu} \to \Der^\theta \p(S,\vv)$ is injective. This is not known. However, it is known by \cite{hain:torelli,morita-sakasai-suzuki} to be injective in ``weights $\ge -6$''. This will be explained further in Part~\ref{part:stab_coho}.

The geometric Johnson homomorphism factors through the Kawazumi--Kuno action. This is essentially a result of Kawazumi and Kuno \cite[Thm.~5.2.1]{kk:groupoid}. The version below is formulated and proved in \cite[\S13]{hain:goldman}.

\begin{theorem}
\label{thm:fund_diag}
There is a Lie algebra homomorphism
$$
\phitilde : \g_{g,n+\vr} \to \Q\lambda(S)^\wedge
$$
that depends only on $(\Sbar,P\cup Q)$, such that for each $\vv\in \vV$ the geometric Johnson homomorphism (\ref{eqn:geom_J}) factors
$$
\xymatrix{
\g_{g,n+\vr} \ar[r]^(.45){d\varphi_\vv}\ar[d]_\phitilde & \Der^\theta \p(S,\vv)\ar@{^{(}->}[d] \cr 
\Q\lambda(S) ^\wedge\ar@{->>}[r]^(.35){\kappahat_\vv} & \Der^\theta \Q\pi_1(S,\vv)^\wedge
}.
$$
\end{theorem}

The main ingredient in the proof is the non-abelian generalization of the classical Picard--Lefschetz formula due to Kawazumi and Kuno \cite[Thm.~5.2.1]{kk:groupoid}. Because of its beauty and importance, we state and sketch a proof of it in the next section.

Kawazumi and Kuno \cite{kk:intersections} observed that this provides an upper bound on the image of the Johnson homomorphism. This will be explained in Section~\ref{sec:johnson}.

\subsection{The non-abelian Picard--Lefschetz formula}

Suppose that $(\Sbar,P,\vV)$ is a surface of type $(g,n+\vr)$. For each $\alpha\in \lambda(S)$, $1-\alpha$ is in $I\Q\lambda(S)$, which implies that for all $n\ge 1$,
$$
(1-\alpha)^n := \sum_{j=0}^n (-1)^j\binom{n}{j}\psi_j(\alpha) \in I^n\Q\lambda(S).
$$
This implies that any power series in $1-\alpha$ converges in $\Q\lambda(S)^\wedge$. In particular, we can define
$$
\log\alpha := - \sum_{n=1}^\infty \frac{1}{n}(1-\alpha)^n \in \Q\lambda(S)^\wedge
$$
and its powers, such as
$$
\big(\log\alpha\big)^2 :=  \sum_{n=1}^\infty \sum_{\substack{{j+k=n}\cr j,k>0}}\frac{1}{jk}(1-\alpha)^n \in \Q\lambda(S)^\wedge.
$$
Since $\log \alpha$ lifts naturally to $\Q\pi_1(S,\vv)^\wedge$, and since this lift lies in $\p(S,\vv)$, we see that $(\log \alpha)^n \in |\Sym^n\p(S,\vv)|$.

When $\alpha$ is a simple closed curve $(\log\alpha)^2$ has a concrete meaning. Denote by $\bt_\alpha \in \G_{g,n+\vr}$ the (positive) Dehn twist  about $\alpha$. It does not depend on the orientation of $\alpha$. Since $\bt_\alpha$ acts unipotently on $H_1(S)$, its image in $\cG_{g,n+\vr}$ lies in a prounipotent subgroup. It therefore has a canonical logarithm $\log \bt_\alpha \in \g_{g,n+\vr}$.

\begin{theorem}[Kawazumi--Kuno {\cite[Thm.~5.2.1]{kk:groupoid}}]
\label{thm:kk-pl}
For each simple closed curve $\alpha\in \lambda(S)$ and each $\vv \in \vV$, we have
$$
\frac{1}{2}\kappahat_\vv\big((\log\alpha)^2\big) = d\varphi_\vv(\log \bt_\alpha) \in \Der^\theta \Q\pi_1(S,\vv)^\wedge.
$$
\end{theorem}

For a once punctured surface $S$, this reduces to the standard Picard--Lefschetz formula $u \mapsto u + \langle \alpha, u \rangle \alpha$ for the action of $\bt_\alpha$ on $H_1(S)$.

\begin{proof}[Sketch of Proof]
The first step in the proof is to reduce to the case where $S$ is an annulus $[0,1]\times S^1$, $\alpha$ is its core $\{1/2\}\times S^1$, and $u$ is represented by a path that is a fiber of the projection of the annulus onto $\alpha$. This follows from the naturality of the Kawazumi--Kuno action under inclusion of surfaces and the derivation property (\ref{eqn:kk-derivation}). More precisely, we replace $S$ by a regular neighbourhood $N$ of $\alpha$ and $u$ by the path $\gamma$ defined by $\gamma(s) = (s,1)$.

Denote the Kawazumi--Kuno action on the paths in $N$ from $x_0 := (0,1)$ to $x_1:=(1,1)$ by $\kappa$. The definition of $\kappa$ implies that
\begin{equation}
\label{eqn:kappa}
\kappa(\alpha^n)(\gamma) =  n \bt_\alpha^n(\gamma).
\end{equation}

Every power series
$$
\Psi(u) = \sum_{n=0}^\infty a_n (u-1)^n \in \Q[[u-1]]
$$
in $u-1$ can be evaluated on $\alpha$ to obtain an element
$
\Psi(\alpha) \in \Q\lambda(N)^\wedge
$
and on $\bt_\alpha$ to obtain an element
$$
\Psi(\bt_\alpha) \in \End\Q\pi(N;x_0,x_1)^\wedge.
$$
The formula (\ref{eqn:kappa}) implies that
$$
\kappa\big(\Psi(\alpha)\big)(\gamma) = \bt_\alpha \Psi'(\bt_\alpha)(\gamma).
$$
The result follows by taking
$$
\Psi(u) = \frac{(\log u)^2}{2} := \bigg(\sum_{n=1}^\infty (-1)^{n+1} \frac{(u-1)^n}{n}\bigg)^2
$$
as it has logarithmic derivative $u\Psi'(u) = \log u$.
\end{proof}

\subsection{Centralizers}

For technical reasons, it is important to understand the centralizer of a Dehn twist in the relative completion $\cG_S$ of the mapping class group $\G_S$ of a surface $S$. Since relative completion is not, in general, left exact, it is not clear that the centralizer of a Dehn twist in $\G_S$ is Zariski dense in the stabilizer of the Dehn twist in $\cG_S$.

Suppose that $A$ is a simple closed curve in $S$. We allow $A$ to be a boundary component. Let $T$ be the surface obtained by cutting $S$ open along $A$. Assume that each connected component of $S$ with $A$ removed is hyperbolic. Denote the Dehn twist on $A$ by $t_A$. Then it is well known \cite[Fact~3.8]{farb-margalit} that the centralizer $Z(t_A)$ of $t_A$ in $\G_{S,\partial S}$ is
$$
Z(t_A) = \{\phi \in \G_{S,\partial S} : \phi(A) = A\}.
$$
It is an extension by a finite group of the quotient of $\G_{T,\partial T}$ obtained by identifying the two Dehn twists on the two boundary components of $T$ that are identified to obtain $S$. One can ask whether this holds for relative completion:

\begin{question}
\label{quest:centralizer}
Is the centralizer $Z(t_A)$ of $t_A$ in $\G_S$ Zariski dense in the centralizer of the image of $t_A$ in $\cG_S$? In particular, is the center of $\g_{g,n+\vr}$ spanned by the logarithms of the Dehn twists on the boundary components?
\end{question}

An affirmative answer to the first question would greatly simplify Takao's proof \cite{takao} of the Oda Conjecture. See Theorem~\ref{thm:arith_extn}.

\section{Hodge Theory}
\label{sec:hodge}

Hodge theory provides two important tools for studying the Johnson homomorphism. The first is that, once one has established that a map between topological invariants (such as the Johnson homomorphism, the Goldman bracket, the Turaev cobracket, \dots) is a morphism of mixed Hodge structure (MHS), one can replace it by the induced map on the associated weight graded objects without loss of (topological) information. The second is that, since the category of (graded polarizable) $\Q$ mixed Hodge structures is tannakian, every MHS is a module over an affine group that we shall denote by $\pi_1(\MHS)$, and every morphism of MHS is $\pi_1(\MHS)$-equivariant. This provides a large hidden group of symmetries that acts on all invariants once once has chosen a complex structure on $(\Sbar,P,\vV,\xi)$.

\subsection{Summary of Hodge theory}

We will assume the reader is familiar with the basics of mixed Hodge theory. Minimal background suitable for this discussion, as well as further references on Hodge theory, can be found in \cite[\S10]{hain:goldman}.

We will consider (and need) only $\Q$-MHS. A $\Q$-MHS $V$ consists of a finite dimensional $\Q$ vector space $V_\Q$ endowed with an increasing {\em weight filtration}
$$
0 = W_n V_\Q \subseteq \dots \subseteq W_{j-1}V_\Q \subseteq W_j V_\Q \subseteq \dots \subseteq W_N V_\Q = V_\Q
$$
and a decreasing {\em Hodge filtration}
$$
V_\C = F^a V_\C \supseteq \dots \supseteq F^p V_\C \supseteq F^{p+1} V_\C \supseteq \dots \supseteq F^b V_\C = 0
$$
of its complexification $V_\C$. These are required to satisfy the condition that, for all $m\in\Z$, the $m$th weight graded quotient $\Gr^W_m V$ of $V$, whose underlying $\Q$ vector space is
$$
\Gr^W_m V := W_m V_\Q/W_{m-1}V_\Q
$$
is a Hodge structure of weight $m$. This means simply that
$$
\Gr^W_m V_\C = \bigoplus_{p+q=m} (F^p \Gr^W_m V_\C) \cap (\Fbar^q \Gr^W_m V_\C)
$$
where $\Fbar^q V_\C$ is the conjugate of $F^q V_\C$ under the action of complex conjugation on $V_\C$. A MHS $V$ is said to be {\em pure} of weight $m\in \Z$ if $\Gr^W_r V = 0$ when $r\neq m$. A {\em Hodge structure} of weight $m$ is simply a MHS that is pure of weight $m$.

Morphisms of MHS are weight filtration preserving $\Q$-linear maps of the underlying $\Q$-vector spaces which induce Hodge filtration preserving maps after tensoring with $\C$. The category of MHS is a $\Q$-linear abelian tensor category. This is not obvious, as the category of filtered vector spaces is not an abelian category.

The key property for us is that, for each $m\in \Z$, the functor $\Gr^W_m$ from the category of MHS to $\Vec_\Q$, the category of $\Q$ vector spaces, is exact. In particular, this implies that if $\phi : V \to V'$ is a morphism of MHS, then there are natural isomorphisms
\begin{equation}
\label{eqn:exactness}
\Gr^W_\bdot \ker \phi \cong \ker \Gr^W_\bdot \phi
\text{ and }
Gr^W_\bdot \im \phi \cong \im \Gr^W_\bdot \phi.
\end{equation}

Deligne \cite{deligne:h2,deligne:h3} proved that the cohomology of every complex algebraic variety has a natural MHS that is functorial with respect to morphisms of varieties. This was extended to homotopy invariants by Morgan \cite{morgan} and the author in \cite{hain:dht}.

\begin{example}
\label{ex:tate_twist}
The Hodge structure $\Q(n)$ is a pure Hodge structure of weight $-2n$ and has type $(-n,-n)$. Equivalently, $F^{-n}\Q(n)_\C = \Q(n)_\C$ and $F^{-n+1} = 0$. One can also define $\Q(1) = H_1(\C^\ast;\Q)$ and $\Q(n) = \Q(1)^{\otimes n}$ for all $n\ge 0$. When $n<0$, one defines $\Q(n)$ to be the dual of $\Q(-n)$. Mixed Hodge structures whose weight graded quotients are direct sums of $\Q(n)$'s are called Tate MHSs.
\end{example}

The tensor product of a MHS $V$ with $\Q(n)$ is denoted by $V(n)$ and called a {\em Tate twist} of $V$.

\begin{example}
If $X$ is a compact Riemann surface of genus $g>0$, then $H_\Q := H_1(X;\Q)$ is pure of weight $-1$. The intersection pairing
$$
\langle \phantom{x},\phantom{x} \rangle : H_1(X)\otimes H_1(X) \to \Q(1)
$$
is a morphism of MHS. Poincar\'e duality is an isomorphism of MHS
$$
H_1(X) \to H^1(X)\otimes \Q(1).
$$
\end{example}

\subsection{The category $\MHS$}

In order to have a good category of mixed Hodge structures, we need to introduce the notion of a polarized Hodge structure. Polarizations generalize the Riemann bilinear relation
$$
i\int_X \w \wedge \overline{\w} \ge 0 \text{ with equality if and only if }\w = 0
$$
that is satisfied by holomorphic 1-forms $\w$ on a compact Riemann surface $X$.

In general, a {\em polarization} of a Hodge structure $V$ of weight $m$ is a $(-1)^m$-symmetric bilinear form
$$
Q : V_\Q \otimes V_\Q \to \Q
$$
that satisfies the {\em Riemann--Hodge bilinear relations}
\begin{enumerate}

\item the restriction of $Q$ to $V^{p,m-p}\otimes V^{s,m-s}$ vanishes except when $s=m-p$, where $V^{p,m-p} := F^p V\cap F^{m-p}V$;

\item The hermitian form
$$
v \otimes \bar{w} \mapsto i^{p-q} Q(v,\bar{w}),\qquad v,w \in V^{p,q},\ p+q=m
$$
is positive definite on each $V^{p,q}$.
\end{enumerate}
Polarizations $Q$ are non-degenerate.

A {\em polarized Hodge structure} is a Hodge structure endowed with a polarization. From our point of view, the significance of polarizations is that if $A$ is a Hodge substructure of a polarized Hodge structure $V$, then the restriction of the polarization to $A$ is non-degenerate (and therefore a polarization) and
$$
V = A \oplus A^\perp
$$
as MHS. So polarized Hodge structures are semi-simple objects in the category of mixed Hodge structures.

More generally, a MHS $V$ is {\em graded polarizable} if each of its weight graded quotients $\Gr^W_\bdot V$ admits a polarization. All mixed Hodge structures that arise in geometry are graded polarizable, as are all MHS that occur in this paper. Denote the category of {\em graded polarizable} $\Q$-MHS by $\MHS$.

\subsection{Splittings}
\label{sec:hodge_splittings}

The category $\MHS$ is a $\Q$-linear neutral tannakian category with fiber functor $\MHS \to \Vec_\Q$ that takes an MHS $V$ to its underlying rational vector space $V_\Q$. There is therefore an affine $\Q$-group
$$
\label{def:pi_1(mhs)}
\pi_1(\MHS) := \pi_1(\MHS,\w) := \Aut^\otimes \w
$$
such that (a) every mixed Hodge structure is naturally a $\pi_1(\MHS)$-module and (b) the category of finite dimensional representations of $\pi_1(\MHS)$ is equivalent to $\MHS$, \cite{deligne:h2,deligne:tannakian}.

The group $\pi_1(\MHS)$ is an extension
$$
\xymatrix{
1 \ar[r] & \U^\MHS \ar[r] & \pi_1(\MHS) \ar[r] & \pi_1(\MHS^\ss) \ar[r] & 1 
}
$$
where $\U^\MHS$ is prounipotent and $\MHS^\ss$ denotes the full sub-category of $\MHS$ whose objects are direct sums of graded polarizable Hodge structures (the semi-simple objects of $\MHS$). There is a canonical cocharacter
$$
\label{def:chi}
\chi : \Gm \to \pi_1(\MHS^\ss)
$$
The copy of $\Gm$ acts on a Hodge structure $V$ of weight $j$ by $\chi(t) : v \mapsto t^j v$. Levi's theorem implies that $\chi$ lifts to a cocharacter $\chitilde : \Gm \to \pi_1(\MHS)$ and that all such lifts
$$
\chitilde : \Gm \to \pi_1(\MHS)
$$
form a torsor (principal homogeneous space) under $\U^\MHS$, which acts on them by conjugation.

Since every $V$ in $\MHS$ is a $\pi_1(\MHS)$-module, each choice of lift $\chitilde$ makes $V$ into a $\Gm$-module. This determines a natural (not canonical) splitting
$$
V_\Q \cong \Gr^W_\bdot V_\Q := \bigoplus_{m\in \Z} \Gr^W_m V_\Q
$$
of the weight filtration of each MHS $V$ that is preserved by morphisms of MHS. A more detailed explanation can be found in \cite[\S10]{hain:goldman}.

Finally, the category of {\em all} representations of $\pi_1(\MHS)$ is equivalent to the category ind-$\MHS$. Many objects we consider will be pro-objects of $\MHS$ and their continuous duals will be ind-objects.

\section{Hodge Theory and Surface Topology}
\label{sec:mcg_hodge}

Here we recall basic results about the existence of MHS on invariants of surface groups and their mapping class groups. These MHS require the choice of a complex structure $\phi : (\Sbar,P,\vV) \to (\Xbar,Y,\vV')$ on $(\Sbar,P,\vV)$ (as defined in Section~\ref{sec:topology}) and depend non-trivially on it. The main results are proved in \cite{hain:dht}, \cite{hain:torelli}, \cite{hain:goldman}, \cite{hain:turaev}.

\subsection{Surface groups}

Each choice of a complex structure $\phi : (\Sbar,P,\vV) \to (\Xbar,Y,\vV')$ on $(\Sbar,P,\vV)$ determines a canonical pro-MHS on
$$
\Q\pi_1(S,\vv)^\wedge \text{ and } \p(S,\vv).
$$
The pro-MHS on $\Q\pi_1(S,\vv)$ descends to a pro-MHS on $\Q\lambda(S)^\wedge$. These MHS are preserved by the product and coproduct of $\Q\pi_1(S,\vv)^\wedge$ and by the Lie bracket of $\p(S,\vv)$.

When $S$ is a surface of type $(g,n+\vr)$, where $n+\vr \le 1$, then $H_1(S)$ is pure of weight $-1$ and the weight filtrations are given by
$$
W_{-m}\Q\pi_1(S,\vv)^\wedge = I^m\Q\pi_1(S,\vv)^\wedge \text{ and } W_{-m}\p(S,\vv) = L^m\p(S,\vv),
$$
where $L^m$ denotes the $m$th term of the lower central series and $I^m$ is the $m$th power of the augmentation ideal. When $S$ is hyperbolic, the MHS on each of these invariants depends non-trivially on the complex structure on $(S,\vv)$. In addition
$$
W_{-m}\Q\lambda(S)^\wedge = I^m\Q\lambda(S)^\wedge.
$$
All of these MHS are functorial with respect to holomorphic maps.

\subsection{Mapping class groups}
\label{sec:hodge_mcg}

Each choice of a complex structure $\phi$ on $(\Sbar,P,\vV)$ determines a canonical pro-MHS on the Lie algebra $\g_{g,n+\vr}$ of the relative completion of the mapping class group $\G_{g,n+\vr} \cong \pi_1(\M_{g,n+\vr},\phi)$. It depends non-trivially on $\phi$ as, for example, the Hodge structures on its weight graded quotients are are subquotients of the Hodge structures on $H_1(S)$ determined by $\phi$. The weight filtration satisfies.
\begin{equation}
\label{eqn:w_filt}
W_0 \g_{g,n+\vr} = \g_{g,n+\vr},\ W_{-1}\g_{g,n+\vr} = \u_{g,n+\vr},\ \Gr^W_0 \g_{g,n+\vr} \cong \sp(H).
\end{equation}
The Lie bracket is a morphism of MHS.

By \cite[Lem.~4.5]{hain:torelli}, for each complex structure $\phi$, and for each $\vv \in \vV$, the geometric Johnson homomorphism
$$
\g_{g,n+\vr} \to \Der^\theta \p(S,\vv)
$$
is a morphism of MHS.

When $g\ge 3$, the weight filtration of $\u_{g,n+\vr}$ is its lower central series
$$
W_{-m}\u_{g,n+\vr} = L^m \u_{g,n+\vr}.
$$
This holds for all non-negative $n$ and $r$. This is not true when $g\le 2$.

When $g\ge 3$, there is a natural MHS on $\t_{g,n+\vr}$ such that the surjection $\t_{g,n+\vr} \to \u_{g,n+\vr}$ is a morphism of MHS with kernel $\Q(1)$. The weight filtration of $\t_{g,n+\vr}$ is its lower central series:
$$
W_{-m} \t_{g,n+\vr} = L^m \t_{g,n+\vr}
$$
so that $\Gr^W_{-m} \t_{g,n+\vr} \cong \big(\Gr_L^m T_{g,n+\vr}\big)\otimes \Q$.

\subsection{The Goldman--Turaev Lie bialgebra}
\label{sec:gt_hodge}


For each choice of a complex structure $\phi$ on $(\Sbar,P,\vV)$, the Goldman bracket
$$
\gold : \Q\lambda(S)^\wedge \otimes \Q\lambda(S)^\wedge \to \Q\lambda(S)^\wedge \otimes\Q(1)
$$
twisted by $\Q(1)$ is a morphism of MHS, and, for each $\vv \in \vV$ and the (completed) Kawazumi--Kuno action
\begin{equation}
\label{eqn:kk-action}
\kappa_\vv : \Q\lambda(S)^\wedge\otimes \Q(-1) \to \Der^\theta \Q\pi_1(S,\vv)^\wedge
\end{equation}
is also a morphism of MHS. These are the main results of \cite{hain:goldman}.

In order for the cobracket of the framed surface $(S,\xi)$ to be a morphism of MHS, we have to choose a complex structure on $S$ in which $\xi$ is {\em quasi-algebraic}. By this we mean a complex structure
$$
\phi : (\Sbar,P,\vV) \to (\Xbar,Y,\vV')
$$
on $S$ for which $\phi_\ast \xi$ is homotopic to a meromorphic section of $T\Xbar\otimes L$, the holomorphic tangent bundle of $\Xbar$ twisted by a torsion line bundle $L$ over $\Xbar$. The framing $\xi$ is {\em algebraic} when $L$ is the trivial bundle $\cO_\Xbar$. A precise definition of quasi-algebraic framings is given in \cite[Def.~7.1]{hain:turaev}.

When $\xi$ is quasi-algebraic, the twisted cobracket
$$
\delta_\xi : \Q\lambda(S)^\wedge\otimes \Q(-1) \to \Q\lambda(S)^\wedge \comptensor \Q\lambda(S)^\wedge
$$
is a morphism of MHS. This is the main result of \cite{hain:turaev}.

Provided that $n+r>0$, there is always a complex structure on $(S,P,\vV)$ that admits an algebraic framing. For example, for surfaces of type $(g,1)$ case with $g\ge 1$, one has the framing $y\partial/\partial x$ of the genus $g$ affine curve
$$
y^2 = \prod_{j=0}^{2g} (x-a_j).
$$
When $g\neq 1$, every topological framing $\xi$ is homotopic to a quasi-algebraic framing. See \cite[\S9]{hain:turaev}.

Finally, observe that since $\p(S,\vv)$ is a sub-MHS of $\Q\pi_1(S,\vv)^\wedge$, each factor $|\Sym^n\p(S,\vv)|$ in the PBW decomposition (\ref{eqn:adams_decomp}) of $\Q\lambda(S)^\wedge$ is a sub-MHS. This implies that (\ref{eqn:adams_decomp}) is also a decomposition of pro-MHS.

\subsection{Splittings}
\label{sec:splittings}

As explained in Section~\ref{sec:hodge_splittings}, each choice of a lift $\chitilde: \Gm \to \pi_1(\MHS)$ of the central cocharacter $\chi : \Gm \to \pi_1(\MHS^\ss)$ determines, for every $V$ in $\MHS$, an isomorphism
$$
V_\Q \overset{\simeq}{\To} \Gr^W_\bdot V_\Q = \bigoplus_{m\in \Z} \Gr^W_m V_\Q 
$$
of the rational vector space underlying $V$ with its associated graded. These isomorphisms commute with maps induced by morphisms of MHS and are compatible with tensor products and duals. This extends verbatim to ind-objects of $\MHS$. For pro-objects of $\MHS$, each lift $\chitilde$ determines isomorphisms
$$
V_\Q \overset{\simeq}{\To} \big(\Gr^W_\bdot V_\Q\big)^\wedge := \prod_{m\in \Z} \Gr^W_m V_\Q
$$
which are natural for morphisms of pro-MHS. In all cases, the effect of a twist by $\Q(r)$ is to shift this grading by $2r$:
$$
\Gr^W_m V(r) = \Gr^W_{m+2r} V.
$$

In particular, for each complex structure $\phi$ on $(\Sbar,P,\vV)$, the choice of a lift $\chitilde$ of the central cocharacter $\chi$, determines natural Lie algebra isomorphisms
\begin{align}
\Q\lambda(S)^\wedge &\cong \big(\Gr^W_\bdot \Q\lambda(S)\big)^\wedge := \prod_{m\ge 0} \Gr^W_{-m} \Q\lambda(S)^\wedge
\cr
\p(S,\vv) &\cong \big(\Gr^W_\bdot \p(S,\vv)\big)^\wedge := \prod_{m\ge 0} \Gr^W_{-m} \p(S,\vv)
\cr
\g_{g,n+\vr} &\cong \big(\Gr^W_\bdot \g_{g,n+\vr}\big)^\wedge := \prod_{m\ge 0} \Gr^W_{-m} \g_{g,n+\vr}
\cr
\Der^\theta\p(S,\vv) &\cong \prod_{m\ge 0} \Gr^W_{-m} \Der ^\theta \p(S,\vv) =: \big(\Der^\theta \Gr^W_\bdot\p(S,\vv)\big)^\wedge
\cr
\Der^\theta\Q\pi_1(S,\vv)^\wedge &\cong \prod_{m\ge 0} \Gr^W_{-m}\Der ^\theta \Q\pi_1(S,\vv)^\wedge =: \big(\Der^\theta \Gr^W_\bdot\Q\pi_1(S,\vv)\big)^\wedge
\end{align}
In the last 2 cases $S$ is of type $(g,\uu)$ and
$$
\Der ^\theta\Gr^W_{-m}\p(S,\vv) = \{D\in \Der\Gr^W_\bdot\p(S,\vv) : D(\theta) = 0\},
$$
where $\theta = \sum [a_j,b_j]$ and $\{a_1,\dots,a_g,b_1,\dots,b_g\}$ is a symplectic basis of $\Gr^W_{-1}\p(S,\vv) = H$. Similarly for $\Der^\theta\Q\pi_1(S,\vv)^\wedge$.

When $S$ is of type $(g,\uu)$, there is a {\em canonical} graded Lie algebra isomorphism
$$
\Gr^W_\bdot \p(S,\vv) \cong \L(H)
$$
of the associated graded of $\p(S,\vv)$ with the free Lie algebra generated by $H$.

\begin{remark}
\label{rem:graphical_rep}
There are well-known and psychologically helpful graphical representations of elements of $\L(H)$ and $\Der^\theta \L(H)$. Elements of $\L(H)$ can be represented by linear combinations of rooted, planar, trivalent trees whose vertices are labelled by element of $H$, modulo two relations: the IHX relation and skew symmetry (the sign changes when the orientation at a vertex is reversed). The weight is minus the number of leaves. In the $(g,\uu)$ case, the derivations in $\Der^\theta\L(H)$ of degree $m$ (i.e., weight $-m$) are linear combinations of planar trivalent graphs with $m$ vertices whose $m+2$ leaves are labelled by elements of $H$. These are subject to the same relations: IHX and skew symmetry. A precise description of this correspondence can be found in \cite{levine}. 
\end{remark}

In all but the first isomorphism above, the bracket of the associated weight graded Lie algebra preserves the weights. In the first case the bracket increases weights by 2 because of the Tate twist:
$$
\gold : \Gr^W_{-a}\Q\lambda(S)^\wedge \otimes \Gr^W_{-b}\Q\lambda(S)^\wedge \to \Gr^W_{2-a-b}\Q\lambda(S)^\wedge.
$$
When we also have an algebraic framing $\xi$ of $S$, the Goldman--Turaev Lie bialgebra is isomorphic to its associated weight graded Lie bialgebra, where the bracket and cobracket
$$
\delta_\xi : \Gr^W_{-m}\Q\lambda(S)^\wedge \to \bigoplus_{a+b=m-2}\Gr^W_{-a}\Q\lambda(S)^\wedge \otimes \Gr^W_{-b}\Q\lambda(S)^\wedge 
$$
both increase weights by 2.

Since the geometric Johnson homomorphism (\ref{eqn:geom_J}) is a morphism of MHS, we have for each $\vv\in \vV$ a commutative diagram
\begin{equation}
\label{diag:grW}
\xymatrix@C=40pt{
\g_{g,n+\vr} \ar[r]^{d\varphi_\vv}\ar[d]^\cong & \Der^\theta \p(S,\vv) \ar[d]^\cong\cr
\big(\Gr^W_\bdot\g_{g,n+\vr}\big)^\wedge \ar[r]^(0.45){\Gr^W_\bdot d\varphi_\vv} & \big(\Der^\theta \Gr^W_\bdot \p(S,\vv)\big)^\wedge
}
\end{equation}
We will call the homomorphism
\begin{equation}
\label{eqn:graded_johnson}
\Gr^W_\bdot d\varphi_\vv : \Gr^W_\bdot\u_{g,\uu} \to \Der^\theta \Gr^W_\bdot \p(S,\vv)
\end{equation}
in the $(g,\uu)$ case the {\em graded Johnson homomorphism}.

\begin{remark}
The exactness property (\ref{eqn:exactness}) of morphisms of MHS implies that
$$
\Gr^W_\bdot \im d\varphi_\vv = \im \Gr^W_\bdot d\varphi_\vv.
$$
Morita's higher Johnson homomorphism is the inclusion
$$
\im \Gr^W_\bdot d\varphi_\vv \hookrightarrow \Der^\theta \Gr^W_\bdot \p(S,\vv).
$$
Exactness also implies that $\t_{g,n+\vr} \to \Der^\theta \p(S,\vv)$ is not injective as $\ker\Gr^W_{-2}\ker d\varphi_\vv$ is non trivial. See \cite[\S14]{hain:torelli} for a complete discussion.
\end{remark}

The final observation is that the diagram in Theorem~\ref{thm:fund_diag} is isomorphic to the diagram
$$
\xymatrix@C=40pt{
\Gr^W_\bdot\g_{g,n+\vr} \ar[r]^(.45){\Gr^W_\bdot d\varphi_\vv}\ar[d]_{\Gr^W_\bdot\phitilde} & \Der^\theta \Gr^W_\bdot\p(S,\vv)\ar[d] \cr
\Gr^W_{\bdot-2} \Q\lambda(S) ^\wedge\ar[r]^(.4){\Gr^W_\bdot\kappahat_\vv} & \Der^\theta \Gr^W_\bdot\Q\pi_1(S,\vv)^\wedge
}
$$
of associated weight graded Lie algebras.

The following is a weaker version of Question~\ref{quest:johnson_inj}.

\begin{question}
\label{quest:gr_johnson_inj}
Is the Johnson homomorphism stably injective? That is, is the map
$$
\Gr^W_m d\varphi_\vv : \Gr^W_m \g_{S,\vv} \To \Der^\theta \Gr^W_m \p(S,\vv)
$$
when $g \gg -m$? 
\end{question}

If it is stably injective, Proposition~\ref{prop:stability} will imply that it is injective when $g\ge -m/3$. When $-m\le 6$ and $g\ge -3m$, the computations in \cite{morita-sakasai-suzuki} imply that the Johnson homomorphism is injective.

\subsubsection{The Lie bialgebra $\Gr^W_\bdot\Q\lambda(S)$}

Fix a complex structure on $(\Sbar,P,\vV)$ and a lift $\chitilde$ of $\chi$. These determine an isomorphism
$$
\Q\lambda(S)^\wedge \cong \prod_{n\ge 0} \Gr^W_{-n} \Q\lambda(S).
$$
Under this isomorphism, the Goldman bracket is graded and increases weights by 2. For each choice of $\vv \in \vV$, the map
$$
\Gr^W_\bdot \kappa_\vv : \Gr^W_\bdot \Q\lambda(S) \to \Der^\theta \Gr^W_{\bdot+2} \Q\pi_1(S,\vv)
$$
induced by (\ref{eqn:kk-action}) increases weights by 2 and is a Lie algebra homomorphism.

For each algebraic framing $\xi$ of $S$, the cobracket $\delta_\xi$ is also graded and increases weights by 2, making $\Gr^W_\bdot \Q\lambda(S)$ into a graded Lie bialgebra.

The bracket and cobracket on $\Gr^W_\bdot \Q\lambda(S)$ have a combinatorial description, which can be found in \cite[\S4.2]{akkn2}. In the $(g,\uu)$ case it is Schedler's Lie bialgebra \cite[\S2]{schedler}. To explain it, we need some notation. For the rest of this sub-section, $S$ is a surface of type $(g,\uu)$.

Suppose that $A$ is an associative algebra. Denote the image of $a\in A$ in its cyclic quotient
$$
\label{def:cyclic_A}
|A| := A/\langle uv - vu : u,v \in A \rangle
$$
by $|a|$. When 
$$
A = T(H) = \Q\langle \aa_1,\dots ,\aa_g,\bb_1,\dots, \bb_g\rangle \cong \Gr^W_\bdot \Q\pi_1(S,\vv)
$$
the image of $x_1 x_2 \dots x_m \in H^{\otimes m}$ is the cyclic word
$$
|x_1 x_2 \dots x_m| \in |\Gr^W_\bdot \Q\pi_1(S,\vv)| = \Gr^W_\bdot |\Q\lambda(S)^\wedge|.
$$
With this notation, the graded Goldman bracket is
\begin{multline*}
\gold : |x_1x_2\dots x_n| \otimes |y_1y_2 \dots y_m| \mapsto
\cr
 \sum_{j=1}^n \sum_{k=1}^m \langle x_j,y_k \rangle |x_{j+1} \dots x_n x_1 \dots x_{j-1} y_{k+1} \dots y_m y_1 \dots y_{k-1}|
\end{multline*}
where all $x_j,y_k \in H$. The graded cobracket $\delta_\xi$ is given by\footnote{One may be concerned that this formula does not appear to depend on the framing. But, in order to get a splitting of the weight filtration via Hodge theory, we need a complex structure on the surface. A once punctured compact Riemann surface $X=\Xbar-\{p\}$ has at most one algebraic framing as two algebraic framings of $X$ differ by a rational function whose divisor is supported at $p$, and is therefore constant. Similarly all quasi algebraic framings lie in the same homotopy class.}
\begin{multline*}
\delta_\xi : |x_1 \dots x_n| \mapsto 
\sum_{j<k} \langle x_j,x_k\rangle \big(|x_{j+1}\dots x_{k-1}|\otimes |x_{k+1}\dots x_1 x_n \dots x_{j-1}| 
\cr
- |x_{k+1}\dots x_1 x_n \dots x_{j-1}| \otimes |x_{j+1}\dots x_{k-1}|\big)
\end{multline*}
where each $x_j \in H$.

The action $\kappa_\vv : |T(H)| \to \Der^\theta T(H)$ also has a simple combinatorial description:
\begin{multline}
\label{eqn:gr_kappa}
\kappa_\vv : |x_1\dots x_n|\otimes \, y_1\dots y_m \mapsto
\cr
\sum_{j=1}^n \sum_{k=1}^m \langle x_j,y_k \rangle\, y_1 \dots y_{k-1} x_{j+1} \dots x_n x_1 \dots x_{j-1} y_{k+1} \dots y_m
\end{multline}

\subsubsection{Sketch of Proof of Theorem~\ref{thm:kk-isom}}
\label{sec:proof-kk-isom}

Suppose that $(S,\vv)$ is a surface of type $(g,n+\uu)$. Fix a complex structure on $(S,\vv)$. This will give $\Q\lambda(S)^\wedge$ and $\Q\pi_1(S,\vv)^\wedge$ mixed Hodge structures so that the KK-action
$$
\kappa_\vv : \Q\lambda(S)^\wedge \to \Der^\theta \Q\pi_1(S,\vv)^\wedge
$$
is a morphism of pro-MHS twisted by $\Q(1)$. Fix a lift $\chitilde$ of the canonical cocharacter $\chi:\Gm \to \pi_1(\MHS^\ss)$ to $\pi_1(\MHS)$. This determines isomorphisms
$$
\Q\pi_1(S,\vv)^\wedge \cong T(H)^\wedge \text{ and } \Q\lambda(S)^\wedge \cong |T(H)|^\wedge
$$
of each with its associated weight graded. Since $\kappa_\vv$ is a morphism of MHS, it corresponds to a graded Lie algebra homomorphism
$$
|T(H)| \to \Der^\theta T(H)
$$
with a shift of grading by $2$. We need to prove that it is surjective with kernel $\Q$ and is given by the formula (\ref{eqn:gr_kappa}).

Identify $H$ with its dual $H^\vee$ using the symplectic form: $u \mapsto \langle u,\blank \rangle$. The derivations of $T(H)$ are determined by their value on $H$. So we have a natural identification
$$
\Der_{n-2} T(H) = \Hom_\Q(H,H^{\otimes(n-1)}) = H^\vee \otimes H^{\otimes(n-1)} \cong H\otimes H^{\otimes(n-1)}.
$$
of the derivations of degree $(n-2)$ with $H\otimes H^{\otimes(n-1)}$. The derivation corresponding to $\sum_j u_j \otimes w_j \in \Der^\theta T(H)$ takes $\theta$ to 
$$
\sum (u_j w_j - w_j u_j)\in H^{\otimes n}.
$$
Consequently $\Der^\theta_{n-2} T(H)$ is the kernel of the commutator map
$$
c : H\otimes H^{\otimes(n-1)} \to H^{\otimes n}\quad u\otimes w \mapsto uw - wu.
$$
The cyclic group $C_n$ acts on $H^{\otimes n}$. Observe that the kernel of the commutator map consists of the $C_n$ invariant vectors $(H^{\otimes n})^{C_n}$, and its cokernel consists of the $C_n$ coinvariants $|H^{\otimes n}|$. The projection $H^{\otimes n} \to |H^{\otimes n}|$ induces isomorphisms
$$
\Gr^W_{-n+2}\Q\pi_1(S,\vv)^\wedge \cong \Der^\theta_{n-2} T(H) \cong (H^{\otimes n})^{C_n} \cong |H^{\otimes n}| \cong \Gr^W_{-n} \Q\lambda(S)^\wedge.
$$
This identification is given by the formula (\ref{eqn:gr_kappa}), as required.

\subsubsection{Sketch of Proof of Proposition~\ref{prop:S2}}
\label{sec:proof}

Recall that $(S,\vv)$ is a surface of type $(g,\uu)$. To prove that $\kappa_\vv$ induces an isomorphism $|\Sym^2 \p(S,\vv)| \to \Der^\theta\p(S,\vv)$ is an isomorphism, it suffices to prove the graded version. We use the graphical description of elements of $\L(H)$ and $\Der^\theta \L(H)$ described in Remark~\ref{rem:graphical_rep}.

Consider a derivation represented by a connected trivalent graph. Cutting any graph that represents it at the mid-point of any edge divides it into 2 rooted trees that are well-defined mod the relations IHX and skew symmetry. This gives an element of $\Sym^2\L(H)$. See figure~\ref{fig:simplification}.
\begin{figure}[!ht]
\includegraphics[width=4in]{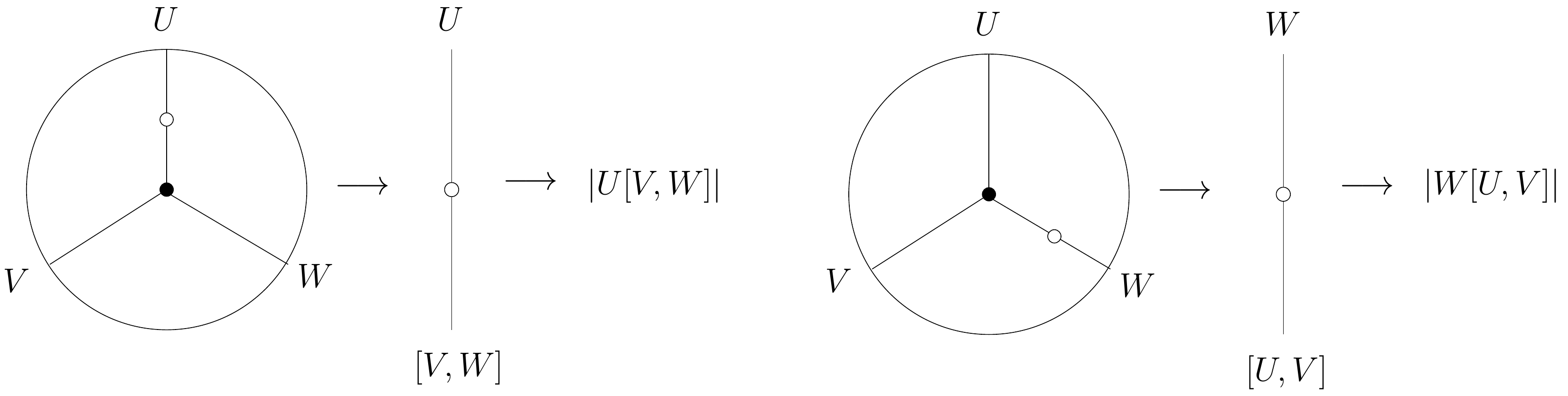}
\caption{Two ways of decomposing a planar tree into a symmetric pair of rooted trees}
\label{fig:simplification}
\end{figure}
This element of $|\Sym^2\L(H)|$ is well-defined as, for $U,V,W\in \L(H)$ (and so, rooted trees) we have
$$
|[U,V]W| =  |UVW|-|VUW| = |U[V,W]|
$$
The formula (\ref{eqn:gr_kappa}) implies that the $\kappa$ action of $|\Sym^2\L(H)|$ on $T(H)$ is easily seen to coincide with the standard action of decorated planar trivalent graphs on $\L(H)$.

\section{Presentations of the Lie Algebras \texorpdfstring{$\g_{g,n+\uu}$}{g sub n+v}}
\label{sec:presentations}

The restrictions (\ref{eqn:w_filt}) on the weight filtration of $\g_{g,n+\vr}$ imply that there is a graded Lie algebra isomorphism
$$
\Gr^W_\bdot \g_{g,n+\vr} \cong \sp(H) \ltimes \Gr^W_\bdot \u_{g,n+\vr}.
$$
The choice of a complex structure $\phi$ and a lift $\chitilde$ determines a Lie algebra isomorphism
$$
\g_{g,n+\vr} \cong \sp(H)\ltimes \prod_{m>0}\Gr^W_{-m} \u_{g,n+\vr}.
$$
So, to give a presentation of $\g_{g,n+\vr}$, it suffices to give a presentation of the graded Lie algebra $\Gr^W_\bdot \u_{g,n+\vr}$ in the category of $\Sp(H)$-modules.

In this section, we recall the presentations of $\Gr^W_\bdot \u_{g,n+\vr}$ when $n+\vr \le 1$ for all $g\neq 2$. When $g=2$, the only case where we do not have complete presentations, we give a partial presentation.

The free Lie algebra on a vector space $V$ will be denoted by $\L(V)$ and the free Lie algebra on the set $\{\ee_\alpha : \alpha \in A\}$ will be denoted by $\L(\ee_\alpha: \alpha \in A)$. Their enveloping algebras are the tensor algebra $T(V)$ and the free associative algebra $\kk\langle \ee_\alpha : \alpha \in A\rangle$, respectively. \label{def:free_LA}

\subsubsection{Genus 0}
\label{sec:genus0}

When $g=0$, the group $\cG_{0,n+\vr}$ is unipotent completion of $\G_{0,n+\vr}$. The presentations below follow directly from Kohno's work \cite{kohno}, which in turn follow directly from Chen's work \cite[Thm.~2.1.1]{chen:malcev} using \cite[Lem.~5]{brieskorn}. When $r=0$, its associated weight graded Lie algebra $\Gr^W_\bdot \g_{0,n+1}$ is the quotient of the free Lie algebra
$$
\L(\ee_{j,k}: \{j,k\} \text{ is a 2-element subset of } \{0,\dots,n\})
$$
by the ideal of relations generated by
\begin{align*}
\sum_{j=0}^n \ee_{j,k} &= 0,\ \text{ for each } k \in \{0,\dots,n\}, \cr
[\ee_{j,k},\ee_{s,t}] &= 0 \text{ when } j,k,s,t \text{ distinct},\cr
[\ee_{j,\ell}+\ee_{\ell,k},\ee_{j,k}] &= 0 \text{ when } j,k,\ell \text{ distinct}.
\end{align*}
Here we are using the convention that each $\ee_{j,j}=0$.


Each generator $\ee_{j,k}$ has weight $-2$. The symmetric group $\Sigma_{n+1}$ acts on $\g_{0,n+1}$ by permuting the indices. The first homology $H_1(\g_{0,n+1})$ is the irreducible $\Sigma_{n+1}$-module corresponding to the partition $[n-1,2]$. It has dimension
$$
\dim H_1(\g_{0,n+1}) = \binom{n}{2}-1.
$$

The Lie algebra of $\g_{0,n+\uu}$ is
$$
\g_{0,n+\uu} = \g_{0,n+1} \oplus \Q\zz_0
$$
where $\zz_0$ is central and spans a copy of $\Q(1)$. There is a natural isomorphism of the weight graded of the Lie algebra of $\pi_1^\un(\P^1-\{p_0,\dots,p_n\},\vv)$ with
\begin{equation}
\label{eqn:punct_P1}
\L\big(H_1(\P^1-\{p_0,\dots,p_n\})\big) \cong \L(\ee_0,\dots,\ee_n)/(\ee_0 + \dots + \ee_n)
\end{equation}
where $\ee_j$ is the homology class of a small loop encircling $x_j$ and $\vv \in T_{p_0}\P^1$. The homomorphism
$$
\Gr^W_\bdot \g_{0,n+\uu} \to \Der \L(\ee_0,\dots,\ee_n)/(\ee_0 + \dots + \ee_n)
$$
induces by the action of $\G_{0,n+\uu}$ on $\pi_1^\un(\P^1-\{p_0,\dots,p_n\},\vv)$, where $\vv \in T_{p_0}\P^1$, is
given by
\begin{equation}
\label{eqn:action_genus0}
\ee_{j,k} : \ee_t \mapsto (\delta_{j,t} - \delta_{k,t}) [\ee_j,\ee_k] \text{ and } \zz_0 : \ee_t \mapsto [\ee_0,\ee_t].
\end{equation}

These formulas are derived from the formula for the action of the unipotent completion of the pure braid group $P_n$ on the unipotent completion of a free group of rank $n$. The formula for this action is obtained from the presentation of the unipotent completion of $P_{n+1}$ (see \cite{kohno}) as $\C-\{x_1,\dots,x_n\}$ is the fiber of the projection of the configuration space $\C^{n+1} - \{z_j\neq z_k\}$ over $(z_1,\dots,z_n)$ under the projection $(z_0,z_1,\dots,z_n)\mapsto (z_1,\dots,z_n)$.

\subsubsection{Representation theory preliminaries}
\label{sec:rep_th}

There are several $\Sp(H)$-modules that play a prominent role when $g$ is positive, particularly when $g\ge 3$. First set
$$
\label{def:theta}
\theta = \sum_{j=1}^g a_j \wedge b_j \in \Lambda^2 H
$$
where $a_1,\dots,a_g,b_1,\dots,b_g$ is a symplectic basis of $H$. It spans a copy of the trivial representation in $\Lambda^2 H$.

When $g\ge 2$, the $\Sp(H)$-module map $\theta \wedge \blank : H \to \Lambda^3 H$ is an $\Sp(H)$-invariant injection. When $g=2$, it is an isomorphism, but when $g\ge 3$, it has an irreducible $\Sp(H)$-invariant complement that we denote by $\Lambda^3_0 H$. It is the kernel of the contraction
$$
\Lambda^3 H \to H,\quad u \wedge v \wedge w \mapsto \langle u,v \rangle w + \langle v,w \rangle u + \langle w,u \rangle v.
$$

\begin{remark}
Johnson's computation of the abelianization of the Torelli groups implies that for all $g\ge 3$ and all $n,r$, there are $\Sp(H_\Z)$-invariant isomorphisms
\begin{equation}
\label{eqn:johnson_iso}
H_1(T_{g,n+\vr};\Q) \cong H_1(\u_{g,n+\vr}) \cong \Lambda^3_0 H_\Q \oplus H^{\oplus n}_\Q.
\end{equation}
This is stated in \cite[Prop.~7.3]{hain:completions} and is easily proved by induction on $n+r$ after noting that the monodromy representation  associated to adding one point (or tangent vector) is non-trivial.\footnote{It also follows from \cite[Prop.~2.1]{hain:torelli} and \cite[Prop.~3.5]{hain:torelli}. Note that $g$ should be $\ge 2$ in \cite[Prop.~2.1]{hain:torelli}.} This and Hodge theory imply that $H_1(\u_{g,n+\vr})$ is pure of weight $-1$ for all $g\ge 3$. (See \cite[Prop.~4.6]{hain:torelli}.)
\end{remark}

For all $g\ge 2$, there is an irreducible $\Sp(H)$-module corresponding to the partition $[2,2]$. We denote it by $V_\boxplus$. It is the highest weight submodule of $\Sym^2\Lambda^2 H$. When $g\ge 3$, there is a unique copy of $V_\boxplus$ in $\Lambda^2 \Lambda^3_0 H$ and in $\Lambda^2 \Lambda^3 H$. When $g=3$,
$$
\Lambda^2\Lambda^3_0 H  = \Q \oplus V_\boxplus.
$$

For all $g\ge 2$, we have
$$
\Gr^W_{-m} \Der^\theta \L(\Lambda^3 H) \cong
\begin{cases}
\sp(H) & m=0,\cr
\Lambda^3 H & m = 1, \cr
V_\boxplus \oplus \Lambda^2 H & m=2.
\end{cases}
$$
When $g\ge 3$, the homomorphism
$$
\Gr^W_\bdot d\varphi_\vv : \Gr^W_\bdot \g_{g,\uu} \to \Der^\theta \Gr^W_\bdot \p(S,\vv)
$$
is an isomorphism in weights $m=0,-1,-2$. In particular,
$$
\Gr^W_m \u_{g,\uu} \cong \Gr^W_m \Der^\theta \L(H) \text{ when } m = -1,-2.
$$
See \cite[\S10]{hain:torelli}. It also holds in genus 2 as we explain below.

\subsubsection{Genus $g\ge 3$}
In this case, all generators of $\u_{g,n+\vr}$ have weight $-1$ and $\Gr^W_\bdot \u_{g,n+\vr}$ is quadratically presented for all $g\ge 4$. In genus 3, there are quadratic and cubic relations. Explicit relations can be found in \cite{hain:torelli, hain:kahler} and, from a different point of view, in \cite[\S9]{hain:rat_pts}.

\begin{theorem}[Hain \cite{hain:torelli,hain:kahler}]
For all $g>3$ and all $n,r\ge 0$, the graded Lie algebra $\Gr^W_\bdot\u_{g,n+\vr}$ is quadratically presented. In particular,
$$
\Gr^W_\bdot \u_{g,\uu} \cong \L(\Lambda^3 H)/(R_2)
$$
where $R_2$ is the kernel of the Lie bracket $\Lambda^2 \Gr^W_{-1}\Der^\theta \L(H) \to \Gr^W_{-2}\Der^\theta \L(H)$. In genus 3, the graded Lie algebra $\u_{3,n+\vr}$ has non-trivial quadratic and cubic relations. These are determined by the condition that $\Gr^W_\bdot \u_{3,\uu} \to \Der^\theta\Gr^W_\bdot \L(H)$ is an isomorphism in weights $-1,-2$, and injective when $m=-3$.
\end{theorem}

Since the geometric Johnson homomorphism is a morphism of MHS \cite[Lem.~4.5]{hain:torelli}, exactness of $\Gr^W_\bdot$ and the previous result imply that the Lie algebra $\bigoplus_k (\Gr^k_J T_{g,\uu})\otimes \Q$ is generated in degree 1.

\begin{corollary}
\label{cor:gen_deg_1}
When $g\ge 3$, the image of the graded Johnson homomorphism (\ref{eqn:graded_johnson}) is generated by
$$
\Gr^W_{-1} \u_{g,\uu} \cong \Lambda^3 H.
$$
\end{corollary}

The action of $\Gr^W_\bdot \u_{g,\uu}$ on $\Gr^W_\bdot \p(S,\vv) \cong \L(H)$ is determined by the action of $\Gr^W_{-1} \u_{g,\uu}\cong\Lambda^3 H$ on $\L(H)$. This is given by
\begin{multline*}
u_1\wedge u_2 \wedge u_3 \mapsto -\left\{v \mapsto \langle u_1, v\rangle
[u_2,u_3] + \langle u_2, v\rangle[u_3,u_1] + \langle u_3, v\rangle[u_1,u_2]\right\} \cr
\in \Hom(H,\Lambda^2 H) \subseteq \Der \L(H).
\end{multline*}
In the graphical version sketched in Remark~\ref{rem:graphical_rep}, this derivation corresponds to the planar trivalent graph with one vertex whose 3 leaves are labelled by $u_1, u_2, u_3$.

\subsubsection{Genus 1}
\label{sec:genus1}

While the genus 0 story is combinatorial (a feature of varieties that are closely related to hyperplane complements) and the $g\ge 3$ story, which is both geometric (via its relation the Ceresa cycle \cite[\S6]{hain:completions} and the Johnson isomorphisms), the genus 1 story is rich, with a distinctly arithmetic flavour because of its connection to classical modular forms. Here we will give a brief introduction. Full details can be found in \cite{hain:modular,hain-matsumoto:mem}.

One has the central extension
$$
0 \to \Q(1) \to \g_{1,\uu} \to \g_{1,1} \to 0.
$$
The first difference between genus 1 and all other genera is that neither $\u_{1,1}$ nor $\u_{1,\uu}$ is finitely generated. The Lie algebra $\Gr^W_\bdot\u_{1,\uu}$ is generated by the graded $\SL(H)$-module $V$ with
$$
\Gr^W_{-1} V = \bigoplus_{n>0} H^1_\cusp(\SL_2(\Z);\Sym^{2n}H)^\vee\otimes \Sym^{2n}H
\text{ and }
\Gr^W_{-2n} V = \Sym^{2n-2}H.
$$
Here $(\phantom{X})^\vee$ denotes dual.
The only relations in $\Gr^W_\bdot \u_{1,\uu}$ assert that the natural generator of $\Gr^W_{-2}\L(V)$, denoted $\ee_2$, is central, so that $\Gr^W_\bdot \u_{1,1}$ is a free Lie algebra. There is a natural torus in $\SL_2$ (explained in \cite[\S10]{hain:modular}) and the natural choice of a highest weight vector $\ee_{2n}$ in $\Sym^{2n-2}H$. It is dual to the normalized Eisenstein series of weight $-2n$, even when $n=1$.

Let $(S,\vv)$ be a surface of type $(1,\uu)$. Identify $\Gr^W_\bdot \p(S,\vv)$ with $\L(H)$ via the canonical isomorphism. The representation
$$
\Gr^W_\bdot\g_{1,\uu} \to \Der^\theta \L(H)
$$
is far from injective. Its kernel is not even finitely generated as it contains the infinite dimensional vector space $\Gr^W_{-1}V$. However, the images of the $\ee_{2n}$ are non-zero. In fact, the monodromy homomorphism factors canonically
$$
\g_{1,\uu} \to \g_{1,\uu}^\geom \to \Der^\theta \p(S,\vv)
$$
through the Lie algebra $\g_{1,\uu}^\geom$ in the category of MHS. This Lie algebra is an extension of $\sp(H)$ by a pronilpotent Lie algebra that is generated topologically by
$$
\bigoplus_{n\ge 0} \Sym^{2n}H.
$$
The Lie algebra $\g^\geom_{1,\uu}$ is the Lie algebra of the group $\pi_1^\geom(\MEM_\uu,\w^B)$ defined in \cite[\S7.1]{hain-matsumoto:mem}. It is the image of $\g_{1,\uu}$ in the Lie algebra $\g^\MEM_{1,\uu}$ mentioned after Question~\ref{quest:johnson_inj}.

We abuse notation and denote the highest weight vector of $\Sym^{2n-2}H$ by $\ee_{2n}$. It has weight $-2n$. The fact that $\g^\geom_{1,\uu}$ has a mixed Hodge structure forces a countable set of relations to hold between the $\ee_{2n}$ in $\g_{1,\uu}^\geom$. So that $\g_{1,\uu}^\geom$ is far from being free; each normalized Hecke eigenform $f$ of $\SL_2(\Z)$ determines a countable set of independent relations, one of each ``degree'' $\ge 2$. These are the ``Pollack relations''.

Fix a symplectic basis of $\aa,\bb$ of $H$. For each $n\ge 0$, there is a unique derivation\footnote{These derivations were first considered by Tsunogai \cite{tsunogai}.} $\e_{2n} \in \Gr^W_{-2n} \Der^\theta\L(H)$ of $\L(\aa,\bb) = \L(H)$ satisfying
$$
\e_{2n}(\theta) = 0 \text{ and } \e_{2n}(\bb) = \ad_\bb^{2n}(\aa).
$$
The graded monodromy representation $\Gr\g^\geom_{1,\uu} \to \Der^\theta \Gr^W_\bdot \p(S,\vv)$ takes $\ee_{2n}$ to $2\e_{2n}/(2n-2)!$\ , \cite[Thm.~15.7]{hain:modular}. The first 2 quadratic Pollack relations are:
$$
[\e_4,\e_{10}] - 3[\e_6,\e_8] = 0 \text{ and }
2[\e_4,\e_{14}] -7[\e_6,\e_{12}] + 11[\e_8,\e_{10}] = 0.
$$
These correspond to the cusp forms of $\SL_2(\Z)$ of weights 12 and 16, respectively.

\begin{remark}
Pollack \cite{pollack} found all quadratic relations between the $\e_{2n}$ in $\Der^\theta\L(H)$ and all higher degree relations modulo the third term of the ``elliptic depth filtration'' of $\Der^\theta\L(H)$ in \cite{pollack}. All relations were proved to lift to $\Der^\theta\L(H)$ and to be motivic in \cite[\S25]{hain-matsumoto:mem} and \cite[Thm.~20.3]{brown:mmm}. It is not known whether these generate all relations between the $\ee_{2n}$ in either $\u_{1,\uu}^\geom$ or their images in $\Der^\theta\L(H)$.
\end{remark}

\subsubsection{Genus 2}

This is the only genus in which we do not have a complete presentation of $\Gr^W_\bdot \u_{g,n+\vr}$, although Watanabe has made significant progress. He has computed $H_1(\u_2)$ and proved that $\u_{2}$ is finitely presented. This is surprising as $T_2$ is countably generated free group.

\begin{theorem}[Watanabe \cite{watanabe}]
There is an isomorphism of $\Sp(H)$-modules
$$
H_1(\u_2) \cong V_\boxplus
$$
where $V_\boxplus$ is in weight $-2$ and $\Gr^W_\bdot \u_2$ has a minimal presentation
$$
\Gr^W_\bdot \u_2 = \L(V_\boxplus)/(R_4,R_6,R_8,R_{10},R_{14})
$$
where $R_m$ is in weight $-m$.
\end{theorem}

These relations have yet to be determined. A conjectural description is given below.

Exactness properties of relative completion imply that
$$
0 \to H^{n+r} \to H_1(\u_{2,n+\vr}) \to H_1(\u_2) \to 0
$$
is exact. This implies that
$$
0 \to H^{n+r} \to \Gr^W_\bdot H_1(\u_{g,n+\vr}) \to \Gr^W_\bdot H_1(\u_2) \to 0
$$
is exact and that $\Gr^W_\bdot \u_{2,n+\vr}$ is finitely presented for all $n$ and $r$.

\begin{corollary}
In genus 2, the image of the graded Johnson homomorphism (\ref{eqn:graded_johnson}) is generated by
$$
\Gr^W_{-1} \u_{2,\uu} \oplus \Gr^W_{-2}\u_{2} \cong \Gr^W_\bdot H_1(\u_{2,1}) \cong H \oplus V_\boxplus.
$$
\end{corollary}

The indecomposable relations $R_{2n}$ can occur only when $2n$ is a weight for which there are no cusp forms of $\SL_2(\Z)$ of weight $2n$. Using work of Dan Petersen \cite{petersen}, Watanabe shows that $R_{2n}$ is related to the Eisenstein series of $\SL_2(\Z)$ of weight $2n$. Why this happens is still quite mysterious. We now propose a relations in $\L(V_\boxplus)$ that should generate the $\Sp(H)$-module $R_{2n}$.

Fix a decomposition $S = S' \cup S''$ of a closed genus 2 surface $S$ into two surfaces of type $(1,\uu)$. The obvious homomorphism
$$
\G_{1,\uu}\times \G_{1,\uu} \cong \G_{S',\partial S'} \times \G_{S'',\partial S''} \to \G_S \cong \G_2.
$$
induces a homomorphism $\cG_{1,\uu}\times \cG_{1,\uu} \to \cG_2$ on relative completions and a homomorphism $\u_{1,\uu}\oplus \u_{1,\uu} \to \u_2$. Using limit MHS (see Section~\ref{sec:decompositions}), we can make this a morphism of MHS. It can therefore be identified with its associated graded
$$
\Gr^W_\bdot \u_{1,\uu}\oplus \Gr^W_\bdot \u_{1,\uu} \to \Gr^W_\bdot \u_2.
$$
Set
$$
\e_{2n}' = (\e_{2n},0) \text{ and } \e_{2n}'' = (0,\e_{2n}).
$$
The first is supported in $S'$ and the second in $S''$. These have weight $-2n$ and $\e_{2j}'$ commutes with $\e_{2k}''$ for all $j,k\ge 1$.

\begin{conjecture}
When $2n = 6, 8, 10, 14$, the relations $[\ee_{2j}',\ee_{2k}''] = 0$ with $j+k=n$ generate $R_{2n}$ as an $\Sp(H)$-module.
\end{conjecture}

Once one has a presentation of $\Gr^W_\bdot \u_2$, it is straightforward to get a presentation for all $\u_{2,n+\vr}$. In the $r=0$ case, this is because the kernel of $\u_{2,\n+\vr} \to \u_2$ is the Lie algebra $\p_{2,n}$ of the configuration space of $n$ points on a genus 2 surface. This has a natural MHS which is generated in weight $-1$ and has a well-known presentation due to Bezrukavnikov \cite{bezrukavnikov}.\footnote{See \cite[\S2,\S12]{hain:torelli} for another proof plus the Hodge theory. Note that $g$ should be $\ge 2$ in Proposition~2.1 and in \S12.} So, to lift the presentation of $\Gr^W_\bdot\u_2$ to a presentation of $\Gr^W_\bdot\u_{2,n}$, one only has to determine the action
$$
\Gr^W_{-2}\u_2 \otimes \Gr^W_{-1}\p_{2,n} \to \Gr^W_{-3} \p_{2,n}
$$
which one can compute using the graded Johnson homomorphism.

\part{Arithmetic Johnson Homomorphisms and the Galois Image}

The primary goal of this part is to define the arithmetic Johnson homomorphism $\ghat \to \Der\p$ associated to a surface $S$ of type $(g,n+\uu)$ and to show that, when $n=0$, its image ``almost lies'' in the kernel of the Turaev cobracket
$$
\delta_\xi : \Der^\theta \p \to \Q\lambda(S)^\wedge \comptensor \Q\lambda(S)^\wedge.
$$
Since the image $\gbar$ of the geometric Johnson homomorphism lies in $\ghat$, this also constrains the classical Johnson image. We also explain that Oda's conjecture (whose proof was completed by Takao \cite{takao}) implies that
$$
\ghat/\gbar \cong \k = \L(\sigma_3,\sigma_5,\sigma_7,\dots )^\wedge,
$$
so that there is a significant gap between the kernel of the cobracket and the classical Johnson image, with the space between occupied by ``arithmetic derivations''. Arithmetic Johnson homomorphisms are constructed in Section~\ref{sec:arith_johnson}, where we also discuss Oda's Conjecture and its relevance. The constraints on $\ghat$, as well as their relation the constraints on the graded Johnson image given by the Enomoto--Sato trace, are discussed in Section~\ref{sec:johnson}.

Explaining these results requires two digressions. The first (Section~\ref{sec:decompositions}) contains the technical tools that allow us to decompose an algebraic curve into pieces while preserving the action of $\pi_1(\MHS)$ on the invariants we are considering. This is the algebraic analogue of cutting a surface along a set of disjoint simple closed curves. Important examples of such decompositions are ``Ihara curves''. These are higher genus generalizations of Tate's elliptic curve and are algebraic analogues of pants decompositions of a surface. Their importance is that they allows us to reduce the higher genus case to the case of the algebraic analogue of a pair of pants, the thrice punctured sphere $\Pminus$. This is needed to relate the Johnson story to mixed Tate  motives via the work \cite{brown} of Brown. The second (Section~\ref{sec:edge}) is more technical and concerns ``edge homomorphisms'' and the Turaev cobracket in genus 0. It allows us to show that, while the cobracket vanishes on $[\k,\k]$, the homomorphism induced on $H_1(\k)$ by the cobracket is injective. The work on edge homomorphisms is part of the work of Alekseev--Kawazumi--Kuno and Naef \cite{akkn:genus0} on the Kashiwara--Vergne problem in genus 0. Both sections can be skimmed on a first reading but are included as they are an important part of the story.

\section{Arithmetic Johnson Homomorphisms}
\label{sec:arith_johnson}

The geometric Johnson homomorphism (\ref{eqn:geom_J}) extends to a larger Lie algebra that can be constructed using the action of the arithmetic fundamental group of the moduli stack $\M_{g,n+\vr/\Q}$ on $\pi_1^\un(S,\vv)$. One way to do this is to use the weighted completion of arithmetic mapping class groups, \cite{hain-matsumoto:density,hain:rat_pts}. Here we construct it using Hodge theory.

Throughout this section, $(\Sbar,P,\vV)$ is a surface of type $(g,n+\vr)$. Fix $\vv \in \vV$. Denote the image of the geometric Johnson homomorphism
$$
d\varphi_\vv : \g_{g,n+\vr} \to \Der^\theta \p(S,\vv)
$$
by $\gbar_{g,n+\vr}$ and the image of $\u_{g,n+\vr}$ by $\ubar_{g,n+\vr}$.

The MHS on $\p(S,\vv)$ associated to a complex structure $\phi$ on $(\Sbar,P,\vV)$ is determined by a homomorphism
$$
\pi_1(\MHS) \to \Aut^\theta \p(S,\vv).
$$
The image of this homomorphism is (by definition) the Mumford--Tate group $\MT_\phi$ of the MHS on $\p(S,\vv)$ associated to $\phi$. Denote its Lie algebra by $\m^\phi$. The size and location of $\m^\phi$ in $\Der\p(S,\vv)$ are very sensitive to the choice of $\phi$. Since the geometric Johnson homomorphism is a morphism of MHS, its image $\gbar_{g,n+\vr}$ is a sub-MHS of $\Der^\theta \p(S,\vv)$. It is therefore normalized by $\m^\phi$. \label{def:MT}

Define the Lie algebra $\ghat_{g,n+\vr}^\phi$ to be the Lie subalgebra of $\Der^\theta\p(S,\vv)$ generated by $\gbar_{g,n+\vr}$ and $\m^\phi$. Since $\m^\phi$ normalizes $\gbar_{g,n+\vr}$, $\gbar_{g,n+\vr}$ is an ideal of $\ghat_{g,n+\vr}^\phi$ and $\ghat_{g,n+\vr}^\phi/\gbar_{g,n+\vr}$ is a quotient of $\m^\phi$. \label{def:ghatphi}

Denote the category of mixed Tate motives, unramified over $\Z$, by $\MTM$.\label{def:MTM} This category is constructed in \cite{deligne-goncharov}. It is a tannakian category and therefore the category of representations of a group, $\pi_1(\MTM,\w^B)$, where $\w^B$ is the Betti fiber functor.\footnote{It is typically better to use the de~Rham fiber functor as in the mixed Tate case, there is a canonical choice of the cocharacter $\chitilde$. But to make constructions compatible with more topological constructions in this paper, we use $\w^B$.} Its Lie algebra $\mtm$ is an extension
$$
\label{def:k}
0 \to \k \to \mtm \to \Q \to 0
$$
where $\Q$ is the Lie algebra of $\Gm$ and $\k \cong \L(\sigma_3,\sigma_5,\sigma_7,\sigma_9,\dots)^\wedge$,
with $\sigma_{2n+1} \in \Gr^W_{-4n-2}\k$.

\begin{remark}
\label{rem:sigmas}
Note that this isomorphism is {\em not} canonical as the splitting of the weight filtration depends on the choice of a lift $\chitilde$, and also because, apart from rescaling the $\sigma_{2n+1}$'s, we can replace $\sigma_{11}$ by $\sigma_{11} + [\sigma_3,[\sigma_3,\sigma_5]]$, etc.
\end{remark}

\begin{theorem}[Brown \cite{brown}]
The Mumford--Tate group of $\pi_1^\un(\Pminus,\vv)$ is $\pi_1(\MTM)$.
\end{theorem}

Takao's affirmation \cite{takao} of the Oda Conjecture \cite{oda}, Brown's result above, and the proof of \cite[Prop.~13.1]{hain:turaev} show that, although $\m^\phi$ can be large, the quotient $\m^\phi/(\gbar\cap \m^\phi)$ is constant and isomorphic to $\mtm$.

\begin{theorem}[Oda Conjecture]
\label{thm:arith_extn}
The Lie subalgebra $\ghat_{g,n+\vr}^\phi$ of $\Der^\theta\p(S,\vv)$ does not depend on the complex structure $\phi$ on $(\Sbar,P,\vV)$. It is an extension
$$
0 \to \gbar_{g,n+\vr} \to \ghat_{g,n+\vr}^\phi \to \mtm \to 0
$$
where $\mtm$ denotes the Lie algebra of the category $\MTM$ of mixed Tate motives unramified over $\Z$.
\end{theorem}

I suspect that, now that Brown's Theorem is known, a shorter proof of the prounipotent version of Oda's Conjecture can be given. (See also Question~\ref{quest:centralizer}.) Also, since $\ghat_{g,n+\vr}^\phi$ does not depend on the choice of $\phi$, we will henceforth drop the $\phi$ and write $\ghat_{g,n+\vr}$.

The following result is proved using the argument in the proof of \cite[Prop.~13.3]{hain:goldman}, the main ingredient of which is the injectivity result \cite[Thm.~5.2.1]{kk:groupoid} of Kawazumi and Kuno.

\begin{proposition}
\label{prop:lift}
The inclusion of $\ghat_{g,n+\vr}$ into $\Der^\theta \p(S,\vv)$ lifts to a Lie algebra homomorphism
$$
\phihat : \ghat_{g,n+\vr} \to \Q\lambda(S)^\wedge
$$
such that the diagram
$$
\xymatrix{
\ghat_{g,n+\vr} \ar[r]\ar[d]_\phihat & \Der^\theta \p(S,\vv)\ar[d] \cr
\Q\lambda(S) ^\wedge\ar[r]^(.4){\kappahat_\vv} & \Der^\theta \Q\pi_1(S,\vv)^\wedge
}
$$
commutes.
\end{proposition}

\section{Hodge Theory and Decomposition of Surfaces}
\label{sec:decompositions}

It is often useful to decompose a surface into subsurfaces, such as when one takes a pants decomposition. In this section, we explain how to make such constructions compatible with Hodge theory. The basic idea to consider the decomposed surface as a model of a first order smoothing of the nodal surface obtained by contracting each circle in the decomposition to a point. The mixed Hodge structure on an invariant of the decomposed surface is a ``limit MHS".

If $T$ is a subsurface of a closed surface $S$, then the induced map $\Z\lambda(T) \hookrightarrow \Z\lambda(S)$ preserves the Goldman bracket. If $\xi$ is a framing of $S$ (and thus of $T$), then this inclusion also preserves the cobracket $\delta_\xi$ and induces a Lie bialgebra homomorphism $\Q\lambda(T)^\wedge \to \Q\lambda(S)^\wedge$. More generally, we can consider how the Goldman--Turaev Lie bialgebra behaves when an oriented surface is decomposed into a union of closed subsurfaces by cutting along a finite set of disjoint simple closed curves.

Since surfaces with boundary are not algebraic curves, the map $T \to S$ induced by a decomposition of $S$ into subsurfaces does not appear to be compatible with algebraic geometry and Hodge theory. In this section, we explain briefly how decompositions $T\to S$ do induce morphisms of MHS. The essential point is to regard the decomposed surface as a first order smoothing of a nodal surface. The MHSs that arise are ``limit MHS''. Material in this section is explained in detail in \cite{hain:ihara}.

\subsection{Smoothings of nodal surfaces}

Suppose that $S$ is an oriented surface. Note that we are not assuming that $S$ be connected. A non-zero tangent vector $\ww$ of $S$ at a point $p$ determines a marked point $[\ww]$ on the boundary circle $(T_p S-\{0\})/\R_{>0}$ of the real oriented blowup $\Bl_pS$ of the surface at $p$. This boundary circle is naturally an $S^1$-torsor. Given two non-zero tangent vectors $\ww' \in T_{p'}S$ and $\ww'' \in T_{p''} S$ anchored at two distinct points $p',p''\in S$, we can ``smooth'' the nodal surface
$$
S_0 = S/(p'\sim p'')
$$
by glueing the two boundary circles of the real oriented blowup $\Shat := \Bl_{p',p''}S$ of $S$ that lie over $p'$ and $p''$ by identifying $e^{i\theta}[\ww'] \in \partial\Shat$ with $e^{-i\theta}[\ww'']\in \partial\Shat$. Denote the resulting surface
\begin{figure}[!ht]
\begin{tikzpicture}
\node[above right] at (0,0.5) {\includegraphics[width=4in]{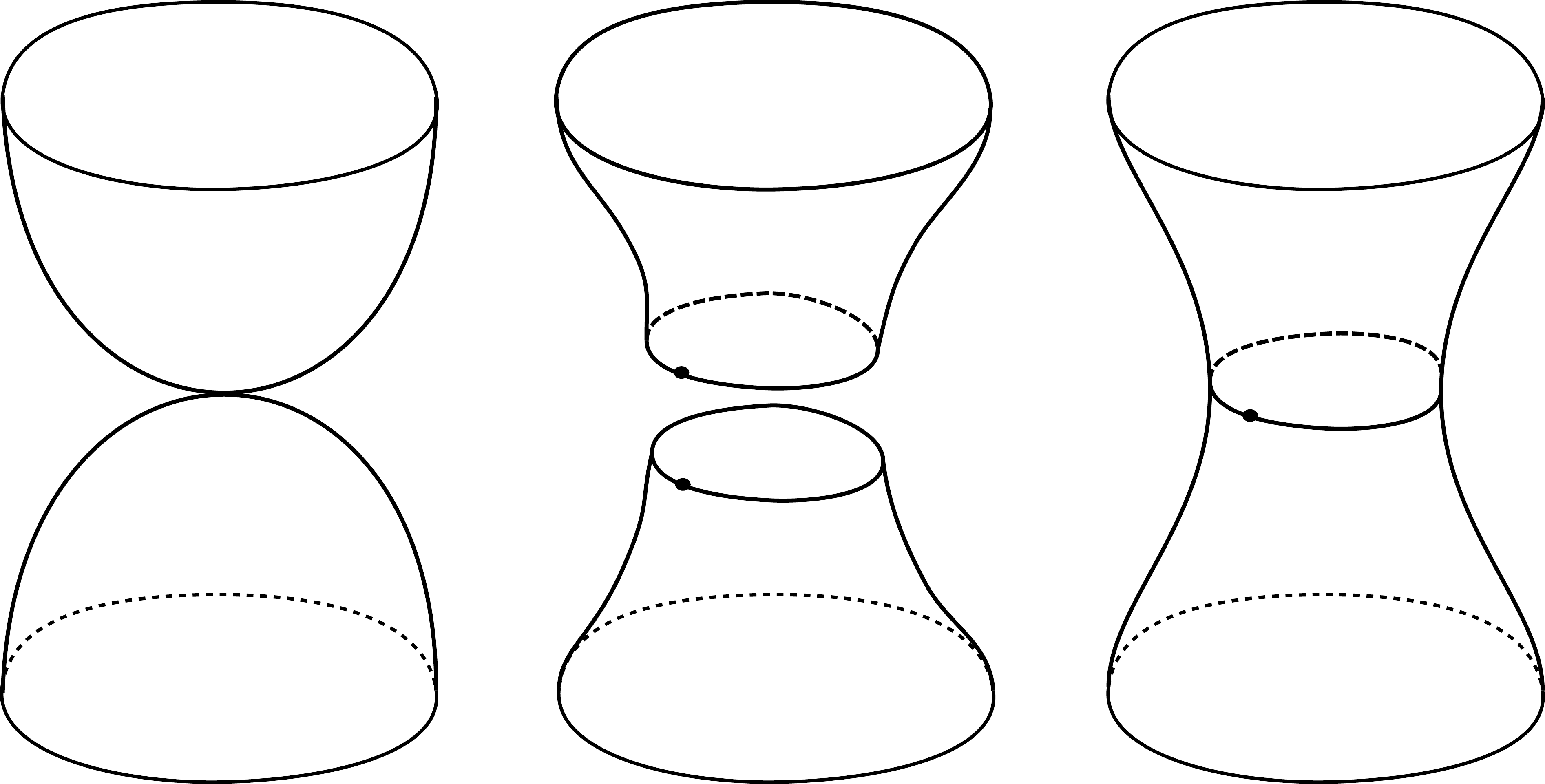}};
\node at (1.5,0) {$S_0$};
\node at (5.25,0) {$\Shat$};
\node at (9,0) {$S_{\ww'\otimes \ww''}$};
\node at (5,2.2) {$[\ww'']$};
\node at (5,3.5) {$[\ww']$};
\end{tikzpicture}
\caption{Smoothing a node}
\end{figure}
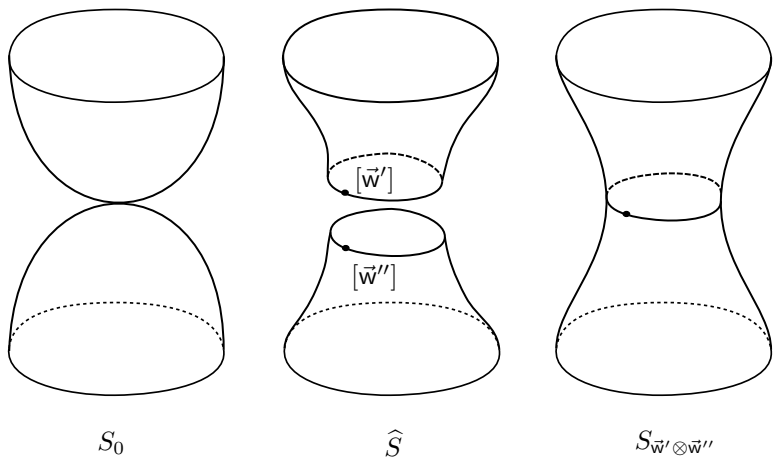
by $S_{\ww'\otimes\ww''}$. {Observe that $S_{(e^{i\theta}\ww')\otimes\ww''}=S_{\ww'\otimes(e^{i\theta}\ww'')}$.} We call it the {\em (topological) smoothing associated to} $\ww'\otimes\ww''$ of the nodal surface $S_0$. The image of the boundary circles of $\Shat$ in $S_{\ww'\otimes\ww''}$ is the {\em vanishing cycle}. The (positive) Dehn twist on the vanishing cycle is called the {\em monodromy} operator. The quotient of $S_{\ww'\otimes\ww''}$ obtained by collapsing the vanishing cycle to a point is canonically homeomorphic to $S_0$.

This construction generalizes to nodal surfaces $S_0$ that are obtained by identifying disjoint pairs of distinct points $p_j'$ and $p_j''$ ($j=1,\dots,m$) of $S$. A smoothing is determined by the set $\{\ww_j'\otimes \ww_j'':j=1,\dots m\}$, where $\ww_j' \in T_{p_j'}S$ and $\ww_j'' \in T_{p_j''}S$ are non-zero. There is one vanishing cycle for each node. The monodromy operator is the product of the Dehn twists about the vanishing cycles. We will denote the smoothed surface\footnote{The family over $\prod_{j=1}^m \big(T_{p_j'}S\otimes T_{p_j''}S-\{0\}\big)$ whose fiber over $\vv$ is $S_\vv$ is topologically locally trivial. The monodromy about the $j$th coordinate hyperplane $\ww_j'\otimes\ww_j'' = 0$ is the Dehn twist about the $j$th vanishing cycle.} by $S_\vv$, where $\vv = \sum_{j=1}^m \ww_j'\otimes \ww_j''$.

A complex structure on $S_0$ is, by definition, the structure of a nodal algebraic curve on it. Denote this algebraic curve by $X_0$. Each topological smoothing $S_\vv$ of $S_0$ corresponds to a first order smoothing $X_\vv$ of $X_0$ as an algebraic curve. It is useful to think of $S_\vv$ as the topological space underlying $X_\vv$. For each complex structure on $S_0$, each of the topological invariants of surfaces that we are considering: $\Q\pi_1(S_\vv,x)^\wedge$, $\Q\lambda(S_\vv)^\wedge$, $\g_{S_\vv,\partial S_\vv}$, \dots, has a natural {\em limit} mixed Hodge structure which depends non-trivially on the collection $\{\ww_j'\otimes\ww_j'' : j=1,\dots,m\}$. These limit MHS form a nilpotent orbit of MHS over $\prod_{j=1}^m \big(T_{p_j'}S\otimes T_{p_j''}S-\{0\}\big)$.

Such a limit MHS has two weight filtrations: $W_\bdot$, which we have already encountered, and the {\em relative weight filtration} $M_\bdot$, which is constructed from the weight filtration $W_\bdot$ and the action of the monodromy operator. The weight filtration of the limit MHS is $M_\bdot$, not $W_\bdot$. However, each $W_r$ is a sub MHS of the limit MHS. See \cite{hain:morita} for an exposition of relative weight filtrations written for topologists.

A morphism $S'_\vv \to S_\vv$ between topological smoothings of nodal surfaces is a map between the underlying topological spaces that is a generic inclusion. More precisely,
\begin{enumerate}

\item $S_0' = S_0 - \Sigma$, where $\Sigma$ is finite and may contain nodes of $S_0$,

\item the smoothings of $S_0'$ and $S_0$ agree at each node of $S_0'$. That is, if $\vu_j'\otimes\vu_j''$ is the vector corresponding to the $j$th node of $S_0'$ and $\ww_j'\otimes\ww_j''$ is the vector corresponding to its image in $S_0$, then $\vu_j'\otimes\vu_j''=\ww_j\otimes\ww_j''$.

\end{enumerate}
In this case, we write $\vv'$ for the sum of the $\vu_j'\otimes \vu_j''$ over the nodes of $S_0'$. This may not equal $\vv$ as the nodes of $S_0'$ may be a strict subset of the nodes of $S_0$.

With this notation, a morphism $X_{\vv'}' \to X_\vv$ between smoothings of nodal algebraic curves is a morphism $X_0' \to X_0$ of the algebraic curves together with a morphism $S_{\vv'}' \to S_\vv$ of the underlying topological spaces. Morphisms of smoothed nodal curves induce map morphisms of (limit) MHS on the standard invariants we are considering.

\begin{theorem}
\label{thm:glueing}
The homomorphism $\Q\lambda(X_{\vv'}')^\wedge \to \Q\lambda(X_\vv)^\wedge$ induced by a morphism $X'_{\vv'} \to X_\vv$ of smoothed nodal curves induces is a morphism of MHS. Moreover the functors $\Gr^M_\bdot$ and $\Gr^W_\bdot$ are both exact.
\end{theorem}


\subsection{Indexed pants decompositions}

An indexed pants decomposition of a surface is a pants decomposition in which there are only two ways to identify two boundary circles. In the next section, we will see that they correspond to very special smoothings of maximally degenerate nodal curves.

The starting point is to note that the Riemann sphere $\P^1 = \P^1(\C)$ has 6 canonical tangent vectors, namely $\partial/\partial z\in T_0 \P^1$ and its images under the canonical action of the symmetric group $\Sigma_3$ on $(\P^1,\{0,1,\infty\})$. An {\em indexed pair of pants} is the real oriented blow up
$$
\Phat := \Bl_{0,1,\infty} \P^1
$$
of $\P^1$ at $\{0,1,\infty\}$ together with the 6 boundary points (2 on each boundary circle) that correspond to the 6 canonical tangent vectors.
\begin{figure}[!ht]
\begin{tikzpicture}
\node[above right] at (0,0) {\includegraphics[width=4in]{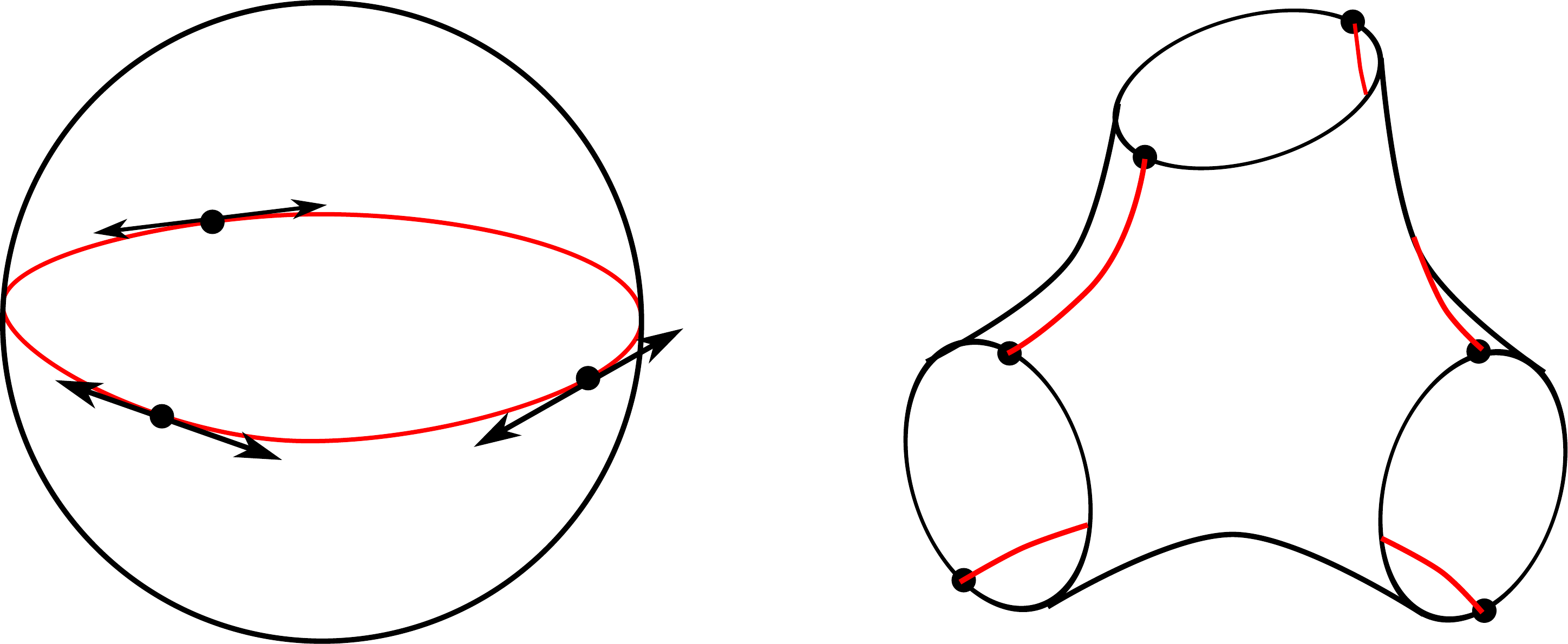}};
\node at (4.5,3.25) {$\P^1$};
\node at (3.25,3.1) {\color{red}$\P^1(\R)$};
\node at (9.75,3.25) {$\Phat$};
\node at (1.5,2.5) {$\infty$};
\node at (1.25,1.15) {$0$};
\node at (3.8,2.1) {$1$};
\end{tikzpicture}
\caption{The 6 canonical tangent vectors of $\P^1$ and $\Phat=\Bl_{0,1,\infty}\P^1$}
\end{figure}

An {\em indexed pants decomposition} of a surface $S$ is a decomposition of $S$ into indexed pairs of pants in which the indexing of adjacent pants match on each common boundary component.\footnote{Indexed pants decompositions correspond to isotopy classes of {\em quilt decompositions} of a pants decomposition as defined by Nakamura and Schneps in \cite{nakamura-schneps}.} The associated nodal curve $S_0$ is obtained from $S$ by collapsing the boundaries of all pairs of pants. Collapsed boundary circles should be regarded as marked points on $S_0$.

\subsection{Ihara curves}

A {\em maximally degenerate algebraic curve} of type $(g,n)$ is a nodal algebraic curve $X_0$, each of whose components is a copy of $\P^1$ and where the number of nodes plus the number of marked points on each component of $X_0$ is 3. Tangent vectors can be added at the marked points to obtain a maximally degenerate curve of type $(g,m+\vr)$, where $n=m+r$.

An {\em Ihara curve} of type $(g,n)$ is a first order smoothing $X_\vv$ of a maximally degenerate algebraic curve $X_0$ of type $(g,n)$ where
$$
\vv = \sum_{\substack{\text{nodes } p\cr \text{of }S_0}} \vv_p' \otimes \vv''_p \in \bigoplus_{\substack{\text{nodes } p\cr \text{of }S_0}} T_{p'}U'\otimes T_{p''}U''.
$$
Here $U'$ and $U''$ are the two analytic branches of $X_0$ at the node $p$; $p'\in U'$ and $p''\in U''$ are the preimages of $p$; $\vv'$ is a canonical tangent vector on the copy of $\P^1$ that contains $U'$, and $\vv''$ is a canonical tangent vector on the copy of $\P^1$ that contains $U''$. An Ihara curve of type $(g,m+\vr)$ is an Ihara curve of type $(g,m+r)$, with $r$ additional canonical tangent vectors added at $r$ distinct points of $X_0$ that are not nodes. Ihara curves correspond to indexed pants decompositions of a surface of type $(g,n)$.

Ihara and Nakamura \cite{ihara-nakamura} constructed a canonical formal smoothing
$\X/\Z[[q_p : p \text{ a node of }X_0]]$ of each maximally degenerate nodal curve $X_0$. These generalize Tate's elliptic curve in genus 1. The Ihara curve $X_\vv$ is the fiber of $\X$ over a tangent vector $\vv = \sum_p \pm \partial/\partial q_p$. For us, the importance of Ihara curves is that their invariants are mixed Tate motives, unramified over $\Z$.

\begin{theorem}[\cite{hain:ihara}]
\label{thm:ihara_curves}
If $(X,\vv)$ is an Ihara curve, then each of the invariants $\p(X,\vv)$, $\Q\lambda(X)$, $\gbar$, $\ghat$, \dots is a pro-object of $\MTM$ and therefore has a canonical $\k$ action. The action of $\k$ commutes with the Dehn twist on each vanishing cycle.
\end{theorem}

\begin{corollary}
\label{cor:decomp}
An indexed pants decomposition $X = \bigcup_{T\in \sP} T$ of an Ihara curve induces a Lie bialgebra homomorphism
$$
\bigoplus_{T\in\sP} \Q\lambda(T)^\wedge \to \Q\lambda(X)^\wedge
$$
whose kernel is spanned by the powers $\{\psi_n\log v: n \ge 1\}$ of the logarithms of the vanishing cycles $v$. It is a morphism of mixed Tate motives (and thus also, MHS). The homomorphism $\k \to \Q\lambda(X)^\wedge$ is ``decomposable'' in the sense that the diagram
$$
\xymatrix{
\k \ar[d]_{\oplus \phihat_T}\ar[dr]^{\phihat_X}\cr
\bigoplus_{T\in\sP} \Q\lambda(T)^\wedge \ar[r] & \Q\lambda(X)^\wedge
}
$$
commutes.
\end{corollary}

\section{Special Derivations and Edge Homomorphisms}
\label{sec:edge}

This section is quite long and technical. It is needed in for the proof of Theorem~\ref{thm:edge_sigma} and can be skipped if one is prepared to believe that result.

A major point of the works \cite{akkn:genus0,akkn2} of Alekseev--Kawazumi--Kuno--Naef is that much of the algebra of $\Gr^W_\bdot\Q\lambda(S)$ can be expressed in terms of several formal, non-commutative analogues of the classical divergence. This allows them to prove several key formulas for the Turaev cobracket on $\Gr^W_\bdot\Q\lambda(S)$. This section is a review of one aspect of their work useful for understanding the image of $\kappa_\vv : \ghat_{g,n+\uu} \to \Der^\theta \p(S)$. 

\subsection{Special derivations}
\label{sec:sder}

Suppose that $(S,\vv)$ is a surface of type $(g,n+\uu)$. Index the punctures by the integers $j$ with $0\le j \le n$. Choose the indexing so that the distinguished tangent vector $\vv$ is anchored at the 0th puncture. Let $\mu_j$ ($0\le j \le n$) be a circle that bounds a small disk centered at the $j$th puncture. For each $j\ge 1$, choose a path $\gamma_j$ in $S$ from $\vv$ to a point on $\mu_j$. Set $s_0 = \mu_0$, where we consider $\mu_0$ to be a loop based at $\vv$, and
$$
s_j := \gamma_j \mu_j \gamma_j^{-1} \in \pi_1(S,\vv) \text{ when } j = 1, \dots, n.
$$

Denote $\pi_1^\un(S,\vv)$ by $\cP$ and its Lie algebra by $\p$. Identify $s_j$ with its image in $\cP(\Q)$. Define the group of {\em special automorphisms} of $\cP$ by
$$
\SAut \cP = \{\phi \in \Aut \cP : \phi(s_0)=s_0 \text{ and } \phi(s_j) \sim s_j,\ j=1,\dots,n\},
$$
where $\sim$ denotes ``conjugate to''. This is an affine $\Q$-group which does not depend on the choice of the $\gamma_j$. The Lie algebra of $\SAut\cP$ is the Lie algebra
$$
\SDer \p := \{D \in \Der \p : D(\log s_0)=0,\ D(\log s_j) = [u_j,\log s_j] \text{ where } u_j \in \p,\ j>0\}
$$
of {\em special derivations} of $\p$. This is a Lie subalgebra of $\Der^\theta \p$.

\begin{proposition}
For each complex structure (possibly a first order smoothing) on $(S,\vv)$, there are canonical MHSs on $\SAut \p$ and $\SDer \cP$. 
\end{proposition}

\begin{proof}
We prove the assertion for $\SDer\p$, which we will need, and sketch the proof for $\SAut \cP$, which we will not. The canonical homomorphisms
$$
\p \to \Q\pi_1(S,\vv)^\wedge \to  \Q\lambda(S)^\wedge
$$
are morphisms of MHS. Each derivation of $\p$ induces a derivation of its enveloping algebra $\Q\pi_1(S,\vv)^\wedge$, and thus an endomorphism of $\Q\lambda(S)^\wedge$. The corresponding linear map
$$
\Der \p \to \End \Q\lambda(S)^\wedge
$$
is a morphism of MHS. Standard Hodge theory implies that each $\log|\mu_j|$ spans a copy of $\Q(1)$ in $\Q\lambda(S)^\wedge$. Consequently, the derivations of $\p$ that annihilate each $\log|\mu_j|$ form a sub MHS of $\Der \p$. Since $|\log s_j| = \log|\mu_j|$, $\SDer\p$ is the intersection of this Lie algebra with $\Der^\theta\p$. It follows that $\SDer\p$ is a sub MHS of $\Der\p$.

The corresponding statement for $\SAut\cP$ is proved similarly using the action of $\pi_1(\MHS)$ on $\cP$ and the fact that it acts on each $\log s_j$ via the canonical character $\pi_1(\MHS) \to \Aut \Q(1) \cong \Gm$. Just use the fact that the canonical maps $\cP \to \Q\pi_1(S,\vv)^\wedge \to \Q\lambda(S)^\wedge$ are $\pi_1(\MHS)$-equivariant.
\end{proof}

For each $j = 0,\dots,n$, set $z_j = \log s_j \in \Gr^W_{-2}\p$. Define
$$
\SDer \Gr^W_\bdot\p = \{D\in \Der\Gr^W_\bdot\p : \text{ for all } j,\ D(z_j) = [u_j,z_j] \text{ for some } u_j \in \Gr^W_\bdot \p, \}.
$$
As a consequence of the exactness properties of $\Gr^W_\bdot$ we have:

\begin{proposition}
For each complex structure (possibly a first order smoothing) on $(S,\vv)$ and each lift of the canonical central cocharacter $\chi : \Gm \to \pi_1(\MHS^\ss)$, there is a natural Lie algebra isomorphism
$$
\Gr^W_\bdot \SDer \p \cong \SDer \Gr^W_\bdot \p.
$$
\end{proposition}

The image of the natural homomorphism $\varphi_\vv : \cGhat_{g,n+\uu} \to \Aut \cP$ lies in $\SAut \cP$. This implies that the image of $\ghat_{g,n+\uu}$ lies in $\SDer\p$. For each complex structure on $(S,\vv)$ (possibly a smoothing), the homomorphisms $\phihat_\vv : \cGhat_{g,n+\uu} \to \SAut \cP$ and $d\phihat_\vv : \ghat_{g,n+\uu} \to \SDer\p$ are morphisms of MHS. We can therefore replace $\Der^\theta\p$ by $\SDer\p$ in the diagram (\ref{diag:grW}).

Similarly, one defines $\SDer \Q\pi_1(S,\vv)^\wedge$ and $\SDer \Gr^W_\bdot\Q\pi_1(S,\vv)$. Both have a natural MHS for each complex structure on $(S,\vv)$. The image of $\kappa_\vv$ is easily seen to lie in $\SDer \Q\pi_1(S,\vv)^\wedge$, so we will regard it as a homomorphism
$$
\kappa_\vv : \Q\lambda(S)^\wedge \to \SDer\Q\pi_1(S,\vv)^\wedge.
$$
It is a morphism of MHS for each choice of complex structure on $(S,\vv)$. One can thus replace it by its associated weight graded.

\subsection{The edge homomorphism}

As above, $(S,\vv,\xi)$ is a framed surface of type $(g,n+\uu)$. The {\em edge map} $\edg_\xi$ is defined to be the composite
$$
\xymatrix{
\Q\lambda(S)^\wedge \ar[r]^(.35){\delta_\xi} & \Q\lambda(S)^\wedge\comptensor \Q\lambda(S)^\wedge \ar[r]^(.6){\id\otimes \e} &  \Q\lambda(S)^\wedge ,
}
$$
where $\epsilon : \Q\lambda(S)^\wedge \to \Q$ is the map induced by the augmentation. It is a morphism of MHS for each choice of complex structure on $S$. Its dependence on the framing is easily determined.

Suppose that $\xi_1$ and $\xi_0$ are two framings of $S$ and that $\xi_1 - \xi_0 = \phi \in H^1(S;\Z)$. Set
$
\edg_\phi(\gamma) := \edg_{\xi_1}(\gamma) - \edg_{\xi_0}(\gamma).
$

\begin{lemma}
\label{lem:edge}
If $\xi_1 - \xi_0 = \phi \in H^1(S)$, then on $\Q\lambda(S)^\wedge$, we have
$$
\edg_\phi \circ \psi_n = n\, \psi_n \circ \edg_\phi.
$$
Consequently, $\edg_\phi$ maps the subspace $|\Sym^k\p|$ of $\Q\lambda(S)^\wedge$ into $|\Sym^{k-1}\p|$.
\end{lemma}

\begin{proof}
For all immersed loops $\gamma$ in $S$ we have
\begin{equation}
\label{eqn:phi}
\edg_\phi(\gamma) = \edg_{\xi_1}(\gamma) - \edg_{\xi_0}(\gamma) = \big(\rot_{\xi_1}(\gamma) - \rot_{\xi_0}(\gamma)\big)\gamma = \phi(\gamma) \gamma.
\end{equation}
This implies that
$$
\edg_\phi (\psi_n(\gamma)) = \phi(\psi_n(\gamma)) \psi_n(\gamma) =  n\phi(\gamma)\psi_n(\gamma) = n \psi_n(\edg_\phi(\gamma))
$$
for all $\gamma\in \lambda(S)$. Since $\edg_\phi$ is continuous, this formula also holds in $\Q\lambda(S)^\wedge$. The second assertion follows from this and Corollary~\ref{cor:pbw}.
\end{proof}

\begin{corollary}
\label{cor:edg_square}
If $u,v\in \p$, then
$$
\edg_\phi |uv| = \big(\phi(u)|v| + \phi(v)|u|\big)/2 \in H_1(S).
$$
\end{corollary}

\begin{proof}
First note that there is a canonical isomorphism $|\p| \cong H_1(S)$, so Lemma~\ref{lem:edge} implies that $\edg_\phi : |\Sym^2\p| \to H_1(S)$. It also implies that
\begin{multline*}
\edg_\phi|e^u e^v| = \edg_\phi\big(1 + |u+v| + |uv + u^2/2 + v^2/2| + \cdots \big)
\cr
= \phi(u+v)\big(1 + |u+v| + |uv + u^2/2 + v^2/2| + \cdots \big)
\end{multline*}
which implies that
$$
\edg_\phi|uv + u^2/2 + v^2/2| = \phi(u+v)|u+v|.
$$
The lemma also implies that $\edg_\phi|u^2/2 + v^2/2| = \phi(u)|u| + \phi(v)|v|$. Together these imply that $\edg_\phi|uv| = \phi(u)|v| + \phi(v)|u|$.
\end{proof}

\subsection{Genus 0 setup}

Suppose that $n\ge 2$ and $S=\P^1-\{p_0,\dots,p_n\}$. We use the notation of equation (\ref{eqn:punct_P1}), where $p_0 = \infty$. In particular, $H_1(S)$ is spanned by $\ee_0,\dots,\ee_n$, with the single relation $\ee_0 + \dots + \ee_n = 0$. There is a vector space inclusion
$$
\SDer\Gr^W_\bdot \p \cong \SDer \big(\L(\ee_0,\dots,\ee_n)/(\ee_0 + \dots + \ee_n)\big) \hookrightarrow \L(H_1(S))^{n+1}.
$$
It takes the special derivation $D$ of $\Gr^W_\bdot \p$ of weight $-2m$ to the vector
$$
\bu = (u_1,\dots,u_n) \in \big(\Gr^W_{-2m+2}\p\big)^n
$$
where $D(\ee_j) = [u_j,\ee_j]$.\footnote{When $m\neq 1$, $\bu$ is uniquely determined by $D$. When $m=1$, $\bu$ is uniquely determined if we insist that each $u_j$ be orthogonal to $\ee_j$ with respect to the natural inner product on $H_1(S)$ with orthonormal basis $\ee_1,\dots,\ee_n$.} The image consists of those $\bu$ satisfying
\begin{equation}
\label{eqn:spec_reln}
[u_1,\ee_1] + \dots + [u_n,\ee_n] =  -D(\ee_0) = 0.
\end{equation}
Denote this derivation by $D_\bu$. Observe that
$$
[D_\bu,D_\bv] = D_\bw \text{ where } w_j = D_\bu(v_j) - D_\bv(u_j) - [u_j,v_j],\ j=1,\dots,n.
$$

Similarly, the Lie algebra of special derivations of
$$
\Gr^W_\bdot\Q\pi_1(S,\vv) \cong \Q\langle \ee_0,\dots,\ee_n\rangle/(\ee_0+\dots+\ee_n) \cong \Q\langle \ee_1,\dots,\ee_n\rangle
$$
is the Lie algebra of derivations that satisfy
$$
D_\bu  : \ee_j \to [u_j,\ee_j] := u_j\ee_j - \ee_j u_j,\quad j=1,\dots,n
$$
where each $u_j \in \Q\langle \ee_0,\dots,\ee_n\rangle/(\ee_0+\dots+\ee_n)$, $u_0 = 0$ and (\ref{eqn:spec_reln}) holds.

Alekseev, Kawazumi, Kuno and Naef \cite{akkn:genus0} show that, in genus 0, the graded Goldman--Turaev Lie bialgebra is essentially isomorphic to $\SDer\Gr^W_\bdot \Q\pi_1(S,\vv)$.

\begin{proposition}[{\cite[Lem.~8.3]{akkn:genus0}}]
\label{prop:SDer}
If $(S,\vv)$ is a surface of type $(0,n+\uu)$, then the Lie algebra homomorphism
$$
\kappa_\vv: \Gr^W_\bdot \Q\lambda(S) \to \SDer \Gr^W_\bdot \Q\pi_1(S,\vv)
$$
is surjective with kernel spanned by $\{|\ee_j^k|:k\ge 0,j=0,\dots,n\}$. The isomorphism $\kappa_\vv : \Gr^W_\bdot I\Q\lambda(S) \to \SDer\Gr^W_\bdot\Q\pi_1(S,\vv)$ restricts to an isomorphism
$$
\SDer \Gr^W_\bdot \p \cong \sum_{j=1}^n |\ee_j\L(\ee_1,\dots,\ee_n)|/|\ee_j^2| \subset \Gr^W_\bdot |\Sym^2\p|/\langle |\ee_j^2|: j=1,\dots n \rangle.
$$
Under this isomorphism, $D_\bu \in \SDer\Gr^W_\bdot \p$ corresponds to $\sum_{j=1}^n |\ee_j u_j| \in |\Sym^2\Gr^W_\bdot\p|$. Consequently, for each framing $\xi$ of $S$, the edge maps descends to a map
$$
\edg_\xi : \SDer \Gr^W_\bdot \p \to \Gr^W_\bdot \Q\lambda(S).
$$
\end{proposition}

\begin{remark}
Since $\kappa_\vv$ is a morphism of MHS, this result implies that for all complex structures on $(S,\vv)$, and all (necessarily algebraic) framings $\xi$, the cobracket descends to a MHS morphism $\delta_\xi : \SDer \p \to \Q\lambda(S)^\wedge$. It also implies that there is a canonical isomorphism
$$
\kappa_0 : \bigoplus_{j=1}^n |(\log\mu_j) \p_j|/|(\log\mu_j)^2| \overset{\simeq}{\To} \SDer\p.
$$
where $\p_j$ denotes the Lie algebra of $\pi_1^\un(S,\vv_j)$, where each $\vv_j \in T_{p_j}\P^1$ is non-zero. Note that $\p_0 = \p$.
\end{remark}

\begin{corollary}
\label{cor:edge_phi}
Fix a complex structure on $(S,\vv)$. If $\xi_0$ and $\xi_1$ are two (necessarily algebraic) framings of $S$ that differ by $\phi\in H^1(S)$, then $\edg_\phi : \SDer\p \to H_1(S)$ is a morphism of MHS. The induced map on associated weight gradeds is
$$
\edg_\phi : D_\bu \mapsto \sum_{j=1}^n \phi(\ee_j) |u_j|.
$$
\end{corollary}

\begin{proof}
The first assertion follows from Proposition~\ref{prop:SDer} and Lemma~\ref{lem:edge}. That $\edg_\phi$ is a morphism of MHS follows from the fact that $\delta_\xi$ is a morphisms of MHS for all algebraic framings (Section~\ref{sec:gt_hodge}) and, in genus 0, every framing is homotopic to an algebraic framing.

Since $D_\bu$ corresponds to $\sum_{j=1}^n |\ee_j u_j|$ with $\sum_j (\ee_j u_j - u_j \ee_j) = 0$, and since $|u_j|$ and $\phi(u_j)$ vanish if $u_j \in [\p,\p]$, Corollary~\ref{cor:edg_square} implies that
$$
\edg_\phi : D_\bu \mapsto \sum_{j=1}^n \big(\phi(u_j)|\ee_j| + \phi(\ee_j)|u_j|\big)/2 = \sum_{j=1}^n \phi(\ee_j)|u_j|.
$$
\end{proof}

\subsection{The divergence cocycle in genus $0$}

For each surface $(S,\vv)$ of type $(0,n+\uu)$, Alekseev, Kawazumi, Kuno and Naef construct \cite[\S3.2]{akkn:genus0} a Lie algebra 1-cocycle
$$
\div_0 : \SDer \Gr^W_\bdot\p \to \Gr^W_\bdot\Q\lambda(S).
$$
Write $u_j = \sum_{k=1}^n \ee_k  u_j^{(k)}\in \Q\langle \ee_1,\dots,\ee_n\rangle$. Then
\begin{equation}
\label{eqn:div}
\div_0 D_\bu := \sum_{j=1}^n |\ee_j u^{(j)}_j|.
\end{equation}
For each framing $\xi$ of $S$, define $\br_{0,\xi} : \SDer\Gr^W_\bdot\p \to \Gr^W_\bdot \Q\lambda(S)$ by
$$
\br_{0,\xi} : D_\bu \mapsto \sum_{j=1}^n \rot_\xi(\mu_j)|u_j| \in |\p| \cong H_1(S).
$$
It is a homomorphism $\SDer\Gr^W_\bdot\p \to H_1(S)$. The subscript 0 in $\div_0$ and $\br_{0,\xi}$ indicates the distinguished boundary component of $S$. Note that $\p = \p_0$.

\begin{theorem}[Alekseev--Kawazumi--Kuno--Naef]
\label{thm:div}
For each framed surface $(S,\vv,\xi)$ of type $(0,n+\uu)$, the restriction of $\edg_\xi$ to $\SDer\Gr^W_\bdot \p$ is a 1-cocycle equal to $\div_0+\br_{0,\xi}$. More precisely, the diagram
$$
\xymatrix{
\Gr^W_\bdot\Q\lambda(S)\ar[d]_{\kappa_\vv} \ar[r]^{\edg_\xi} & \Gr^W_\bdot\Q\lambda(S) \cr
\SDer\Gr^W_\bdot\Q\pi_1(S,\vv) & \ar@{_{(}->}[l] \SDer\Gr^W_\bdot\p \ar[u]_{\div_0+\br_{0,\xi}} \ar@{_{(}->}[ul]_{\kappa_\vv^{-1}} 
}
$$
commutes. Consequently, the restriction of $\edg_\xi : \SDer \p \to \Q\lambda(S)^\wedge$ to $W_{-3}\SDer \p$ does not depend on the choice of the framing $\xi$.
\end{theorem}

\begin{proof}[Remarks on the Proof]
First observe that the cocycle $c$ of \cite[Lem.~3.2]{akkn:genus0} is $\br_{0,\xi_0}$, where $\xi_0 := \partial/\partial z$ is the ``blackboard framing'' of $S$. Note that the cobracket $\delta^+$ in \cite{akkn:genus0} denotes $\delta_{\xi_0}$.

The result holds for the blackboard framing. This is \cite[Prop.~6.7]{akkn:genus0}, which follows from \cite[Prop.~3.5]{akkn:genus0} and the discussion at the beginning of \cite[\S5.3]{akkn:genus0}. To deduce the result for all framings $\xi$, note that Corollary~\ref{cor:edge_phi} implies that
$$
\br_{0,\xi}(D_\bu) - \br_{0,\xi_0}(D_\bu) = \sum_{j=1}^n \phi(\mu_j)|u_j| = \edg_\xi(D_\bu) - \edg_{\xi_0}(D_\bu),
$$
where $\xi - \xi_0 = \phi$. Finally, the last statement follows as $\br_{0,\xi}$ vanishes on $W_{-3}$ as $W_{-3}H_1(S)=0$.
\end{proof}

\begin{remark}
\label{remark:reconstruction}
At least when $S$ has genus 0, \cite[Prop.~3.5]{akkn:genus0} implies that the restriction
$$
\delta_\xi : \SDer \p \to \Q\lambda(S)^\wedge\comptensor \Q\lambda(S)^\wedge
$$
of the cobracket to $\SDer\p$ is determined by the edge homomorphism $\edg_\xi$. 
\end{remark}

\begin{problem}
Give a clear conceptual account of the results of this section for algebraic curves of type $(g,n+\uu)$ with a quasi-algebraic framing for all $g\ge 0$.
\end{problem}

\section{Constraints on the Arithmetic Johnson Image}
\label{sec:johnson}

In this section $(S,\vv)$ is a hyperbolic surface of type $(g,n+\uu)$ with framing $\xi$. Set $\pi = \pi_1(S,\vv)$, $\p = \p(S,\vv)$, $\gbar=\gbar_{g,n+\uu}$ and $\ghat=\ghat_{g,n+\uu}$, etc. We explain how the Turaev cobracket gives an upper bound on the image of the arithmetic Johnson image $\ghat$ in $\Der^\theta \p$. The solution of the Oda Conjecture, Theorem~\ref{thm:arith_extn}, implies that this also gives a bound on the size of the geometric Johnson image $\gbar$, if not its precise location in $\Der^\theta\p$.

Fix a complex structure $\phi$ on $(\Sbar,\vv,\xi)$ and a lift $\chitilde : \Gm \to \pi_1(\MHS)$. These fix natural isomorphisms between a MHS and its associated weight graded quotients. In particular, they fix natural isomorphisms
$$
\p \cong \big(\Gr^W_\bdot\p\big)^\wedge,\ \ghat \cong \big(\Gr^W_\bdot\ghat\big)^\wedge,\ \Q\lambda(S)^\wedge \cong \big(\Gr^W_\bdot \Q\lambda(S)\big)^\wedge, \dots
$$
compatible with their respective algebraic structures (Lie algebra, Lie bialgebra, \dots). This means that all results in this section apply equally to the proalgebraic objects (such as $\ghat$, $\Der^\theta\p$) and their associated weight graded objects.

Denote the normalizer \label{def:normalizer} of the geometric Johnson image $\gbar$ in $\Der^\theta\p$ by $\n$.\footnote{Note that this is the normalizer of $\gbar$. This is strictly smaller than the normalizer of $\ubar$.} Observe that $\ghat$ is contained in $\n$ as $\gbar$ is an ideal of $\ghat$. Since the adjoint action of $\gbar$ on $\n/\gbar$ is trivial, the weight graded Lie algebra
$$
\Gr^W_\bdot\n/\Gr^W_\bdot \gbar
$$
is a trivial $\sp(H)$-module. The action $\m^\phi \to \Der^\theta \p$ of the Mumford--Tate Lie algebra induces a homomorphism
\begin{equation}
\label{eqn:k_to_n_mod-g}
\k \to \n/\gbar
\end{equation}
which is injective by the solution of the Oda Conjecture. This implies that $\n/\gbar$ is a $\k$-module and thus a pro-object of $\MTM$. (In fact, injectivity implies that $\n/\gbar$ generates $\MTM$.) Since $\gbar$ acts trivially on $\n/\gbar$, the Theorem of the Fixed Part\footnote{This states \cite[\S4]{vmhs1} that if $\V$ is an admissible VMHS over a smooth variety $X$, then $H^0(X,\V)$ has a natural MHS and that for each $x\in X$, the inclusion $H^0(X,\V) \hookrightarrow V_x$ is a morphism of MHS.} implies that the MHS on $\n/\gbar$ does not depend on the complex structure $\phi$.

Denote the stabilizer of $\xi$ in $\ghat$ by $\ghat^\xi$. \label{def:g^xi} Note that $\ghat^\xi$ stabilizes $\delta_\xi$ in the sense that
$$
\delta_\xi(\{u,v\}) = u\cdot \delta_\xi(v)
$$
for all $u\in \ghat^\xi$ and all $v\in \Q\lambda(S)^\wedge$.

\begin{lemma}
\label{lem:invariance}
Suppose that $L$ is a Lie bialgebra with cobracket $\delta$. If $\h$ is a Lie subalgebra of $L$ that preserves the bracket in the sense that
$$
\delta[h,u] = h\cdot \delta u \text{ for all } u\in L, h\in \h,
$$
then $[\h,\h]\subseteq \ker \delta$.
\end{lemma}

\begin{proof}
Since $\delta$ is $\h$-invariant, $\delta([h',h'']) = h'\cdot\delta h'' = -h''\cdot \delta h'$ for all $h',h''\in\h$. On the other hand, since $L$ is a Lie bialgebra, $\delta([h',h'']) = h'\cdot\delta h'' -h''\cdot \delta h'$. Combining these two statement, we see that $\delta([h',h'']) = 0$ all $h',h''\in \h$.
\end{proof}

Denote the pronilpotent radical of $\m^\phi$ by $\u^\MT_\phi$. \label{def:u^MT} Since $\delta_\xi$ acts by a Tate twist, it is not invariant under all of $\pi_1(\MTM)$. It is, however, invariant under its prounipotent radical, and therefore under $\u^\MT_\phi$. Denote the stabilizer of $\xi$ in $\gbar$ by $\gbar^\xi$. It is the image of $\g^\xi$.

\begin{corollary}
\label{cor:H_1(k)}
If $g\ge 1$, then $[\u^\MT_\phi,\u^\MT_\phi] + \gbar^\xi \subseteq \ker \delta_\xi$. Consequently, the homomorphism (\ref{eqn:k_to_n_mod-g}) induces a homomorphism
\begin{equation}
\label{eqn:abelianization}
H_1(\k) \to \Q\lambda(S)^\wedge/\ker\delta_\xi
\end{equation}
which is a morphism of MHS for all complex structures on $(S,\vv)$.
\end{corollary}

\begin{proof}
The cobracket is a morphism of MHS for each complex structure $\phi$. This implies that it is $\u^\MT_\phi$ invariant. Proposition~\ref{prop:lift} implies that there is an inclusion $\u^\MT_\phi \hookrightarrow \Q\lambda(S)^\wedge$ and that it is a morphism of MHS. The first assertion follows as Lemma~\ref{lem:invariance} implies that $[\u^\MT_\phi,\u^\MT_\phi]\subseteq \ker \delta_\xi$. It also implies that the inclusion of $\u^\MT_\phi$ into $\Q\lambda(S)^\wedge$ induces a homomorphism
\begin{equation}
\label{eqn:arith_monod}
H_1(\u^\MT_\phi) \to \Q\lambda(S)^\wedge/\ker \delta_\xi.
\end{equation}
There is a canonical canonical projection $\u^\MT_\phi \to \k$ for all complex structures $\phi$. Corollary~\ref{cor:H_1(k)} implies that (\ref{eqn:arith_monod}) factors through $H_1(\u^\MT_\phi) \to H_1(\k)$.
\end{proof}

Since $\k = W_{-6}\k$, Theorem~\ref{thm:div} implies that $\edg_\xi^S$ is just the divergence and so does not depend on the choice of a framing.

\begin{corollary}
\label{cor:edg_on_k}
The composition of the canonical homomorphism $\phihat$ of Proposition~\ref{prop:lift} with the edge homomorphism induces a homomorphism
\begin{equation}
\label{eqn:edg_k}
\edg^S : H_1(\k) \to \Q\lambda(S)^\wedge
\end{equation}
that is a morphism of MHS for all complex structures on $S$. It is invariant under the mapping class group of $S$ and does not depend on the complex structure on $S$ or on the framing $\xi$. 
\end{corollary}

\subsection{Decomposability}

Our goal in this section is to prove that $\edg^S$ is injective for all hyperbolic surfaces $S$. The first step is to prove that $\edg^S$ behaves well when $S$ is decomposed.

\begin{lemma}
\label{lem:decomp}
If $S=S'\cup S''$ is a decomposition of $S$ into two closed subsurfaces, then 
$$
\edg^S = \edg^{S'} + \edg^{S''} \in \Q\lambda(S')^\wedge + \Q\lambda(S'')^\wedge \subset \Q\lambda(S)^\wedge.
$$
Furthermore, if $S$ has genus 0, then the image of (\ref{eqn:edg_k}) is contained in the center of $\Q\lambda(S)^\wedge$.
\end{lemma}

\begin{proof}
Since $\edg^S$ does not depend on the complex structure on $S$, we can take $S$ to be an Ihara curve. We can further assume that $S=S'\cup S''$ is a decomposition of $S$ into two Ihara curves. The first statement is then a direct consequence of Corollary~\ref{cor:decomp}.

To prove the second, observe that since the cobracket is $\k$-linear, we have
$$
\sigma\cdot\delta_\xi(u) = \delta_\xi (\{\sigma,u\}) = \sigma\cdot \delta_\xi(u) - u\cdot\delta_\xi(\sigma),
$$
for all $u\in \Q\lambda(S)^\wedge$ and $\sigma \in \k$. This implies that $u\cdot \delta_\xi(\sigma) = 0$ for all $u$ and $\sigma$. Since $\delta_\xi$ increases weights by 2, and since $S$ has genus 0, we have
$$
(\id\otimes\e)(u\cdot\delta_\xi(\sigma)) = u \cdot (\id\otimes\e)(\delta_\xi(\sigma)).
$$
Combining these, we see that $u\cdot\edg_\xi(\sigma) = 0$ for all $u\in \Q\lambda(S)^\wedge$.
\end{proof}

Recall that $\mu_j$, $j\in\{0,\dots,n\}$, is a small positive loop about the $j$th puncture. For each $k > 0$
\begin{equation}
\label{eqn:decomp}
\sum_{j=0}^n (\log \mu_j)^k \in \Q\lambda(S)^\wedge.
\end{equation}
Note that this is a Hodge class of weight $-2k$ for all complex structures on $S$.

The weight graded version of the following result is key. It is proved in \cite{ceg} and reproved in \cite[Thm.5.4]{akkn2}.

\begin{proposition}
\label{prop:center}
These elements span the center of $\Q\lambda(S)^\wedge$.
\end{proposition}

\begin{proof}
The statement for the completed version follows from Hodge theory as the exactness of $\Gr^W_\bdot$ implies that
$$
Z\Gr^W_\bdot \Q\lambda(S)^\wedge \cong \Gr^W_\bdot Z\Q\lambda(S)^\wedge
$$
where $Z$ denotes center.
\end{proof}

\subsection{Injectivity of $\edg^S$}

The abelianization of $\k$ is the semi-simple pro-Hodge structure
$$
H_1(\k) = \prod_{m\ge 1} \Q[\sigma_{2m+1}] = \prod_{m\ge 1} \Q(2m+1).
$$

The following theorem is essentially due to Alekseev, Kawazumi, Kuno and Naef \cite{akkn2}. It was explained to me by Florian Naef. A precursor is \cite[Prop.~4.10]{alekseev-torossian} of Alekseev and Torossian.

\begin{theorem}
\label{thm:edge_sigma}
If $S$ is a hyperbolic surface, then the image of the generator $[\sigma_{2m+1}]$ of $H_1(\k)$ (suitably normalized) under the homomorphism (\ref{eqn:edg_k}) satisfies
$$
\edg^S(\sigma_{2m+1}) = \sum_{j=0}^n (\log \mu_j)^{2m+1} \in \Q\lambda(S)^\wedge.
$$
In particular, $\edg_S$ is injective for all hyperbolic surfaces $S$, so that (\ref{eqn:abelianization}) is also injective.
\end{theorem}

The proof is somewhat technical and is relegated to Appendix~\ref{sec:proof_edge}.

\begin{corollary}
\label{cor:im_ghat}
If $g\ge 1$ and $S$ is a surface of type $(g,\uu)$, then the cobracket induces an isomorphism $\ghat/(\ghat\cap \ker\delta_\xi) \overset{\simeq}{\To} H_1(S) \oplus H_1(\k)$.
\end{corollary}

\subsection{Concluding remarks}
\label{sec:ker_delta}
For surfaces of type $(g,\uu)$ with $g>0$, Enomoto and Satoh \cite{enomoto-satoh} constructed $\Sp(H)$-invariant trace maps
$$
\Tr_\ES : \Gr^W_{-m}\Der^\theta \p \to |H^{\otimes m}| \quad  m\ge 1
$$
that generalize Morita's trace maps and vanish on $\Gr^W_\bdot \gbar$. In \cite[\S8.2]{akkn2}, it is shown that $\Tr_\ES$ is the weight graded of the edge map defined in Section~\ref{sec:edge}. The trace map in genus 0 is the divergence $\div_0$. The ``reconstruction formula'' mentioned in Remark~\ref{remark:reconstruction} then implies that the kernels of $\Tr_\ES$ and $\delta_\xi$ on $\Der^\theta\L(H)$ are identical.

In another paper \cite{eks} with Kuno, they consider the kernel of the {\em reduced} cobracket $\deltabar$ on $\Der^\theta\L(H)$. They show that
$$
\Gr^W_{-m}\gbar \subseteq \ker \Tr_\ES \subseteq \ker \Gr^W_{-m} \deltabar \cap \Der^\theta\L(H)
$$
when $2 \le m\le 2g-2$, so these inclusions hold stably. They also show \cite[Thm.~2]{eks} that
$$
(\Gr^W_\bdot\deltabar \cap \Der^\theta\L(H))/\ker \Tr_\ES
$$
contains a non-trivial $\Sp(H)$-module in weight $-8$ and when $m > 5$ is congruent to 1 mod 5. Since $\Gr^W_\bdot(\n/\gbar)$ is a trivial $\Sp(H)$-module, this implies that $\ker \deltabar\cap \Der^\theta\L(H)$ does not normalize $\gbar$.

In $(g,\uu)$ case, we have the following diagram in $\Der^\theta\L(H)$:
$$
\xymatrix@C=36pt{
\Der^\theta\L(H) \cap \ker \delta_\xi \ar@{-}[dr]& & \n \ar@{-}[dl]\ar@{-}[d] \cr
& \n  \cap \ker \delta_\xi  \ar@{-}[d] & \ghat \ar@{-}[dl]\cr
& \ghat^\xi 
}
$$
This raises the following question, which is related to Drinfeld's GRT, \cite{drinfeld}.

\begin{question}
\label{quest:ker_delta}
Is $\k = (\ker\delta_\xi \cap \Der^\theta\p)/\gbar^\xi$? Equivalently, is $\ghat^\xi = \ker\delta_\xi\cap \Der^\theta\p$?
\end{question}

\section{Remarks on a Conjecture of Morita}

Suppose that $(S,\vv)$ is a surface of type $(g,\uu)$ with an indexed pants decomposition. As explained in Section~\ref{sec:decompositions}, such a decomposition corresponds to an Ihara curve and thus determines an action of the motivic Lie algebra $\mtm$ on $\p(S,\vv)$.  Morita \cite{morita:galois} has proposed a location of the image of the generator $\sigma_{2n+1}$ of $\k$ in $\Gr^W_{-4n-2}\Der^\theta\p(S,\vv)$. Here we discuss Morita's proposal. As in the previous section, we set $\p = \p(S,\vv)$, $\gbar=\gbar_{g,\uu}$ and $\ghat=\ghat_{g,\uu}$.

Since these generators $\sigma_{2n+1}$ of $\k$ are not canonical (see Remark~\ref{rem:sigmas}), the best we can hope for is that we can determine the image of each $\sigma_{2n+1}$ mod $[\k,\k]$. Since the image of $\k\to \Der^\theta \p$ depends on the pants decomposition, the best we can hope for {\em if we ignore the pants decomposition} is to determine the image of $\sigma_{2n+1}$ mod $[\k,\k]+\gbar$. This is because the isomorphism $\ghat \cong \mtm\ltimes \gbar$ depends non-trivially on the pants decomposition and because the image of $\ghat$ in $\Der^\theta \p$ does not depend on the complex structure by Theorem~\ref{thm:arith_extn}. Finally, since the image of $\k$ is contained in $W_{-6}\Der^\theta\p$, and since $W_{-2}\gbar = W_{-2}\ghat^\xi$ by Corollary~\ref{cor:H_1(k)}, the image of $\sigma_{2n+1}$ is well defined mod $\ker \delta$. Morita's starting point is the following observation of Nakamura (unpublished).

\begin{lemma}
For all $g\ge 2$ and $n\ge 0$, there is unique copy of $\Sym^{2n+1}H$ in $\Gr^W_{-2n-1} \Der^\theta \p$. The restriction of Morita's trace map to it is an isomorphism.
\end{lemma}

\begin{proof}[Sketch of Proof]
Morita's trace map implies that there is at least one copy of $\Sym^{2n+1}H$ in $\Gr^W_{-2n-1} \Der^\theta \L(H)$. We need to prove uniqueness. Recall from Remark~\ref{rem:graphical_rep} that elements of $\Der^\theta\L(H)$ of weight $-2n-1$ are represented by planar trivalent graphs with $2n+1$ vertices modulo the IHX relation. The unique copy of $\Sym^{2n+1}H$ is the image of the map that takes $x_0x_1\dots x_{2n} \in \Sym^{2n+1}H$ to
$$
\sum_{j=1}^g \sum_{\sigma \in \Sigma_{2n+1}}
\parbox{2.75in}{
\begin{tikzpicture}[scale=.75]
\draw (-1,0) -- (7,0); 
\draw (0,0) -- (0,0.75);
\draw (1,0) -- (1,0.75);
\draw (2,0) -- (2,0.75);
\draw (5,0) -- (5,0.75);
\draw (6,0) -- (6,0.75);
\draw[fill] (0,0) circle [radius = .05] ;
\draw[fill] (1,0) circle [radius = .05] ;
\draw[fill] (2,0) circle [radius = .05] ;
\draw[fill] (5,0) circle [radius = .05] ;
\draw[fill] (6,0) circle [radius = .05] ;
\node[above] at (0,0.75) {$x_{\sigma(0)}$} ;
\node[above] at (1,0.75) {$x_{\sigma(1)}$} ;
\node[above] at (2,0.75) {$x_{\sigma(2)}$} ;
\node[above] at (6,0.75) {$x_{\sigma(2n)}$} ;
\node[above] at (3.5,0.5) {$\cdots$} ;
\node[left] at (-1,0) {$a_j$} ;
\node[right] at (7,0) {$b_j$} ;
\end{tikzpicture}
}
\in \Gr^W_{-2n-1}\Der^\theta \L(H).
$$
\end{proof}

The copy of $H$ in $\Gr^W_{-1}\p$ lies in the image of $\Gr^W_{-1}\g_{g,\uu}$ in $\Der^\theta\p$ and is thus ``geometric''. Morita \cite{morita:homom} has shown that $\Sym^{2n+1}H$ is not geometric for all $n\ge 1$.

Fix an $\Sp(H)$ equivariant inclusion $\mu_{2n+1} : \Sym^{2n+1}H \to \Der^\theta\Gr^W_\bdot\p$. Since $\Sym^{2n+1}H$ is an irreducible $\Sp(H)$-module, there is a unique copy of the trivial representation in $\Lambda^2 \Sym^{2n+1}H$. Denote the image of a generator of this under the bracket
$$
[\phantom{x},\phantom{x}] : \big[\Lambda^2 \Sym^{2n+1}H\big]^{\Sp(H)} \to \Gr^W_{-4n-2}\Der^\theta \p
$$
by $\mu_{2n+1}^2$. The following is a refinement of Morita's proposal \cite{morita:galois} for the image of the $\sigma_{2n+1}$.

\begin{question}
\label{quest:mu2}
Is it true that for each $n\ge 1$, after rescaling $\mu_{2n+1}^2$ if necessary,
$$
\sigma_{2n+1} \equiv \mu_{2n+1}^2  \bmod \ker\delta_\xi.
$$
\end{question}

An affirmative answer to Questions~\ref{quest:ker_delta} and \ref{quest:mu2} would imply that in $\Gr^W_{-4n-2}\Der^\theta \p$, we have
$$
\sigma_{2n+1} \equiv \mu_{2n+1}^2 \bmod [\k,\k] + \gbar.
$$
for each indexed pants decomposition of $S$. This is the best that one can hope for without specifying the structure of an Ihara curve on $S$. A negative answer to Question~\ref{quest:mu2} would imply that this is false.

\part{Stable Cohomology and the Johnson Image}
\label{part:stab_coho}

Another approach to computing the graded Johnson image as a graded representation of $\Sp(H)$ is to exploit the relationship between the Lie algebra $\Gr^W_\bdot\u_{g,n+\vr}$ and its cohomology. Enough information about the cohomology of $\Gr^W_\bdot\u_{g,n+\vr}$ will determine it, and thus weight graded quotients of the geometric Johnson homomorphism if the geometric Johnson homomorphism is injective. Otherwise it will give an upper bound.

The starting point is that exactness of the weight filtration implies that
$$
\Gr^W_\bdot H^\bdot(\u_{g,n+\vr}) \cong H^\bdot(\Gr^W_\bdot \u_{g,n+\vr}).
$$
The Chevalley--Eilenberg complex of cochains on $\Gr^W_\bdot \u_{g,n+\vr}$ is a complex in the category of $\Sp(H)$-modules bigraded by weight and degree. The differential preserves weights. The fact that the Euler characteristic of the cochain complex equals the Euler characteristic of its cohomology in the representation ring of $\Sp(H)$ relates the first $m$ weight graded quotients of $\u_{g,n+\vr}$ to the Euler characteristic
\begin{equation}
\label{eqn:euler}
\sum_{j=0}^m (-1)^j [\Gr^W_m H^j(\u_{g,n+\vr})] t^j \in R(\Sp(H))[t],
\end{equation}
of the weight $m$ graded quotient of its cohomology, where $\cR(\Sp(H))$ is the representation ring of $\Sp(H)$. If one knows this for all $m \le M$, then one can inductively compute $\Gr^W_{-m}\u_{g,n+\vr}$ in $\cR(\Sp(H))$ for all $1\le m \le M$. A closed formula for the graded quotients is given by M\"obius inversion, as explained in Appendix~\ref{sec:inversion}.

This program is probably tough to carry out unstably. However, if one knows that $H^m(\u_{g,n+\vr})$ is stably of weight $m$, then the left side of (\ref{eqn:euler}) simplifies to the stable value of $(-1)^m [H^m(\u_{g,n+\vr})]$. In Section~\ref{sec:coho} we explain that the stable value of $\Gr^W_m H^m(\u_{g,\uu})$ is known by the work of Garoufalidis--Getzler, Petersen and Kupers--Randal-Williams. To complete the computation of $\Gr^W_\bdot\u_{g,\uu}$, one needs to know that the stable value of $H^m(\u_{g,\uu})$ is pure of weight $m$. This has been established in degrees $\le 6$ with the most recent contributions in \cite{morita-sakasai-suzuki}.

Basic facts about the cohomology of $\u_g$ are recalled in Section~\ref{sec:coho_u}, Results about the representation stability of the weight graded quotients of various invariants, such as $\u_{g,n+\vr}$, $H^\bdot(\u_{g,n+\vr})$, and $\gbar_{g,\uu}$, are discussed in Section~\ref{sec:stability}. The program for determining the stable cohomology of $\u_{g,n+\vr}$ is described in Section~\ref{sec:coho}.

\section{The Cohomology of \texorpdfstring{$\u_g$}{u g}}
\label{sec:coho_u}

In this section, we recall a few facts about the cohomology of the $\u_{g,n+\vr}$. The first is that the choice of a complex structure on a reference surface of type $(g,n+\vr)$ determines a natural isomorphism
$$
\Gr^W_\bdot H^\bdot(\u_{g,n+\vr}) \cong H^\bdot(\u_{g,n+\vr})
$$
and that the exactness properties of the weight filtration implies that there is a natural isomorphism
$$
\Gr^W_\bdot H^\bdot(\u_{g,n+\vr}) \cong H^\bdot(\Gr^W_\bdot \u_{g,n+\vr}).
$$
In addition, we shall need the following result, which is proved in \cite{hain:torelli,hain:kahler}. We use notation introduced in Section~\ref{sec:rep_th}. In addition we denote the $\Sp(H)$-invariant complement of the trivial representation in $\Lambda^2 H$ by $\Lambda^2_0 H$. \label{def:theta_perp}

\begin{lemma}
For all $g$, $W_{j-1}H^j(\u_{g,n+\vr})$ vanishes. The lowest weight subring
$$
\bigoplus_{j\ge 0} W_j H^j(\u_{g,n+\vr})
$$
of $H^\bdot(\u_{g,n+\vr})$ is quadratically presented for all $g\ge 3$. In particular, the lowest weight subring of $\u_{g,\ast}$, where $\ast \in \{0,1,\uu\}$ have quadratic presentations
\begin{align*}
&\bigoplus_{j\ge 0} W_j H^j(\u_{g}) = \Lambda^\bdot (\Lambda^3_0 H)/(V_\boxplus)\cr
&\bigoplus_{j\ge 0} W_j H^j(\u_{g,1}) = \Lambda^\bdot (\Lambda^3 H)/(V_\boxplus+\Lambda^2_0 H)\cr
&\bigoplus_{j\ge 0} W_j H^j(\u_{g,\uu}) = \Lambda^\bdot (\Lambda^3_0 H)/(V_\boxplus+\Lambda^2_0 H+\Q),
\end{align*}
where the generators have Hodge weight 1, and the generators of the ideal of relations have weight 2.
\end{lemma}

When $g\ge 6$, $\Lambda^2_0 H$ has multiplicity 3 in $\Lambda^2\Lambda^3 H$ and the trivial representation $\Q$ has multiplicity 2. Precise locations of the relations above can be deduced from the quadratic relations in $\u_{g,\ast}$ determined in \cite[\S11]{hain:torelli}. (See also \cite[\S9]{hain:rat_pts}.)

Theorem~\ref{thm:lie_torelli} implies that $\t_g$ is a non-trivial central extension of $\u_g$. Since the Gysin sequence of this extension is an exact sequence of MHS, the lowest weight subring of $H^\bdot(\t_g)$ is the lowest weight subring of $\u_g$ mod the ideal generated by the class $\kappa_1 \in H^2(\u_g)$ of the central extension, which is non-trivial by Theorem~\ref{thm:lie_torelli}.

\begin{corollary}
For all $g$, $W_{j-1}H^j(\t_{g,n+\vr})$ vanishes. The lowest weight subring
$$
\bigoplus_{j\ge 0} W_j H^j(\t_{g,n+\vr})
$$
of $H^\bdot(\t_{g,n+\vr})$ is quadratically presented for all $g\ge 3$. In particular, the lowest weight subring of $\t_g$ has quadratic presentation
$$
\Lambda^\bdot (\Lambda^3_0 H)/(V_\boxplus+\Q\kappa_1)
$$
where the generators $\Lambda^3_0 H$ have Hodge weight 1, and the relations have weight 2. The class $\kappa_1$ spans the unique copy of the trivial representation in $\Lambda^2\Lambda_0^3 H$.
\end{corollary}

\begin{remark}
\label{rem:cts_coho}
This result implies that $W_j H^j(\u_g)$ is non-zero for all $j\le \binom{g}{3}$ when $g\ge 3$. This is because the generating set $V := \Lambda^3_0 H(-1)$ is a Hodge structure of weight $-1$ and level 3. The set of relations $V_\boxplus$ is a Hodge structure of weight $-2$ and level $4$, as it is contained in $H^{\otimes 4}(-1)$. This implies that $\Lambda^j V$ is a Hodge structure of level $3j$ when
$$
0 \le j \le \dim \Lambda^j V^{2,-1} = \binom{\binom{g}{3}}{j}
$$
and that the ideal of relations has level $\le 3j-2$ in degrees $j\ge 2$. Similarly, the lowest weight subring of the Lie algebra $\t_g$ of the Torelli group is non-zero in all degrees $\le \binom{g}{3}$. Since the homological dimension of $T_g$ is $3g-5$ when $g\ge 2$ (by \cite{bbm}), the canonical homomorphism
$$
H^j(\t_g) \to H^j(T_g;\Q)
$$
cannot be an isomorphism when $3g-5 < j < \binom{g}{3}$.
\end{remark}

\subsection{Injectivity of the Johnson homomorphism}
\label{sec:injectivity}

Injectivity of the Johnson homomorphism $\g_{g,\uu} \to \Der\p(S,\vv)$ can be expressed cohomologically. The following result is a Lie algebra analogue of Stallings' Theorem \cite{stallings}. It follows directly from the self-contained discussion in \cite[\S18]{hain-matsumoto:mem}.

\begin{proposition}
\label{prop:stallings}
Suppose that $g \ge 3$. The truncated Johnson homomorphism $\g_{g,\uu}/W_{-m-1} \to \gbar_{g,\uu}/W_{-m-1}$ is an isomorphism if and only if
$$
W_m H^2(\ubar_{g,\uu}) \to W_m H^2(\u_{g,\uu})
$$
is injective (and therefore an isomorphism).
\end{proposition}

\section{Representation Stability}
\label{sec:stability}

In \cite{church-farb} Church and Farb formulated a general notion of {\em representation stability} associated to a sequence of groups
$$
G_1 \to G_2 \to \cdots \to G_j \to G_{j+1} \to \cdots
$$
and a sequence\footnote{Sometimes the arrows go in the opposite direction. This is not a serious issue; if all $V_j$ are finite dimensional, one can replace each $V_j$ by its dual $V_j^\vee$, for example.}
$$
V_1 \to V_2 \to \cdots \to V_j \to V_{j+1} \to \cdots
$$
of vector spaces, where $V_j$ is a $G_j$-module and $V_j \to V_{j+1}$ is equivariant with respect to the homomorphism $G_j \to G_{j+1}$. In the current context, the sequence of groups $\{G_n\}$ will be a sequence of symplectic groups $\Sp(\Gr^W_{-1}H_1(S_j))$ associated to an ascending sequence of connected surfaces $\{S_j\}$, as explained below. A typical sequence of modules will be the weight graded quotients of $\u_{S_j,\partial S_j}$. Representation stability in this situation was previously defined in \cite[\S6]{hain:torelli}. The stable values of some of some of the weight graded quotients of $\u_{g,n+\vr}$ and its cohomology were determined in subsequent sections of \cite{hain:torelli}.

Recall that the representation ring $\cR(G)$ of an affine group $G$ is defined to be the Grothendieck group of the category $\Rep(G)$ of its finite dimensional representations. When $G$ is reductive, it is the free $\Z$-module generated by the isomorphism classes of irreducible $G$-modules. We will denote the element of $\cR(G)$ corresponding to the finite dimensional $G$-module $V$ by $[V]$.

To obtain our sequences of groups and modules we start with a sequence of surfaces. This is done by starting with a surface $(\Sbar,P,\vV)$ of type $(g,n+\vr+\uu)$ with a distinguished boundary component and successively attaching a surface of type $(1,\vec{2})$ at the distinguished boundary component; the new boundary component of the resulting surface is its distinguished boundary component. This gives a sequence
\begin{equation}
\label{eqn:seq_surf}
S_g \subset S_{g+1} \subset S_{g+2} \subset \cdots
\end{equation}
of surfaces, where $S_k$ is of type $(k,n+\vr+\uu)$, and the sequence
$$
H_1(S_0) \hookrightarrow H_1(S_1) \hookrightarrow H_1(S_2) \hookrightarrow \cdots 
$$
and
$$
\Gr^W_{-1}H_1(S_0) \hookrightarrow \Gr^W_{-1}H_1(S_1) \hookrightarrow \Gr^W_{-1}H_1(S_2) \hookrightarrow \cdots
$$
of vector spaces. This, in turn, induces a sequence
\begin{equation}
\label{eqn:seqce_sp}
\Sp(\Gr^W_{-1}H_1(S_0)) \hookrightarrow \Sp(\Gr^W_{-1}H_1(S_1)) \hookrightarrow \Sp(\Gr^W_{-1}H_1(S_2)) \hookrightarrow \cdots
\end{equation}
of symplectic groups, which will be our sequence of groups.

The sequence (\ref{eqn:seq_surf}) of surfaces also induces a sequence of inclusions
$$
\G_{S_0,\partial S_0} \hookrightarrow \G_{S_1,\partial S_1} \hookrightarrow \G_{S_2,\partial S_2} \hookrightarrow \cdots
$$
of mapping class groups. This, in turn, induces sequences of invariants such as
$$
\g_{S_0,\partial S_0} \to\g_{S_1,\partial S_1} \to \g_{S_2,\partial S_2} \to \cdots, \quad
\gbar_{S_0,\partial S_0} \to\gbar_{S_1,\partial S_1} \to \gbar_{S_2,\partial S_2} \to \cdots
$$
and
\begin{align*}
H^\bdot(\u_{S_0,\partial S_0}) &\rightarrow H^\bdot(\u_{S_,\partial S_1}) \rightarrow H^\bdot(\u_{S_2,\partial S_2}) \rightarrow \cdots \cr
H^\bdot(\ubar_{S_0,\partial S_0}) &\rightarrow H^\bdot(\ubar_{S_,\partial S_1}) \rightarrow H^\bdot(\ubar_{S_2,\partial S_2}) \rightarrow \cdots
\end{align*}
These are examples of sequence of modules compatible with the sequence (\ref{eqn:seqce_sp}).

Using limit MHS, we can arrange for all of these stabilization maps to be morphisms of MHS.\footnote{Technical point: Since the vanishing cycles are homologically trivial, the relative weight filtration $M_\bdot$ and the weight filtration $W_\bdot$ coincide.} Consequently, the stabilizations of each of these invariants admits an ind- (or pro-) MHS. This implies that the weight graded quotients of each of these invariants stabilizes as a Hodge structure. Each of these invariants also stabilizes in the representation ring of $\Sp(H)$. These issues are discussed further in Sections 8, 9 and 10 of  \cite{hain:torelli} and in \cite{church-farb}.

The following is an {\em incomplete} list of invariants associated to completed mapping class groups that stabilize in the representation ring of $\Sp(H)$. When $S$ is of type $(g,\uu)$, the stability of $\Gr^W_\bdot\p(S,\vv)$ was proved in \cite[\S8]{hain:torelli}, and the stability of $\Gr^W_\bdot\gbar_{S,\partial S}$ was proved by Patzt in \cite{patzt}.

\begin{proposition}
\label{prop:stability}
Suppose that $(\Sbar,P,\vV)$ is a surface of type $(g,n+\vr+\uu)$ with $n,r\ge 0$. For $V$ being one of $\p(S)$, $\Der^\theta \p(S,\vv)$, $\g_{S,\partial S}$, $\gbar_{S,\partial S}$, $H^\bdot(\u_{S,\partial S})$, $H^\bdot(\ubar_{S,\partial S})$, the weight graded quotient $\Gr^W_m V$ stabilizes in the representation ring of $\Sp(H)$ as $g\to \infty$ and $n$ and $r$ remain constant. The stable value of $\Gr^W_m V$ occurs when
$$
g \ge
\begin{cases}
|m| & V = \p(S,\vv), \cr
|m|+2 & V = \gbar_{S,\partial S},\ \Der^\theta \p(S,\vv),\cr
|m|/3 & V = \g_{S,\partial S},\ H^j(\u_{S,\partial S}),\ H^j(\ubar_{S,\partial S}).
\end{cases}
$$
\end{proposition}

These stability ranges are not optimal. For example, for $m$ for which $\Gr^W_m \g_{S,\partial S} = \Gr^W_m\gbar_{S,\partial S}$ (e.g., $-6\le m \le -1$ when $g\gg 0$), stabilization  of $\Gr^W_m \g_{S,\partial S}$ occurs when $g\ge -m+2$.

\begin{proof}[Sketch of Proof]
We use the notation above and we apply the conventions (e.g., choice of torus, positive roots) of \cite[\S6]{hain:torelli}, where $a_1,\dots,a_k,b_1,\dots,b_k$ is a symplectic basis of $\Gr^W_{-1}H_1(S_{k})$. Suppose that $V_k$ is an $\Sp_k$-module ($k=h,h+1$) and that $V_h \to V_{h+1}$ is equivariant with respect to the inclusion $\Sp_h \to \Sp_{h+1}$. Then $V_h \to V_{h+1}$ is an isomorphism in the stable representation ring $\cR(\Sp)$ if it induces a bijection on highest weight vectors. In particular, $V_h$ and $V_{h+1}$ have the same highest weight decompositions.

Denote the Schur functor associated to a partition $\mu$ by $\Schur_\mu$. The key point is to observe that if $V_h \to V_{h+1}$ is an isomorphism in the stable representation ring $\cR(\Sp)$ and if $h$ is sufficiently large relative to $\mu$, then by \cite{kabanov:stability}, $\Schur_\mu V_h \to \Schur_\mu V_{h+1}$ is also an isomorphism in $\cR(\Sp)$.\footnote{Kabanov \cite{kabanov:stability} proves that if $V_\nu$ is the irreducible representation of $\Sp_h$ corresponding to the partition $\nu$, then $\Schur_\mu V_\nu$ stabilizes in $\cR(\Sp)$ when $h\ge |\mu||\nu|$.}

Suppose that
$$
V_g \to V_{g+1} \to V_{g+2}  \to \cdots
$$
is a sequence of finite dimensional {\em graded} $\Sp$-modules that stabilizes in $\cR(\Sp)$. Assume that $\Gr^W_m V_g = 0$ when $m\ge 0$. This induces a sequence
$$
\L(V_g) \to \L(V_{g+1}) \to \L(V_{g+2})  \to \cdots
$$
of graded Lie algebras in $\cR(\Sp)$. Representation stability implies that each weight graded quotient of the graded tensor algebra $T(V_g)$ stabilizes in $\cR(\Sp)$. The PBW theorem and the stability of Schur functors then implies that each weight graded quotient of $\L(V_k)$ stabilizes in $\cR(\Sp)$.

We first apply this when $V_k = \Gr^W_\bdot H_1(S_{k})$. Since $\Gr^W_{-m}\p(S_k,\vv)$ is canonically a quotient of $\Gr^W_{-m}\L(\Gr^W_\bdot H_1(S_k))$, $\Gr^W_{-m}\p(S_k,\vv) \to \Gr^W_{-m}\p(S_{k+1},\vv)$ is surjective on highest weight vectors when $k$ is sufficiently large. Denote the kernel of the canonical quotient map
$$
\L(\Gr^W_\bdot H_1(S_k)) \to \p(S_k,\vv)
$$
by $\fr_k$. When $k\ge 2$ this ideal is generated by elements of weight $-2$ by \cite{bezrukavnikov} and \cite[\S12]{hain:torelli}. The surjectivity of the bracket map $\Gr^W_{-m}\fr_k \otimes \Gr^W_{-1}H_1(S_k) \to \Gr^W_{-m-1}\fr_k$ implies that $\Gr^W_{-m}\fr_k \to \Gr^W_{-m}\fr_{k+1}$ is surjective on highest weight vectors when $k$ is sufficiently large. Since (as $\Sp_{g+k}$-modules) $\fr_k \oplus \Gr^W_\bdot\p(S_k,\vv) = \L(\Gr^W_\bdot H_1(S_k))$, we conclude that $\p(S_k,\vv)$ stabilizes in $\cR(\Sp_g)$.

Similarly taking $V_k = H_1(\u_{S_k,\partial S_k})$ and using the fact that each $\Gr^W_\bdot\u_{S_k,\partial S_k}$ is quadratically presented when $k\ge 4$, we see that each weight graded quotient of $\u_{S_k,\partial S_k}$ stabilizes in $\cR(\Sp)$.

The results on cohomology of $\u_{S,\partial S}$ follow similarly. Stability of $\Gr^W_\bdot H^\bdot(\u_{S_k,\partial S_k})$ follows as each weight graded quotient of the complex $C^\bdot(\Gr^W_\bdot\u_{S_k,\partial S_k})$ (vector spaces and maps) stabilizes in $\cR(\Sp)$ by the stability of Schur functors \cite{kabanov:stability}.
\end{proof}

The stability results also hold in the case where there are no boundary components --- that is, when $r=0$.

\begin{corollary}
For all $n\ge 0$, the weight graded quotients of $\g_{g,n}$ and $H^\bdot(\u_{g,n})$ stabilize in the representation ring of $\Sp(H)$.
\end{corollary}

\begin{proof}
The weight graded quotients of $\g_{g,n}$ stabilize in the representation ring of $\Sp(H)$ when $n>0$ as the Gysin sequence
$$
0 \to \Q(1) \to \g_{g,n-1+\uu} \to \g_{g,n} \to 0.
$$
is an exact sequence of MHS. Exactness of $\Gr^W_\bdot$ then implies that the weight graded quotients of $\g_{g,n-1+\uu}$ and $\g_{g,n}$ are isomorphic except in weight $-2$, where they differ by a copy of the trivial representation. Similarly, representation stability of the weight graded quotients is easily deduced from the exactness of the sequence
$$
0 \to \p_g \to \g_{g,1} \to \g_g \to 0
$$
where $\p_g$ is the Lie algebra of the unipotent completion of the fundamental group of a smooth compact surface of genus $g$. Its weight graded quotients stabilize in $\R(\Sp(H))$.
\end{proof}

For example the first 3 graded quotients are 
$$
\Gr^W_{-m} \u_g \cong 
\begin{cases}
\Lambda^3_0 H & m=1,\ g \ge 3, \cr
V_\boxplus & m=2,\ g \ge 3,\cr
V_{[31^2]} & m=3,\ g \ge 3.
\end{cases}
$$
(See \cite[Prop.~9.6]{hain:torelli}.) The next 3 stable values of $\Gr^W_{-m}\u_g$ can be found in \cite{morita-sakasai-suzuki}.\footnote{The computations in \cite{morita-sakasai-suzuki} compute $\Gr^W_{-m}\ubar$, but imply that $W_6H^2(\ubar) = W_6H^2(\u)$. So $\ubar/W_{-7} \cong \u/W_{-7}$ by Proposition~\ref{prop:stallings}.}  This is the current state of the art. However, if the stable cohomology of $\u_g$ satisfies a purity condition, one can compute $\Gr^W_\bdot \u_g$ as will be explained in the next section.

\section{Stable Cohomology}
\label{sec:coho}

The cohomology of a (negatively weighted) pronilpotent Lie algebra $\u$ in the category of MHS determines the associated weight graded Lie algebra $\Gr^W_\bdot \u$. In particular, the stable cohomology of $\u_{g,n+\vr}$ determines the stable value of $\Gr^W_\bdot \u_{g,n+\vr}$, and the stable cohomology of $\ubar_{g,n+\vr}$ determines the stable value of $\Gr^W_\bdot \ubar_{g,n+\vr}$. Here we discuss evidence that the stable cohomology of $\u_{g,n+\vr}$ equals the stable cohomology groups of $\G_{g,n+\vr}$ with symplectic coefficients. If this is the case, it provides another avenue for determining the image of the Johnson homomorphism.

\subsection{Comparison with $H^\bdot(\G_{g,n+\vr};V_\mu)$}
\label{sec:comparison}

The cohomology of an affine (i.e., proalgebraic) group over a field $\kk$ with coefficients in a $G$-module $V$ is defined by
$$
H^j(G;V) := \Ext^j(\kk,V),
$$
where the Ext is taken in the category $\Rep(G)$ of $G$-modules \cite{jantzen}. A homomorphism $\G \to G(\kk)$ from a discrete group into the $\kk$-rational points of $G$ induces a homomorphism
$
H^\bdot(G;V) \to H^\bdot(\G;V)
$
for each $G$-module $V$. In particular, for each $\Sp(H)$-module $V$, we have the homomorphism
\begin{equation}
\label{eqn:cts_coho}
H^\bdot(\cG_{g,n+\vr};V) \to H^\bdot(\G_{g,n+\vr};V)
\end{equation}
induced by $\G_{g,n+\vr} \to \cG_{g,n+\vr}(\Q)$. General properties of relative completion \cite[\S7]{hain:torelli} imply that this map is an isomorphism in degrees $\le 1$ and is injective in degree 2. Standard facts about the cohomology of algebraic groups imply that
$$
H^j(\cG_{g,n+\vr};V) \cong \big[H^j(\u_{g,n+\vr})\otimes V\big]^{\Sp(H)}.
$$
(This is very well known. A complete proof can be found in \cite[Cor.~3.7]{hain:db}.) These homomorphisms are compatible with Hodge theory in a sense to be explained below.

The map (\ref{eqn:cts_coho}) is an isomorphism in genus 0 (in all degrees) as each $\M_{0,n}$ is a hyperplane complement of fiber type, and therefore a rational $K(\pi,1)$. One can check that it is an isomorphism in genus 1 in all degrees as well. A proof can be deduced from results in Sections 9 and 10 of \cite{hain-matsumoto:mem}. However, it fails to be an isomorphism for all $g\ge 3$. This follows from the discussion in Remark~\ref{rem:cts_coho}: since $\Gamma_g$ has virtual cohomological dimension $4g-5$, (\ref{eqn:cts_coho}) cannot be an isomorphism in degrees $4g-4\le j \le \binom{g}{3}$ for at least one coefficient module $V$ in each degree.

\subsection{Towards the stable cohomology of $\cG_{g,n+\vr}$}
\label{sec:to_stab_coho}

Isomorphism classes of irreducible $\Sp_g = \Sp(H)$-modules are in bijective correspondence with partitions
$$
\bmu:\ \mu_1 + \mu_2 + \dots + \mu_g,\quad \mu_1 \ge \mu_2 \ge \dots \ge \mu_g \ge 0
$$
of positive integers into $\le g$ pieces. Set $|\bmu| := \mu_1 + \cdots + \mu_r$. For each partition $\bmu$ fix an irreducible $\Sp(H)$-module $V_\bmu$ in the corresponding isomorphism class. Each Hodge structure on $H$ determines a Hodge structure on $V_\bmu$ of weight $-|\bmu|$.

\begin{conjecture}[cf.\ {\cite[\S9]{hain-looijenga}}]
\label{conj:stab_iso}
The homomorphism (\ref{eqn:cts_coho}) with $V=V_\bmu$ is stably an isomorphism for all $\bmu$.
\end{conjecture}

The local system $\V_\bmu$ over $\M_{g,n+\vr}$ corresponding to $V_\bmu$ has a unique structure of a polarized variation of Hodge structure of weight $-|\bmu|$ over $\M_{g,n+\vr}$. The natural isomorphism
$$
H^j(\G_{g,n+\vr};V_\bmu) \cong H^j(\M_{g,n+\vr};\V_\bmu)
$$
induces a MHS on $H^j(\G_{g,n+\vr};V_\bmu)$ with weights $\ge j-|\bmu|$. Each choice of a complex structure $\phi$ on $(\Sbar,P,\vV)$, the MHS on $\u_{g,n+\vr}$ determines MHSs on $H^\bdot(\u_{g,n+\vr})$ and $V_\bmu$ via linear algebra. With these MHSs, the natural map
\begin{equation}
\label{eqn:hodge_homom}
H^j(\cG_{g,n+\vr};V_\bmu) \cong \big[H^j(\u_{g,n+\vr})\otimes V_\bmu\big]^{\Sp(H)} \to H^j(\G_{g,n+\vr};V_\bmu)
\end{equation}
is a morphism of MHS. For a proof, see \cite[\S7]{hain:torelli}.

If (\ref{eqn:cts_coho}) is stably an isomorphism it will force $H^j(\u_{g,n+\vr})$ to be pure of weight $j$. This leads us to:

\begin{conjecture}[Purity]
\label{conj:purity}
The weight $m$ stable cohomology of $\u_{g,n+\vr}$ is of degree $m$. That is,
$$
\Gr^W_m H^\bdot(\u_{g,n+\vr}) = \Gr^W_m H^m(\u_{g,n+\vr}) \quad \text{for all } g\ge m/3.
$$
\end{conjecture}

The relevance of this conjecture is that purity of $H^\bdot(\u_{g,n+\vr})$ in degrees $\le m$ allows the computation of $\Gr^W_j \u_{g,n+\vr}$ in the representation ring of $\Sp(H)$ for $j\le m$ by the formula in the Appendix~\ref{sec:inversion}. Purity is known when $m=1$ (Johnson \cite{johnson:h1}) and $m=2$ (Kabanov \cite{kabanov:purity} and Hain \cite{hain:torelli}), and $2<m\le 6$ (Morita--Sakasai--Suzuki \cite{morita-sakasai-suzuki}). The general case would follow if one can show that the enveloping algebra of $\u_{g,n+\vr}$ is stably Koszul dual to its cohomology. (Cf.\ \cite[Q.~9.14]{hain-looijenga}.)

Further evidence for these conjectures is provided by the following result, which implies that the lowest weight cohomology of $\cG_{g,\ast}$ is isomorphic to the stable cohomology of $\G_{g,\ast}$ for $\ast \in \{0,1,\uu\}$.

\begin{theorem}[Garoufalidis--Getzler, Petersen, Kupers--Randal-Williams]
\label{thm:stab_iso}
For all partitions $\bmu$ and $\ast \in \{0,1,\uu\}$, the homomorphism
$$
\nu_\ast : \bigoplus_{m\ge 0} W_{m-|\bmu|}H^m(\cG_{g,\ast},V_\bmu) \cong \big[\bigoplus_{m\ge 0}\Lambda^\bdot W_m H^m(\u_{g,\ast}) \otimes V_\bmu\big]^{\Sp(H)} \to H^\bdot(\G_{g,\ast};V_\bmu)
$$
is an isomorphism when $g\gg 0$.
\end{theorem}

The history of this theorem is tangled and still evolving. Garoufalidis and Getzler \cite{garoufalidis-getzler} claim it is true when $\ast = 0,1$. However, as pointed out to me by Dan Petersen, their proof is incomplete. Petersen \cite{petersen:letter} has given a corrected proof in this case. The $\ast = \uu$ case can be deduced from results of Kupers and Randal-Williams in \cite{krw}. The goal of the following proposition is to clarify the relationship between the 3 cases of the theorem. In particular, we prove that the $\ast = 0$ and $\ast = 1$ cases are equivalent.

\begin{proposition}
The following hold:
\begin{enumerate}

\item $\nu_0$ is a stable isomorphism if and only if $\nu_1$ is a stable isomorphism.

\item If $\nu_1$ a stable isomorphism, then $\nu_\uu$ is a stable isomorphism.

\end{enumerate}
\end{proposition}

\begin{proof}[Sketch of Proof]
We will focus on the first assertion as there are several proofs of the $\ast = 1,\uu$ cases. First note that if $S$ is a compact surface of genus $g$, then the natural map $H^\bdot(\p(S)) \to H^\bdot(S;\Q)$ is an isomorphism.\footnote{This is long well-known. A proof can found in \cite[\S5.1]{hain:goldman}.} Since $H^2(\cG_{g,1}) \to H^2(\G_{g,1};\Q)$ is an isomorphism for all $g\ge 3$, there is a class
$$
\psi \in H^2(\cG_{g,1}) \cong H^2(\u_{g,1})^{\Sp(H)}
$$
that corresponds to the first Chern class of the relative tangent bundle of $\M_{g,1} \to \M_g$. Its restriction to $H^2(\p)$ is non-zero. Choose a closed, $\Sp(H)$-invariant representative of it in $\Lambda^2 \Gr^W_1 \u_{g,1}^\vee$, where $V^\vee$ denotes the dual in ind-MHS of a pro-MHS $V$. We'll also denote it by $\psi$. The homomorphism $\G_{g,1} \to \cG_{g,1}(\Q)$ induces a map of extensions 
$$
\xymatrix{
1 \ar[r] & \pi_1(S,x_o) \ar[r]\ar[d] & \G_{g,1} \ar[r]\ar[d] & \G_g \ar[r]\ar[d] & 1 \cr
1 \ar[r] & \cP \ar[r] & \cG_{g,1} \ar[r] & \cG_g \ar[r] & 1
}
$$
and thus of the corresponding spectral sequences that compute cohomology with coefficients in $V_\bmu$. Since $\M_{g,1} \to \M_g$ has smooth projective fibers, the spectral sequence of the top extension collapses at $E_2$ by Deligne's theorem \cite{deligne:degen}. Since $H^\bdot(\p) \to H^\bdot(S;\Q)$ is an isomorphism, Deligne's argument also implies that the spectral sequence of the lower extension collapses at $E_2$. This is because cupping with $\psi \in E_r^{0,2}$ commutes with $d_r$ and induces an isomorphism $E_r^{s,0} \to E_r^{s,2}$. It is now an exercise to prove the first claim.

The classes $\psi$ are also the first Chern classes of the central extensions
$$
0 \to \Z \to \G_{g,\uu} \to \G_{g,1} \to 1 \text{ and } 0 \to \Ga \to \cG_{g,\uu} \to \cG_{g,1} \to 1.
$$
Their Gysin sequences are long exact sequences of MHS. The exactness properties the weight filtration implies that the rows of the diagram
$$
\xymatrix{
W_{m-2} H^{j-2}(\cG_{g,1},V_\bmu) \ar[r]^(.55)\psi \ar[d]^{\nu_1^{j-2}} & W_m H^j(\cG_{g,1},V_\bmu) \ar[r] \ar[d]^{\nu_1^j} & W_m H^j(\cG_{g,\uu},V_\bmu) \ar[d]^{\nu_\uu^j} \ar[r] & 0
\cr 
W_{m-2} H^{j-2}(\G_{g,1},V_\bmu) \ar[r]^(.55)\psi & W_m H^j(\G_{g,1},V_\bmu) \ar[r] & W_m H^j(\G_{g,\uu},V_\bmu) \ar[r] & 0
}
$$
are exact, where $m=j-|\bmu|$. So if $\nu_1$ is stably an isomorphism, then so is $\nu_\uu$.
\end{proof}

Since the sequence
$$
0 \to \pi^\un_1(S,x_0) \to \cG_{g,n+1} \to \cG_{g,n} \to 1
$$
is exact, where $S$ is a surface of type $(g,n)$, and since $H^\bdot(\pi_1^\un(S,x_0)) \to H^\bdot(S;\Q)$ is an isomorphism, Theorem~\ref{thm:stab_iso} implies, by induction on $n$, that 

\begin{corollary}
\label{cor:conj}
Conjecture~\ref{conj:purity} is equivalent to Conjecture~\ref{conj:stab_iso}.
\end{corollary}

\part{Afterword}

There is significantly more to Torelli groups than what one can see with the rich story surrounding Johnson homomorphisms and relative unipotent completion of mapping class groups. Here we discuss one small (but very significant) example that suggests that what one sees with the aid of Johnson homomorphisms is just the tip of the Torelli iceberg and suggests that what lies beyond may be very profound.

\section{Beyond Johnson Homomorphisms}
\label{sec:beyond}

Recall that $\cA_g$ denotes the moduli space of principally polarized varieties of dimension $g$. It is an orbifold with fundamental group $\Sp_g(\Z)$. Denote the local system over $\cA_g$ that corresponds to the irreducible $\Sp_g$-module $V_\bmu$ by $\V_\bmu$, where $\bmu$ is a partition of the positive integer $|\bmu|$ into $\le g$ parts. It has a unique structure of a polarized variation of Hodge structure over $\cA_g$ of weight $|\bmu|$. We also denote its pullback to $\M_g$, the moduli space of smooth projective curves of genus $g$, along the period map $\M_g \to \cA_g$ by $\V_\bmu$. Saito's work \cite{saito} implies that the cohomology groups
$$
H^j(\cA_g;\V_\bmu) \text{ and } H^j(\M_g;\V_\bmu)
$$
have a natural mixed Hodge structure with weights bounded below by $j+|\bmu|$. Similarly, for every prime number $\ell$, the local system $\V_\bmu\otimes\Ql$ is naturally a lisse sheaf over $\M_{g/\Q}$ of weight $|\bmu|$. Since the stacks $\cA_{g/\Z}$ and $\M_{g/\Z}$ have good reduction at all prime numbers $p$, the cohomology groups
$$
H^j(\cA_g;\V_\bmu)\otimes\Ql \cong H^j_\et(\cA_{g/\Qbar};\V_\bmu\otimes \Ql)
$$
and
$$
H^j(\M_g;\V_\bmu)\otimes\Ql \cong H^j_\et(\M_{g/\Qbar};\V_\bmu\otimes \Ql)
$$
are $G_\Q := \Gal(\Qbar/\Q)$-modules of weight $\ge j+|\bmu|$ that are unramified at all primes $p\neq \ell$ and crystalline at $\ell$.

It is important to study the simple Hodge structures and/or $G_\Q$-modules that appear in the weight graded quotients of $H^\bdot(\M_g;\V_\bmu)$ as well as the extensions that occur between them. The ones that are most interesting are those that do not already occur in $H^\bdot(\cA_g;\V_\bmu)$, or in the cohomology of a moduli space of curves of lower genus, as these are the genuinely new motives associated to $\M_g$. Galois representations in the cohomology of $\M_3$ that do not come from $\cA_3$ or lower genus moduli spaces were found by Cl\'ery, Faber and van der Geer \cite[\S13]{genus3}.

These new motivic structures have to come from $T_g$ as can be seen from the Hochschild--Serre spectral sequence of the group extension
$$
1 \to T_g \to \G_g \to \Sp_g(\Z) \to 1.
$$
It converges to $H^\bdot(\M_g;\V_\bmu)$ and satisfies
\begin{equation}
\label{eqn:hoch-serre_ss}
E_2^{j,k} = H^j(\Sp_g(\Z);H^k(T_g)\otimes V_\bmu).
\end{equation}
The rational cohomology groups of $T_g$ are not all finite dimensional. This was proved by Geoff Mess \cite{mess} in genus 2, Johnson and Millson in genus 3 (see \cite{hain:a3}), and by Akita \cite{akita} for all $g\ge 7$. (This should hold for all $g\ge 2$.) Consequently, the finiteness and motivic properties of this spectral sequence are completely mysterious.

\begin{question}
Are the groups $H^j(\Sp_g(\Z);H^k(T_g)\otimes V_\bmu)$ finite dimensional?
\end{question}

This is open, even when $V_\bmu$ is the trivial representation.

\begin{question}
Does each term $E_2^{j,k}$ of the Hochschild--Serre spectral sequence (\ref{eqn:hoch-serre_ss}) carry a natural ind-MHS or Galois action? If so, are the differentials morphisms?
\end{question}

An important question is what kinds of infinite dimensional representations of $\Sp_g(\Z)$ occur in $H^\bdot(T_g;\Q)$ and what their relation is to Hodge theory and Galois actions. Several comments are in order. The first is that if an $\Sp_g(\Z)$-module $M$ is finite dimensional, then by Margulis super-rigidity \cite{margulis}, its restriction to some level subgroup $\Sp_g(\Z)[N]$ of $\Sp_g(\Z)$ is the restriction of a rational representation of $\Sp_g/\Q$ to $\Sp_g(\Z)[N]$. In this case, there is a natural isomorphism
$$
H^\bdot(\Sp_g(\Z);M) \cong H^\bdot(\Sp_g(\Z)[N];M)^{\Sp_g(\Z/N)}.
$$
This group will have a natural MHS (up to Tate twist) and Galois action.

The second is that these new and interesting motives cannot come from the relative completion $\cG_g$ of $\G_g$. This is because the image of
$$
\big[H^\bdot(\u_g)\otimes V_\lambda\big]^{\Sp(H)} = H^\bdot(\cG_g;V_\bmu) \to H^\bdot(\M_g;\V_\bmu)
$$
is mixed Tate (in both worlds, Hodge and Galois). This means that the Hodge structures/Galois representations in the cohomology of $\M_g$ have to come from deeper inside the Torelli group than is visible to Johnson homomorphisms.

Consequently, the interesting new motives in $H^\bdot(\M_g;\V_\bmu)$ must come from the infinite dimensional cohomology of $T_g$. To understand this better, it is useful to decompose the rational cohomology of $T_g$ under the action of $-I \in \Sp_g(\Z)$:
$$
H^\bdot(T_g;\Q) = H^\bdot(T_g;\Q)^+ \oplus H^\bdot(T_g;\Q)^-
$$
The $+$ part is the orbifold rational cohomology of the locus in $\cA_g$ of jacobians of smooth curves, \cite[Prop.~5]{hain:a3}. The $-$ part appears to be deeper and more interesting. The ``center kills'' argument implies that
$$
H^\bdot(\Sp_g(\Z);H^\bdot(T_g)^+\otimes V_\bmu) = 0
$$
when $|\bmu|$ is odd and
$$
H^\bdot(\Sp_g(\Z);H^\bdot(T_g)^-\otimes V_\bmu) = 0
$$
when $|\bmu|$ is even as $-I$ acts on $V_\bmu$ as $(-1)^{|\bmu|}$.

The $\Sp_3(\Z)$ representations that occur in $H^\bdot(T_3;\Q)^+$ are all induced representations. Their computation in \cite[Thm.~17]{hain:a3} implies via Shapiro's Lemma that
$$
H^j(\Sp_3(\Z);H^k(T_3)^+\otimes V_\bmu) \cong
\begin{cases}
\Q & k = 0,\cr
H^j(\SL_2(\Z)\times \Sp_2(\Z); V_\bmu) & k = 3, \cr
[H^j(\SL_2(\Z)^3;V_\bmu)\otimes U\big]^{\Sigma_3} & k=4, \cr
0 & \text{otherwise},
\end{cases}
$$
where $U$ is irreducible 2-dimensional representation of $\Sigma_3$, the symmetric group on 3 letters. The groups on the right-hand side carry natural MHS and Galois actions. I do not know if they are compatible with the corresponding structure on $H^\bdot(\M_3;\V_\bmu)$, but suspect that they are. The $-$ part remains mysterious and should be linked to the examples of Cl\'ery, Faber and van der Geer. It seems very unlikely that all of the $\Sp_3(\Z)$ representations that occur in $H^\bdot(T_3;\Q)^-$ are induced up from an arithmetic (though not Zariski dense) subgroup of $\Sp_3(\Z)$.

\appendix

\section{Proof of Theorem~\ref{thm:edge_sigma}}
\label{sec:proof_edge}

Lemma~\ref{lem:decomp} implies that we need only prove the result when $S$ is $\Pminus$. Denote $\partial/\partial z \in T_1\P^1$ by $\vv_1$. For $a\in \{0,1,\infty\}$, let $\vv_a \in T_a \P^1$ be an image of $\vv_1$ under the action of the symmetric group action on $\Pminus$. (These are well defined up to a sign.) Set $\p_a = \p(\Pminus,\vv_a)$. By Corollary~\ref{cor:edg_on_k}, $\edg^S$ does not depend on the framing, so we will ignore it.

\subsection{The symmetric depth filtration of $\Q\lambda(\Pminus)^\wedge$}

The inclusion $\Pminus \hookrightarrow \Gm$ induces a Lie algebra homomorphism
$$
\phi_1 : \p_1 \to \p(\Gm,1)
$$
in $\MHS$. The depth filtration $D_1^\bdot$ of $\p_1$ is defined by $D_1^0 \p_1 = \p_1$ and $D_1^k\p_1 = L^k \ker \phi_1$ when $k\ge 1$, the $k$th term of the lower central series of $\ker \phi_1$. The depth filtration of $\p_1$ induces one on its enveloping algebra $\Q\pi_1(\Pminus,\vv_1)^\wedge$ and one on $\Q\lambda(\Pminus)^\wedge$ via the quotient map. Denote all of these depth filtrations by $D_1^\bdot$.

The depth filtration $D_1^\bdot$ of $\Q\lambda(\Pminus)^\wedge$ depends on the choice of the ``boundary component'' $1\in \{0,1,\infty\}$ where the tangent vector $\vv_1$ is anchored. Similarly, we have the depth filtrations $D^\bdot_0$ and $D^\bdot_\infty$ of $\Q\lambda(\Pminus)^\wedge$. These are filtrations in $\MHS$ and are permuted by the $\Sigma_3$ action on $\Pminus$.

Define the {\em symmetric depth filtration} $\sD^\bdot$ of $\Q\lambda(\Pminus)^\wedge$ to be the intersection of these 3 filtrations:
$$
\sD^k \Q\lambda(\Pminus)^\wedge := (D_0^k \cap D_1^k \cap D_\infty^k)\Q\lambda(\Pminus)^\wedge.
$$
It is a filtration by MHSs and is $\Sigma_3$-invariant. It is therefore determined by the induced filtration on $\Gr^W_\bdot\Q\lambda(S)^\wedge$. We now describe this explicitly.

For each $a\in\{0,1,\infty\}$, there is a canonical isomorphism
$$
\Gr^W_\bdot\p_a \cong \L(\ee_0,\ee_1,\ee_\infty)/(\ee_0 + \ee_1 + \ee_\infty).
$$
The associated weight graded of $\p_1 \to \p(\Gm,1)$ is the homomorphism
$$
\L(\ee_0,\ee_1,\ee_\infty)/(\ee_0 + \ee_1 + \ee_\infty) \to \L(\ee_0,\ee_\infty)/(\ee_0+\ee_\infty) \cong \L(\ee_0)
$$
that sends $\ee_1$ to $0$. The filtration $D_1^\bdot$ of $\Gr^W_\bdot\p_1$ is the filtration
\begin{align*}
D_1^k \Gr^W_\bdot\p_1 &= \{\text{Lie words in $\ee_0,\ee_1$ of degree $\ge k$ in $\ee_0$}\}
\cr
&= \{\text{Lie words in $\ee_1,\ee_\infty$ of degree $\ge k$ in $\ee_\infty$}\}.
\end{align*}
Observe that
$$
\Gr_{D_1}^1 \p_1 = \bigoplus_{n=0}^\infty \Q\ad_{\ee_1}^n \ee_0 = \bigoplus_{n=0}^\infty \Q\ad_{\ee_1}^n \ee_\infty.
$$
Formulas for $D_0^\bdot$ and $D_\infty^\bdot$ are obtained by permuting the indices $0,1,\infty$. As a consequence
\begin{multline*}
\sD^k\Gr^W_{-2d}\Q\lambda(\Pminus)^\wedge \cr
= \{f \in |H_1(\Pminus)^{\otimes d}| : \deg_{\ee_a}f \ge k,\ a = 0,1,\infty\},
\end{multline*}
where we regard $\ee_0,\ee_1,\ee_\infty$ as elements of $H_1(\Pminus)$, which is canonically isomorphic to $\Gr^W_{-2}\p_a$ for $a\in\{0,1,\infty\}$.

\subsection{The depth filtrations of $\SDer \p_1$ and $\k$}

In order to compute $\edg^S(\sigma_{2m+1})$ mod $\sD^2\Q\lambda(S)^\wedge$ we need to relate $\sD^\bdot$ to the standard depth filtration $D^\bdot$ of $\SDer \p_a$, which is defined by
$$
D^k \SDer\p_a := \{\sigma \in \SDer \p_a: \sigma(D^j_a\p_a) \subseteq D_a^{j+k}\p_a \text{ all } j\ge 0\}.
$$
It satisfies $[D^j,D^k]\subseteq D^{j+k}$. The depth filtration $D^\bdot$ of $\k$ is defined to be its restriction to $\k$ under the inclusion $\k \hookrightarrow \SDer \p_a$. It does not depend on $a$ and is a central filtration, so that $D^k\k \supseteq L^k \k$.\footnote{These are not, in general, equal. An important open problem is to understand the generators and relations of the Lie algebra $\Gr_D^\bdot\k$. The relations in depth two are known to correspond to classical cusp forms of level 1. See \cite{brown:depth} for precise statements and conjectures.}

Since the depth filtration of $\SDer \p_a$ is a filtration by MHS, it is determined by the depth filtration on $\SDer\Gr^W_\bdot\p_a$. Observe that $D_\bu \in D^k\SDer \Gr^W_\bdot \p_1$ if and only if $u_0$ and $u_\infty$ are both in $D^k\Gr^W_\bdot\p_1$. Combined with Proposition~\ref{prop:SDer}, this implies that
$$
D^k\SDer \p_a = \SDer \p_a \cap \sD^k\Q\lambda(\Pminus)^\wedge
$$
which implies that $D^k \k = \k \cap \sD^k\Q\lambda(\Pminus)^\wedge$.

\subsection{The computation of $\edg^S(\sigma_{2m+1})$ mod $\sD^2$}

The key point for us is that $\k = D^1\k$ and $[\k,\k] = D^2 \k$: The second statement is well known and follows from the action of $\k$ on the ``polylog quotient'' $\p_1/D^2\p_1$ of $\p_1$. It implies that $\edg^S$ vanishes on $D^2\k$.

The action of $\sigma_{2n+1}$ on the polylog quotient $\p_1/D^2_1$ implies that the homomorphism $H_1(\k) \mapsto \Gr_D^1\SDer\Gr^W_\bdot \p_1$ takes $[\sigma_{2m+1}]$ (suitably scaled) to the derivation
$$
\ee_0 \mapsto \ad_{\ee_0}^{2m+1} \ee_1 = [\ee_0,\ad_{\ee_0}^{2m}\ee_1] \text{ and } \ee_\infty \mapsto \ad_{\ee_\infty}^{2m+1} \ee_1 = [\ee_\infty,\ad_{\ee_\infty}^{2m}\ee_1].
$$
This implies that if $\sigma_{2m+1}\mapsto D_\bu$, where $\bu = (u_0,u_\infty)$, then, mod $D^2\p_1$,
$$
u_0 \equiv \ad_{\ee_0}^{2m}\ee_1 = -\ad_{\ee_0}^{2m}\ee_\infty \text{ and } u_\infty \equiv \ad_{\ee_\infty}^{2m}\ee_1 = -\ad_{\ee_\infty}^{2m}\ee_0.
$$
Thus, mod $D^2_1 \Q\langle\ee_0,\ee_\infty\rangle$, we have\footnote{Use the formula $\ad_x^k y = \sum_{j=0}^k (-1)^j\binom{k}{j} x^{k-j}y x^j$.}
$$
u_0^{(0)} \equiv -\sum_{k=1}^{2m} (-1)^k \binom{2m}{k} \ee_0^k \ee_\infty \ee_0^{2m-k} \text{ and }
u_\infty^{(\infty)} \equiv -\sum_{k=1}^{2m} (-1)^k \binom{2m}{k} \ee_\infty^k \ee_0 \ee_\infty^{2m-k}.
$$
Since $\ee_0+\ee_1+\ee_\infty=0$,
$$
\edg^S D_\bu = \div_1 D_\bu \equiv |\ee_0^{2m} \ee_\infty| + |\ee_\infty^{2m} \ee_0| \equiv -\frac{1}{2m+1}\big(|\ee_0^{2m+1}| + |\ee_1^{2m+1}| + |\ee_\infty^{2m+1}|\big)
$$
mod $\sD^2\Gr^W_\bdot\Q\lambda(S)^\wedge$. Since $\edg^S(\sigma_{2m+1})$ is central by Lemma~\ref{lem:decomp}, Proposition~\ref{prop:center} implies that, after rescaling,
$$
\edg^S(\sigma_{2m+1}) = (\log \mu_0)^{2m+1} + (\log \mu_1)^{2m+1} + (\log \mu_\infty)^{2m+1}.
$$

\section{M\"obius inversion}
\label{sec:inversion}

Suppose that $\kk$ is a field of characteristic zero, $G$ is a reductive $\kk$-group, such as $\Sp(H)$ and that
$$
\h = \bigoplus_{n\ge 1} \h_n
$$
is a graded Lie algebra in the category of $G$-modules. The Chevalley--Eilenberg chain complex
$$
C_\bdot(\h) := \Lambda^\bdot \h
$$
has a natural grading induced by the grading of $\h$ that is preserved by the differential. This implies that $H_\bdot(\h)$ is graded and that
$$
\chi\big(\Gr_n H_\bdot(\h)\big) = \chi\big(\Gr_n\Lambda^\bdot \h \big)
$$
in the representation ring $\cR(G)$ of $G$, where $\chi$ denotes Euler characteristic. For example, it implies that
$$
\Gr_3 H_\bdot(\h) = \chi\big(\Lambda^3 \h_1 \to \h_1\otimes \h_2 \to \h_3\big),
$$
where the first map in the complex is the ``Jacobi identity''
$$
x\wedge y \wedge z \mapsto x\otimes [y,z] + y\otimes[z,x] + z\otimes [x,y] \in \h_1\otimes \h_2
$$
and the second map is the bracket. So, if one knows $\Gr_3 H_\bdot(\h)$ and $\h_1$ and $\h_2$, then one can compute $\h_3$ in the representation ring of $G$. This is the method used in \cite[\S 8]{hain:torelli} to compute the graded quotients of the lower central series of $\p(S,x)$.

One can invert this formula to obtain a formula for $\h_\bdot$ in terms of $\Gr_\bdot H_\bdot(\h)$, a fact I learned from Looijenga. To this end, set
$$
\Phi(x) = \sum_{n\ge 0} \chi\big(\Gr_n H_\bdot(\h)\big)\, x^n \in \cR(G)[[x]].
$$

\begin{proposition}
Define $\Psi_n(\h) \in \cR(G)$ to be the coefficients of the power series
$$
\sum_{n\ge 0} \Psi_n(\h) x^n = -\frac{x\Phi'(x)}{\Phi(x)} \ \in \cR(G)[[x]].
$$
Then in $\cR(G)$, we have
\begin{equation}
\h_n = \frac{1}{n} \sum_{d|n} \mu(d) \psi^d \Psi_{n/d}(\h)
\end{equation}
where $\psi^d$ denotes the $d$th Adam's operation and where $\mu$ is the M\"obius function.
\end{proposition}

When every $G$-module is isomorphic to its dual (as it is when $G=\Sp(H)$), homology can be replaced by cohomology in this formula.

If the cohomology of $\h$ is pure in degrees $\le m$ in the sense that
$$
H_j(\h) = \Gr_j H_j(\h), \quad \text{ for all } j\le m,
$$
then
$$
\Phi(x) \equiv \sum_{n\ge 0} (-1)^n H_n(\h)\, x^n \in \cR(G)[[x]] \bmod (x^{m+1}).
$$
This simplifies $\Psi(x) \bmod (x^{m+1})$ and implies that $H^j(\h)$ is determined by $H_k(\h)$ ($k\le j$) for all $j\le m$.

\section{Index of Principal Notation}

This is an index of the principal notation used in the paper. In most instances, the page reference is the first place the notation occurs that is {\em not} in the introduction. Some items do not have a page reference as they are somewhat standard and are included to aid those unfamiliar with the notation.
\bigskip

\begin{tabbing}
loo0gnotation \= an explanation ---- quite looooong xxxxxxxxxxxxxxxxxxx
then
\= page reference \kill

{\bf Mapping class groups and moduli spaces (see especially Section~\ref{sec:topology})}\\

\\

$\Sbar$ \> a closed, oriented surface \> p.~\pageref{sec:topology} \\

$P$ \> finite subset of $\Sbar$ \> p.~\pageref{sec:topology} \\

$\vV$ \> set of non-zero tangent vectors anchored at $Q\subset \Sbar-P$ \> p.~\pageref{sec:topology} \\

$S$ \> the open surface $\Sbar-(P\cup Q)$ \\

$\pi_1(S,\vv)$ \> fundamental group of $S$ with a tangential base point $\vv\in \vV$  \> p.~\pageref{def:pi(S,v)} \\

$\G_{\Sbar,P+\vV}$ \> mapping class group of $\Sbar$ fixing $P$ and $\vV$ pointwise \> p.~\pageref{def:mcg} \\

$T_{\Sbar,P+\vV}$ \> Torelli subgroup of $\G_{\Sbar,P+\vV}$ \> p.~\pageref{def:torelli} \\

$\G_{g,n+\vr}$ \> $\G_{\Sbar,P+\vV}$ where $\Sbar$ is closed of genus $g$, $n=\# P$ and $r=\# \vV$ \> p.~\pageref{def:mcg} \\

$T_{g,n+\vr}$ \> Torelli subgroup of $\G_{g,n+\vr}$ \> p.~\pageref{def:torelli} \\

$T'S$ \> the bundle of non-zero tangent vectors of a surface $S$ \> p.~\pageref{eqn:ses} \\

$\M_{g,n+\vr}$ \> moduli space of genus $g$ curves with $n$ points and $r$ tangents \> p.~\pageref{def:mod_sp}\\

$\cA_g$ \> moduli space of principally polarized abelian varieties of dim $g$ \> p.~\pageref{def:Ag} \\

\\

{\bf Representation theory of $\Sp(H)$ (see especially sub-section~\ref{sec:rep_th})} \\

\\

$H_A$ \> $H_1(\Sbar;A)$ with its intersection form, $A=\Z,\Q,\R,\C$ \> p.~\pageref{def:H} \\

$\Sp(H)$ \> symplectic group associated to $H$ and its intersection form \> p.~\pageref{def:H}\\

$\sp(H)$ \> its Lie algebra \> \\ 

$\Sp_g(A)$ \> $2g\times 2g$ matrices over $A$ fixing the symplectic inner product \>  p.~\pageref{def:H}\\

$\theta$ \> $\sum_j a_j\wedge b_j \in \Lambda^2 H$ dual of intersection pairing \> p.~\pageref{def:theta} \\

$\Lambda^3_0 H$ \> irreducible representation corresponding to $[1^3]$, $g\ge 3$ \> p.~\pageref{sec:rep_th}\\

$V_\boxplus$ \> irreducible representation corresponding to $[2^2]$, $g\ge 2$  \> p.~\pageref{sec:rep_th}\\

$\Lambda^2_0 H$ \> irreducible representation corresponding to $[1^2]$, $g\ge 2$ \> p.~\pageref{def:theta_perp}\\

$V_\bmu$ \> irreducible $\Sp(H)$-module corresponding to partion $\bmu$\\

\\

{\bf Completions (see especially Section~\ref{sec:completions})} \\

\\

$\kk\pi^\wedge$ \> $I$-adic completion of $\kk\G$ where $I$ is augmention ideal \> p.~\pageref{def:comp_gp_alg} \\

$\pi^\un$ \> unipotent completion of group $\pi$ \> p.~\pageref{def:pi^un} \\

$\p$, $\p(S,x)$ \>  Lie algebras of $\pi^\un$ and $\pi_1(X,x)^\un$ \> p.~\pageref{def:p} \\


$\rho$ \> canonical homomorphism $\G_{\Sbar,P+\vV} \to \Sp(H)$ \> p.~\pageref{def:rho} \\

$\cG_{g,n+\vr}$ \> relative completion of $\G_{g,n+\vr}$ with respect to $\rho$ \> p.~\pageref{sec:rel_comp_mcg} \\

$\g_{g,n+\vr}$ \> its Lie algebra \> p.~\pageref{sec:rel_comp_mcg} \\

$\U_{g,n+\vr}$ \> its prounipotent radical \> p.~\pageref{sec:rel_comp_mcg} \\

$\u_{g,n+\vr}$ \> its Lie algebra \> p.~\pageref{sec:rel_comp_mcg} \\

$\t_{g,n+\vr}$ \> Lie algebra of the unipotent completion of $T_{g,n+\vr}$ \> p.~\pageref{def:Lie_torelli} \\

\\

{\bf Goldman--Turaev Lie bialgebra (see especially Section~\ref{sec:gt-lie-bialg})} \\

\\

$\lambda(S)$ \> the set of conjugacy classes of $\pi_1(S,x)$ \> p.~\pageref{sec:gt-lie-bialg} \\

$\kk\lambda(S)$ \> the free $\kk$-module generated by $\lambda(S)$ \> p.~\pageref{sec:gt-lie-bialg} \\

$I\kk\lambda(S)$ \> the free $\kk$-module generated by $\lambda(S)$ \> p.~\pageref{sec:gt-lie-bialg} \\

$|\kk\pi_1(S,x)|$ \> the cyclic quotient of $\kk\pi_1(S,x)$ \> p.~\pageref{def:cyclic_gp_alg} \\

$\gold$ \> the Goldman bracket \> p.~\pageref{def:goldman} \\

$\delta_\xi$ \> the Turaev cobracket of a surface with framing $\xi$ \> p.~\pageref{eqn:def_turaev}\\

$\deltabar$ \> the reduced cobracket \> p.~\pageref{def:red_cobracket} \\

$\pi(S;\vv_1,\vv_2)$ \> torsor of homotopy classes of paths in $S$ from $\vv_1$ to $\vv_2$ \> p.~\pageref{sec:kk-action} \\

$\kappa_\vv$ \> the Kawazumi--Kuno action \> p.~\pageref{def:kappa_v} \\

$\psi_n$ \> the $n$th power operator on $\kk\lambda(S)$ and completions \> p.~\pageref{sec:power_ops} \\

$\rot_\xi(\gamma)$ \> the rotation number of a loop $\gamma$ with respect to $\xi$ \> p.~\pageref{def:rot} \\

$\G_{g,\uu}^\xi$ \> the stabilizer of $\xi$ in $\G_{g,\uu}$ \> p.~\pageref{def:stab_xi} \\

$\kk\lambda(S)^\wedge$ \> the unipotent completion of $\kk\lambda(S)$ \> p.~\pageref{def:gt_comp}\\

$\kappahat_\vv$ \> the completed Kawazumi--Kuno action \> p.~\pageref{def:kappahat_v} \\

$|\Sym^n \p(S,\vv)|$ \> the cyclic quotient of $\Sym^n \p(S,\vv)|$ \> p.~\pageref{def:cyclic_sym} \\

$|A|$ \> the {\em cyclic quotient} of an associaive algebra $A$ \> p.~\pageref{def:cyclic_A} \\

$|a|$ \> image in $|A|$ of $a\in A$  \> p.~\pageref{def:cyclic_A} \\

$\g^\xi$ \>  the stabilizer of $\xi$ in $\g := \g_{g,n+\uu}$ \> p.~\pageref{def:g^xi} \\

$\gbar^\xi$ \>  the stabilizer of $\xi$ in $\gbar$ \> p.~\pageref{def:g^xi} \\

$\ghat^\xi$ \>  the stabilizer of $\xi$ in $\ghat$ \> p.~\pageref{def:g^xi} \\

\\

{\bf Filtrations} \\

\\

$K^\bdot V$ \> a decreasing filtration $\cdots \supset K^r V \supset K^{r+1}V \supset \cdots$ of $V$ \\

$\Gr^r_K V$ \> its $r$th graded quotient $K^r V/K^{r+1}V$ \\

$G_\bdot V$ \> an increasing filtration $\cdots \subset G_{r-1} V \subset G_r V \subset \cdots$ of $V$ \\

$\Gr_r^G V$ \> its $r$th graded quotient $G_r V/G_{r-1}V$ \\

$J^\bdot T_{S,\partial S}$ \>  the Johnson filtration of the Torelli group of $(S,\partial S)$ \> p.~\pageref{def:johnson_filt} \\

\\

Often a filtration $K^\bdot V$ is viewed as a functor $K^\bdot$ on a category with objects $V$:\\

\\

$L^\bdot$ \> the lower central series functor on groups and Lie algebras\\

$W_\bdot$ \> the weight filtration functor on mixed Hodge structures\\

$F^\bdot$ \> the Hodge filtration functor  on mixed Hodge structures\\

$M_\bdot$ \> the {\em relative} weight filtration of a nilpotent endomorphism\\

\\

{\bf Lie theory}\\

\\

$\L(V)$ \> the free Lie algebra generated by the vector space $V$ \> p.~\pageref{def:free_LA} \\

$T(V)$ \> the tensor algebra on $V$; also the enveloping algebra of $\L(V)$ \> p.~\pageref{def:free_LA} \\

$\L(\ee_\alpha)$ \> the free Lie algebra generated by the set $\{\ee_\alpha : \alpha \in A\}$ \> p.~\pageref{def:free_LA}  \\

$\kk\langle \ee_\alpha\rangle$ \> the free associative algebra generated by the set $\{\ee_\alpha : \alpha \in A\}$ \> p.~\pageref{def:free_LA} \\

$\L(V)^\wedge$ \> the completion of $\L(V)$ with respect to a $L^\bdot$ or $W_\bdot$ \\

$\Sym^k V$ \> the $k$th symmetric power of the vector space $V$\\


$\Der^\theta \L(H)$ \> the derivations that annihilate $\theta = \sum [a_j,b_j]$ \> p.~\pageref{def:dertheta} \\


$\SDer \p$ \> the {\em special} derivations of $\p(S,\vv)$  \>  p.~\pageref{sec:sder} \\


$V_\bmu$ \> irreducible $\Sp(H)$-module corresponding to the partition $\bmu$ \> p.~\pageref{sec:rep_th} \\

\\

{\bf Johnson homomorphisms} \\

\\

$\tau_k$ \> the $k$th higher Johnson homomorphim \> p.~\pageref{eqn:higher_johnson} \\

$\tau_{S,\partial S}$ \> the geometric Johnson homomorphism (preliminary) \> p.~\pageref{eqn:johnson_univ}\\

$\tautilde_{S,\partial S}$ \> the arithmetic Johnson homomorphism \> p.~\pageref{eqn:johnson_arith} \\


$d\varphi_\vv$ \>  the geometric Johnson homomorphism (and derivative of $\varphi_\vv$) \> p.~\pageref{eqn:geom_J} \\

$\varphi$ \> Kawazumi--Kuno lift of $d\varphi_\vv$ to $\Q\lambda(S)^\wedge$ \> p.~\pageref{thm:fund_diag} \\

$\gbar_{g,n+\vr}$ \> the image of the geometric Johnson homomorphism $d\varphi_\vv$ \> p.~\pageref{sec:arith_johnson} \\

$\gbar_{g,n+\vr}$ \> the image of $\u_{g,n+\vr}$ under $d\varphi_\vv$ \> p.~\pageref{sec:arith_johnson} \\


$\ghat_{g,n+\vr}^\phi$ \> the Lie subalgebra of $\Der\p$ generated by $\m^\phi$ and $\gbar$ \> p.~\pageref{def:ghatphi} \\

$\ghat_{g,n+\vr}$ \> same as $\ghat_{g,n+vr}^\phi$ as it does not depend on complex structure $\phi$ \> p.~\pageref{thm:arith_extn} \\

$\n$ \> the normalizer of the ``geometric derivations'' $\gbar$ in $\Der^\theta \p$ \> p.~\pageref{def:normalizer} \\

\\

{\bf Hodge theory and motives}\\

\\

$\Q(n)$ \> the Tate Hodge structure of weight $-2n$ \> p.~\pageref{ex:tate_twist} \\

$V(n)$ \> the twist $V\otimes \Q(n)$ of a mixed Hodge structure $V$ \> p.~\pageref{ex:tate_twist} \\

$\MHS$ \> the category of graded polarizable $\Q$-mixed Hodge structures \> p.~\pageref{sec:splittings} \\

$\MHS^\ss$ \> the full subcategory of semi-simple objects \> p.~\pageref{def:pi_1(mhs)} \\

$\pi_1(\MHS)$ \> the fundamental group of $\MHS$ \> p.~\pageref{def:pi_1(mhs)} \\

$\chi$ \> the canonical central cocharacter $\Gm \to \pi_1(\MHS^\ss)$ \> p.~\pageref{def:chi} \\

$\chitilde$ \> any lift of $\chi$ to a cocharacter $\Gm \to \pi_1(\MHS)$ (never central!) \> p.~\pageref{def:chi} \\

$\MT_\phi$ \> Mumford--Tate group of the MHS on $\g_{g,n+\vr}^\phi$ \> p.~\pageref{def:MT} \\

$\m^\phi$ \> its Lie algebra \>  p.~\pageref{def:MT} \\

$\MTM$ \> the category of mixed Tate motives, unramified over $\Z$ \> p.~\pageref{def:MTM} \\

$\mtm$ \> the Lie algebra of $\pi_1(\MTM)$ \> p.~\pageref{def:MTM} \\

$\k$ \> the Lie algebra of the prounipotent radical of $\pi_1(\MTM)$ \> p.~\pageref{def:k} \\

$\sigma_{2n+1}$ \> a generator of $\k$ of weight $-4n-2$, $n > 0$  \> p.~\pageref{def:k} \\

$\u^\MT_\phi$ \> pronilpotent radical of $\m^\MT_\phi$ \> p.~\pageref{def:u^MT} \\

\end{tabbing}


\begin{thebibliography}{99}

\bibitem{akita}
T.~Akita:
{\em Homological infiniteness of Torelli groups}, Topology 40 (2001), 213--221.

\bibitem{akkn:genus0}
A.~Alekseev, N.~Kawazumi, Y.~Kuno, F.~Naef:
{\em The Goldman--Turaev Lie bialgebra in genus zero and the Kashiwara--Vergne problem}, Adv.\ Math.\ 326 (2018), 1--53. \comment{arXiv:1703.05813}

\bibitem{akkn2}
A.~Alekseev, N.~Kawazumi, Y.~Kuno, F.~Naef:
{\em The Goldman--Turaev Lie bialgebra and the Kashiwara--Vergne problem in higher genera}, \comment{arXiv:1804.09566}

\bibitem{alekseev-torossian}
A.~Alekseev, C.~Torossian:
{\em The Kashiwara--Vergne conjecture and Drinfeld's associators}, Ann.\ of Math.\ (2) 175 (2012), 415--463. 

\bibitem{andreadakis}
S.~Andreadakis:
{\em On the automorphisms of free groups and free nilpotent groups}, Proc.\ London Math.\ Soc.\ 15 (1965), 239--268.

\bibitem{bachmuth}
S.~Bachmuth:
{\em Automorphisms of a class of metabelian groups}, Trans.\ Amer.\ Math.\ Soc.\ 127 (1967), 284--293.

\bibitem{bbd}
A.~Beilinson, J.~Bernstein, P.~Deline:
{\em Faisceaux pervers}, in Analysis and topology on singular spaces, I (Luminy, 1981), 5--171, Ast\'erisque, 100, Soc.\ Math.\ France, Paris, 1982.

\bibitem{bbm}
M.~Bestvina, K.-U.~Bux, D.~Margalit:
{\em The dimension of the Torelli group}, J.\ Amer.\ Math.\ Soc.\ 23 (2010), 61--105. 

\bibitem{bezrukavnikov}
R.~Bezrukavnikov:
{\em Koszul DG-algebras arising from configuration spaces}, Geom.\ Funct.\ Anal.\ 4 (1994), 119--135.

\bibitem{brieskorn}
E.~Brieskorn:
{\em Sur les groupes de tresses [d'apr\`es V.\ I Arnol'd]}, S\'eminaire Bourbaki, 24\`eme ann\'ee (1971/1972), Exp.\ No.\ 401, pp.~21--44. Lecture Notes in Math., Vol. 317, Springer, 1973.

\bibitem{brown}
F.~Brown:
{\em Mixed Tate motives over $\Z$}, Ann.\ of Math.\ (2) 175 (2012), 949--976.

\bibitem{brown:depth}
F.~Brown:
{\em Depth-graded motivic multiple zeta values}, \comment{arXiv:1301.3053}

\bibitem{brown:mmm}
F.~Brown:
{\em Multiple Modular Values and the relative completion of the fundamental group of $\M_{1,1}$}, \comment{arXiv:1407.5167}

\bibitem{chas}
M.~Chas:
{\em Combinatorial Lie bialgebras of curves on surfaces}, Topology 43 (2004), 543--568.

\bibitem{chen:malcev}
K.~T.~Chen:
{\em Extension of $C^\infty$ function algebra by integrals and Malcev completion of $\pi_1$},
Advances in Math.\ 23 (1977), 181--210.

\bibitem{cep}
T.~Church, M.~Ershov, A.~Putman:
{\em On finite generation of the Johnson filtrations}, \comment{arXiv:1711.04779}

\bibitem{church-farb}
T.~Church, B.~Farb:
{\em Representation theory and homological stability}, Adv.\ Math.\ 245 (2013), 250--314.

\bibitem{genus3}
F.~Cl\'ery, C.~Faber, G.~van der Geer:
{\em Concomitants of Ternary Quartics and Vector-valued Siegel and Teichmüller Modular Forms of Genus Three}, Selecta Math.\ (N.S.) 26 (2020), Paper No.~55, 39 pp. \comment{arXiv:1908.04248}

\bibitem{conant}
J.~Conant:
{\em The Johnson cokernel and the Enomoto-Satoh invariant}, Algebr.\ Geom.\ Topol.\ 15 (2015), 801--821.

\bibitem{conant-vogtmann}
J.~Conant, K.~Vogtmann:
{\em On a theorem of Kontsevich}, Algebr.\ Geom.\ Topol.\ 3 (2003), 1167--1224.

\bibitem{ceg}
W.~Crawley-Boevey, P.~Etingof, V.~Ginzburg:
{\em Noncommutative geometry and quiver algebras}, Adv.\ Math.\ 209 (2007), 274--336.

\bibitem{deligne:degen}
P.~Deligne:
{\em Th\'eor\`eme de Lefschetz et crit\`eres de d\'eg\'en\'erescence de suites spectrales},
Inst.\ Hautes \'Etudes Sci.\ Publ.\ Math.\ No.~35 (1968), 259--278.

\bibitem{deligne:poids}
P.~Deligne:
{\em Poids dans la cohomologie des vari\'et\'es alg\'ebriques}, Proceedings of the International Congress of Mathematicians (Vancouver, B.~C., 1974), Vol.~1, pp.~79--85. Canad.\ Math.\ Congress, Montreal, Que., 1975.

\bibitem{deligne:h2}
P.~Deligne:
{\em Th\'eorie de Hodge, II}, Inst.\ Hautes \'Etudes Sci.\ Publ.\ Math.\ No.~40 (1971), 5--57.

\bibitem{deligne:h3}
P.~Deligne:
{\em Th\'eorie de Hodge, III}, Inst.\ Hautes \'Etudes Sci.\ Publ.\ Math.\ No.~44 (1974), 5--77.

\bibitem{deligne:tannakian}
P.~Deligne:
{\em Cat\'egories tannakiennes}, The Grothendieck Festschrift, Vol.~II, 111--195, Progr.\ Math., 87, Birkh\"auser Boston, 1990.

\bibitem{deligne-goncharov}
P.~Deligne, A.~Goncharov:
{\em Groupes fondamentaux motiviques de Tate mixte}, Ann.\ Sci.\ \'Ecole Norm.\
Sup.\ (4) 38 (2005), 1--56.

\bibitem{dimca-papadima}
A.~Dimca, S.~Papadima:
{\em Arithmetic group symmetry and finiteness properties of Torelli groups}, Ann.\ of Math.\ (2) 177 (2013), 395--423.

\bibitem{drinfeld}
V.~Drinfeld:
{\em On quasitriangular quasi-Hopf algebras and on a group that is closely connected with $\Gal(\Qbar/\Q)$}, (Russian) Algebra i Analiz 2 (1990), 149--181; translation in Leningrad Math. J. 2 (1991), 829--860.

\bibitem{enomoto-satoh:derivations}
N.~Enomoto, T.~Satoh:
{\em On the derivation algebra of the free Lie algebra and trace maps}, Algebr.\ Geom.\ Topol.\ 11 (2011), 2861--2901. 

\bibitem{enomoto-satoh}
N.~Enomoto, T.~Satoh:
{\em New series in the Johnson cokernels of the mapping class groups of surfaces}, Algebr.\ Geom.\ Topol.\ 14 (2014), 627--669. 

\bibitem{eks}
N.~Enomoto, Y.~Kuno, T.~Satoh:
{\em A comparison of classes in the Johnson cokernels of the mapping class groups of surfaces}, Topology Appl.\ 271 (2020), 107052, 14 pp.

\bibitem{enriquez}
B.~Enriquez:
{\em Elliptic associators}, Selecta Math.\ (N.S.) 20 (2014), 491--584.

\bibitem{ershov-he}
M.~Ershov, S.~He:
{\em On finiteness properties of the Johnson filtrations} Duke Math.\ J.\ 167 (2018), 1713--1759.


\bibitem{farb-margalit}
B.~Farb, D.~Margalit:
{\em A primer on mapping class groups}, Princeton Mathematical Series 49, Princeton University Press, 2012.

\bibitem{garoufalidis-getzler}
S.~Garoufalidis, E.~Getzler:
{\em Graph complexes and the symplectic character of the Torelli group}, \comment{arXiv:1712.03606}

\bibitem{garoufalidis-levine:spheres}
S.~Garoufalidis, J.~Levine:
{\em Finite type 3-manifold invariants, the mapping class group and blinks},
J.\ Differential Geom.\ 47 (1997), 257--320.


\bibitem{goldman}
W.~Goldman:
{\em Invariant functions on Lie groups and Hamiltonian flows of surface group representations}, Invent.\ Math.\ 85 (1986), 263--302.

\bibitem{gross-schoen}
B.~Gross, C.~Schoen:
{\em The modified diagonal cycle on the triple product of a pointed curve}, Ann.\ Inst.\ Fourier (Grenoble) 45 (1995), 649--679.


\bibitem{hain:dht}
R.~Hain:
{\em The de Rham homotopy theory of complex algebraic varieties, I}, K-Theory 1 (1987), 271--324.

\bibitem{hain:malcev}
R.~Hain:
{\it Hodge-de~Rham theory of relative Malcev completion}, Ann.\ Sci.\ \'Ecole
Norm.\ Sup., t.~31 (1998), 47--92.

\bibitem{hain:completions}
R.~Hain:
{\em Completions of mapping class groups and the cycle $C-C^-$}, in {\it Mapping Class Groups and Moduli Spaces of Riemann Surfaces}, C.-F.~B\"odigheimer and R.~Hain, editors, Contemp.\ Math.\ 150 (1993), 75--105.

\bibitem{hain:msri}
R.~Hain:
{\em Torelli groups and geometry of moduli spaces of curves}, Current topics in complex algebraic geometry (Berkeley, CA, 1992/93), 97--143, Math.\ Sci.\ Res.\ Inst.\ Publ., 28, Cambridge Univ.\ Press, 1995. 

\bibitem{hain:torelli}
R.~Hain:
{\em Infinitesimal presentations of the Torelli groups}, J.\ Amer.\ Math.\ Soc.\ 10 (1997), 597--651.

\bibitem{hain:a3}
{\em The rational cohomology ring of the moduli space of abelian 3-folds}, Math.\ Res.\ Lett.\ 9 (2002), 473--491.

\bibitem{hain:problems}
R.~Hain:
{\em Finiteness and Torelli spaces}, in Problems on mapping class groups and related topics, 57--70, Proc.\ Sympos.\ Pure Math., 74, Amer.\ Math.\ Soc., Providence, 2006.

\bibitem{hain:morita}
R.~Hain:
{\em Relative weight filtrations on completions of mapping class groups}, in ``Groups of diffeomorphisms'', 309--368, Adv.\ Stud.\ Pure Math., 52, Math.\ Soc.\ Japan, Tokyo, 2008. \comment{arXiv:0802.0814}

\bibitem{hain:rat_pts}
R.~Hain:
{\em Rational points of universal curves}, J.\ Amer.\ Math.\ Soc.\ 24 (2011), 709--769.


\bibitem{hain:kahler}
R.~Hain:
{\em Genus 3 mapping class groups are not K\"ahler}, J.\ Topol.\ 8 (2015), 213--246.

\bibitem{hain:modular}
R.~Hain:
{\em The Hodge–de Rham theory of modular groups}, Recent advances in Hodge theory, 422--514, London Math.\ Soc.\ Lecture Note Ser., 427, Cambridge Univ. Press, 2016.


\bibitem{hain:db}
R.~Hain:
{\em Deligne-Beilinson cohomology of affine groups}, Hodge theory and $L^2$-analysis, 377--418, Adv.\ Lect.\ Math.\ (ALM), 39, 2017.

\bibitem{hain:goldman}
R.~Hain:
{\em Hodge theory of the Goldman bracket}, Geom.\ Topol.\ 24 (2020), 1841--1906.\comment{arXiv:1710.06053}

\bibitem{hain:turaev}
R.~Hain:
{\em Hodge Theory of the Turaev Cobracket and the Kashiwara--Vergne Problem}, J.\ Eur.\ Math.\ Soc., to appear. \comment{arxiv:1807.09209}

\bibitem{hain:ihara}
R.~Hain:
{\em Unipotent Path Torsors of Ihara Curves}, in preparation.

\bibitem{hain-looijenga}
R.~Hain, E.~Looijenga:
{\em Mapping class groups and moduli spaces of curves}, Algebraic geometry—Santa Cruz 1995, 97--142, Proc.\ Sympos.\ Pure Math., 62, Part 2, Amer.\ Math.\ Soc., 1997

\bibitem{hain-matsumoto:density}
R.~Hain, M.~Matsumoto:
{\em Galois actions on fundamental groups of curves and the cycle $C-C^-$}, J.\ Inst.\ Math.\ Jussieu 4 (2005), 363--403.

\bibitem{hain-matsumoto:mem}
R.~Hain, M.~Matsumoto:
{\em Universal mixed elliptic motives}, J.\ Inst.\ Math.\ Jussieu 19 (2020), no. 3, 66--766. \comment{arXiv:1512.03975}

\bibitem{hain-reed:morita}
R.~Hain, D.~Reed:
{\em Geometric proofs of some results of Morita}, J.\ Algebraic Geom.\ 10 (2001), 199--217.

\bibitem{hodge:icm}
W.~V.~D.~Hodge:
{\em The topological invariants of algebraic varieties}, Proceedings of the International Congress of Mathematicians, Cambridge, Mass., 1950, vol.~1, pp.~182--192. Amer.\ Math.\ Soc., 1952.

\bibitem{ihara-nakamura}
Y.~Ihara, H.~Nakamura:
{\em On deformation of maximally degenerate stable marked curves and Oda's
problem}, J.\ Reine Angew.\ Math.\ 487 (1997), 125--151.

\bibitem{jantzen}
J.~C.~Jantzen, 
{\em Representations of Algebraic Groups}, Pure and Applied Mathematics Vol.~131, Academic Press, 1987.

\bibitem{johnson:homom}
D.~Johnson:
{\em An abelian quotient of the mapping class group $\mathcal I_g$}, Math.\ Ann.\ 249 (1980), 225--242.

\bibitem{johnson:survey}
D.~Johnson:
{\em A survey of the Torelli group}, Low-dimensional topology (San Francisco, Calif., 1981), 165--179, Contemp.\ Math., 20, Amer.\ Math.\ Soc., 1983

\bibitem{johnson:fg}
D.~Johnson:
{\em The structure of the Torelli group, I: A finite set of generators for
$\mathcal I$}, Ann.\ of Math.\ (2) 118 (1983), 423--442.

\bibitem{johnson:twists}
D.~Johnson:
{\em The structure of the Torelli group. II, A characterization of the group generated by twists on bounding curves}, Topology 24 (1985), 113--26.

\bibitem{johnson:h1}
D.~Johnson:
{\em The structure of the Torelli group. III: The abelianization of $\mathcal I$}, Topology 24 (1985), 127--144.

\bibitem{kabanov:stability}
A.~Kabanov:
{\em Stability of Schur functors}, J.\ Algebra 195 (1997), 233--240. 

\bibitem{kabanov:purity}
A.~Kabanov:
{\em The second cohomology with symplectic coefficients of the moduli space of smooth projective curves}, Compositio Math.\ 110 (1998), 163--186. 


\bibitem{kawazumi:framings}
N.~Kawazumi:
{\em The mapping class group orbits in the framings of compact surfaces}, Q.\ J.\ Math. 69 (2018), 1287--1302, \comment{arXiv:1703.02258}


\bibitem{kk:log}
N.~Kawazumi, Y.~Kuno:
{\em The logarithms of Dehn twists}, Quantum Topol.\ 5 (2014), 347--423. \comment{arXiv:1008.5017}

\bibitem{kk:groupoid}
N.~Kawazumi, Y.~Kuno:
{\em Groupoid-theoretical methods in the mapping class groups of surfaces}, \comment{arXiv:1109.6479}

\bibitem{kk:intersections}
N.~Kawazumi, Y.~Kuno:
{\em Intersections of curves on surfaces and their applications to mapping class groups}, Ann.\ Inst.\ Fourier (Grenoble) 65 (2015), 2711--2762. \comment{arXiv:1112.3841}

\bibitem{kawazumi-morita}
N.~Kawazumi, S.~Morita:
{\em The primary approximation to the cohomology of the moduli space of curves and cocycles for the stable characteristic classes}, Math.\ Res.\ Lett.\ 3 (1996), 629--641.

\bibitem{kohno}
T.~Kohno:
{\em On the holonomy Lie algebra and the nilpotent completion of the fundamental group of the complement of hypersurfaces}, Nagoya Math.\ J.~92 (1983), 21--37.

\bibitem{kontsevich}
M.~Kontsevich:
{\em Formal (non)commutative symplectic geometry}, The Gel'fand Mathematical Seminars, 1990–1992, 173--187, Birkh\"auser Boston, 1993.

\bibitem{krw}
A.~Kupers, O.~Randal-Williams:
{\em On the cohomology of Torelli groups}, Forum Math.\ Pi 8 (2020), e7, 83 pp. \comment{arXiv:1901.01862}


\bibitem{levine}
J.~Levine:
{\em Addendum and correction to: "Homology cylinders: an enlargement of the mapping class group''}, Algebr.\ Geom.\ Topol.\ 2 (2002), 1197--1204.

\bibitem{looijenga}
E.~Looijenga:
{\em Stable cohomology of the mapping class group with symplectic coefficients and of the universal Abel--Jacobi map}, J.\ Algebraic Geom.\ 5 (1996), 135--150. 

\bibitem{madsen-weiss}
I.~Madsen, M.~Weiss
{\em The stable moduli space of Riemann surfaces: Mumford's conjecture}, Ann.\ of Math.\ (2) 165 (2007), 843--941.

\bibitem{margulis}
G.~Margulis:
{\em Discrete subgroups of semisimple Lie groups}, Ergebnisse der Mathematik und ihrer Grenzgebiete 17, Springer-Verlag, 1991.

\bibitem{mess}
G.~Mess:
{\em The Torelli groups for genus 2 and 3 surfaces}, Topology 31 (1992), 775--790.


\bibitem{morgan}
J.~Morgan:
{\em The algebraic topology of smooth algebraic varieties}, Inst.\ Hautes
\'Etudes Sci.\ Publ.\ Math.\ No. 48 (1978), 137--204; Correction: Inst.\ Hautes \'Etudes Sci.\ Publ.\ Math.\ No.\ 64 (1986), 185.

\bibitem{morita:homom}
S.~Morita:
{\em Abelian quotients of subgroups of the mapping class group of surfaces}, Duke Math.\ J.\ 70 (1993), 699--726.

\bibitem{morita:chern}
S.~Morita:
{\it A linear representation of the mapping class group of orientable surfaces and characteristic classes of surface bundles}, in the Proceedings of the 37th Taniguchi Symposium on Topology and Teichm\"uller Spaces, July 1995, S.~Kojima et al editors, World Scientific (1996), 159--186.


\bibitem{morita:survey}
S.~Morita:
{\em Structure of the mapping class groups of surfaces: a survey and a prospect}, Proceedings of the Kirbyfest (Berkeley, CA, 1998), 349--406, Geom.\ Topol.\ Monogr., 2, 1999.

\bibitem{morita:taut}
S.~Morita:
{\em Generators for the tautological algebra of the moduli space of curves}, Topology 42 (2003), 787--819.

\bibitem{morita:problems}
S.~Morita:
{\em Cohomological structure of the mapping class group and beyond}, in Problems on mapping class groups and related topics, 329--354, Proc.\ Sympos.\ Pure Math., 74, Amer.\ Math.\ Soc., 2006.

\bibitem{morita:galois}
S.~Morita:
{\em Characteristic classes of moduli spaces---Riemann surface, graph, homology cobordism}, (translation of a paper in Sugaku 69-2 (2017), 113--136, Mathematical Society of Japan). \comment{to be updated}

\bibitem{morita-sakasai-suzuki}
S.~Morita, T.~Sakasai, M.~Suzuki:
{\em Structure of symplectic invariant Lie subalgebras of symplectic derivation Lie algebras}, Adv.\ Math.\ 282 (2015), 291--334.

\bibitem{nakamura-schneps}
H.~Nakamura, L.~Schneps:
{\em On a subgroup of the Grothendieck-Teichm\"uller group acting on the tower of profinite Teichm\"uller modular groups}, Invent.\ Math.\ 141 (2000), 503--560. 

\bibitem{oda}
T.~Oda:
{\em The universal monodromy representations on the pro-nilpotent fundamental groups of algebraic curves}, Mathematische Arbeitstagung (Neue Serie) June 1993, Max-Planck-Institute preprint MPI/93-57.\footnote{Available at {\sf https://www.mpim-bonn.mpg.de/preblob/4902}}

\bibitem{papakyr}
C.~Papakyriakopoulos:
{\em Planar regular coverings of orientable closed surfaces}, Knots, groups, and 3-manifolds (Papers dedicated to the memory of R.~H.~Fox), 261--292. Ann.\ of Math.\ Studies, No. 84, Princeton Univ. Press, 1975.

\bibitem{patzt}
P.~Patzt:
{\em Representation stability for filtrations of Torelli groups}, Math.\ Ann.\ 372 (2018), 257--298.

\bibitem{petersen}
D.~Petersen:
{\em Cohomology of local systems on the moduli of principally polarized abelian surfaces}, Pacific J.\ Math.\ 275 (2015), 39--61. \comment{arXiv:1310.2508}

\bibitem{petersen:letter}
D.~Petersen:
{\em Personal communication}, January, 2018 and August, 2019.

\bibitem{pty}
D.~Petersen, M.~Tavakol, Q.~Yin:
{\em Tautological classes with twisted coefficients}, \comment{arXiv:1705.08875}

\bibitem{pollack}
A.~Pollack:
{\em Relations between derivations arising from modular forms}, undergraduate
thesis, Duke University, 2009.\footnote{Available at: {\sf
http://dukespace.lib.duke.edu/dspace/handle/10161/1281}}

\bibitem{quillen}
D.~Quillen:
{\em Rational homotopy theory}, Ann.\ of Math.\ 90 (1969), 205--295.

\bibitem{saito}
M.~Saito:
{\em Modules de Hodge polarisables}, Publ.\ Res.\ Inst.\ Math.\ Sci.\ 24 (1988), 849--995 (1989).

\bibitem{satoh:survey}
T.~Satoh:
{\em On the Johnson homomorphisms of the mapping class groups of surfaces}, in Handbook of group actions Vol.~I, 373--407, Adv.\ Lect.\ Math.\ (ALM), 31, 2015.

\bibitem{schedler}
T.~Schedler:
{\em A Hopf algebra quantizing a necklace Lie algebra canonically associated to a quiver}, Int.\ Math.\ Res.\ Not.\ 2005, 725--760.


\bibitem{stallings}
J.~Stallings:
{\em Homology and central series of groups}, J.\ Algebra 2 (1965), 170--181.

\bibitem{vmhs1}
J.~Steenbrink, S.~Zucker:
{\em Variation of mixed Hodge structure. I}, Invent.\ Math.\ 80 (1985), 489--542. 

\bibitem{sullivan}
D.~Sullivan:
{\em On the intersection ring of compact three manifolds}, Topology 14 (1975), 275--277. 

\bibitem{takao}
N.~Takao:
{\em Braid monodromies on proper curves and pro-$\ell$ Galois representations}, J.\ Inst.\ Math.\ Jussieu 11 (2012), 161--181.

\bibitem{borel_stab}
B.~Tshishiku:
{\em Borel's stable range for the cohomology of arithmetic groups}, preprint \comment{arXiv:1904.04902}

\bibitem{tsunogai}
H.~Tsunogai:
{\em On some derivations of Lie algebras related to Galois representations}, Publ.\ Res.\ Inst.\ Math.\ Sci.\ 31 (1995), 113--134.

\bibitem{turaev:78}
V.~Turaev:
{\em Intersections of loops in two-dimensional manifolds}, Mat.\ Sb.\ 106 (148) (1978), 566--588.

\bibitem{turaev:skein}
V.~Turaev:
{\em Skein quantization of Poisson algebras of loops on surfaces}, Ann.\ Sci.\ \'Ecole Norm.\ Sup.\ (4) 24 (1991), 635--704.

\bibitem{watanabe}
T.~Watanabe:
{\em On the completion of the mapping class group of genus two},  J.\ Algebra 501 (2018), 303--327. \comment{arXiv:1609.05552}

\end{thebibliography}
\end{document}